\documentclass[12pt]{article}

\usepackage{times}
\usepackage{amsthm}
\usepackage{amsfonts}

\usepackage{bbm}

\usepackage{amsmath,amssymb}
\usepackage[pdftex]{color, graphicx} 

\usepackage{mathrsfs}
\usepackage{plain}
\usepackage{makeidx}

\headsep \headheight

\numberwithin{equation}{section}

\usepackage[english]{babel} 
\usepackage[T1]{fontenc}
\usepackage{indentfirst}
\usepackage{fancyhdr}
\usepackage{afterpage}

\usepackage{mathtools}

\usepackage{wasysym}

\usepackage{scalerel}

\usepackage{amssymb,fge}

\usepackage{calrsfs}
\DeclareMathAlphabet{\pazocal}{OMS}{zplm}{m}{n}

\usepackage[titletoc,toc,title]{appendix}

\usepackage{calc}

\usepackage{accents}

\makeatletter
\newcommand*{\rom}[1]{\expandafter\@slowromancap\romannumeral #1@}
\makeatother

 \usepackage{authblk}

\usepackage{enumitem,nicefrac}
\mathtoolsset{showonlyrefs,showmanualtags}

\theoremstyle{theorem}
\newtheorem{theorem}{Theorem}[section]
\newtheorem{example}{Example}[section]
\newtheorem{thm*}{Theorem*}
\newtheorem{lemma}[theorem]{Lemma}
\theoremstyle{definition}
\newtheorem{definition}[theorem]{Definition}
\theoremstyle{assumption}
\newtheorem{assumption}[theorem]{Assumption}
\theoremstyle{proposition}
\newtheorem{prop}[theorem]{Proposition}
\theoremstyle{corollary}
\newtheorem{cor}[theorem]{Corollary}
\theoremstyle{remark}
\newtheorem{remark}[theorem]{Remark}

\DeclareMathOperator{\R}{\mathbb{R}}

\DeclareMathOperator{\T}{\mathbb{T}}
\DeclareMathOperator{\N}{\mathbb{N}}
\DeclareMathOperator{\Z}{\mathbb{Z}}
\DeclareMathOperator{\TT}{\mathbb{T}}
\DeclareMathOperator{\E}{\mathbb{E}}
\DeclareMathAlphabet{\mathup}{OT1}{\familydefault}{m}{n}
\newcommand{\dx}[1]{\mathop{}\!\mathup{d} #1}

\DeclarePairedDelimiter{\abs}{\lvert}{\rvert}
\DeclarePairedDelimiter{\norm}{\lVert}{\rVert}
\DeclarePairedDelimiter{\bra}{(}{)}
\DeclarePairedDelimiter{\pra}{[}{]}
\DeclarePairedDelimiter{\set}{\{}{\}}

\newcommand{\cF}{\ensuremath{\mathcal F}}
\newcommand{\cP}{\ensuremath{\mathcal P}}
\newcommand{\cC}{\ensuremath{\mathcal C}}
\newcommand{\cG}{\ensuremath{\mathcal G}}
\newcommand{\cK}{\ensuremath{\mathcal K}}
\newcommand{\ve}{\varepsilon}
\newcommand{\eps}{\varepsilon}
\DeclareMathOperator{\C}{C}

\providecommand{\ud}[1]{\, \mathrm{d} #1}
\providecommand{\dx}{\ud{x}}
\providecommand{\dy}{\ud{y}}

\providecommand{\deta}{\ud{\eta}}

\providecommand{\ds}{\ud{s}}
\providecommand{\dt}{\ud{t}}

\providecommand{\dd}{\ud}

\newcommand{\Leb}{\ensuremath{{L}}}

\DeclareMathOperator{\Ent}{Ent}

\DeclareMathOperator{\Lip}{Lip}

\DeclareMathOperator{\sign}{sgn}
\DeclareMathOperator*{\esup}{ess\,sup}

\DeclareMathOperator{\diam}{diam}

\newenvironment{tenumerate}[1][]
  {\enumerate[label=\textup(\alph*\textup),ref=(\alph*),#1]}
  {\endenumerate}

\usepackage[svgnames]{xcolor}

\usepackage[colorinlistoftodos,prependcaption,color=yellow,textsize=tiny]{todonotes}

\usepackage{etoolbox}
\apptocmd{\thebibliography}{\setlength{\itemsep}{1pt}}{}{}
  
\usepackage[debug=true]{hyperref}
\hypersetup{
  linktoc=all,colorlinks=true,linkcolor=Green,citecolor=Green,urlcolor=DarkBlue} 
  \title{Ergodicity and random dynamical systems for conservative SPDEs}
 \author[1]{Benjamin Fehrman \thanks{benjamin.fehrman@maths.ox.ac.uk}}
 \author[2]{Benjamin Gess  \thanks{benjamin.gess@math.uni-bielefeld.de}}
 \author[3]{Rishabh S. Gvalani  \thanks{gvalani@mis.mpg.de}}
 \affil[1]{\textit{University of Oxford}}
 \affil[2]{\textit{Universit{\"a}t Bielefeld and Max Planck Institute for Mathematics in the Sciences, Leipzig}}
 \affil[3]{\textit{Max Planck Institute for Mathematics in the Sciences, Leipzig}}
 
\date{}

\begin{document}
\maketitle
\begin{abstract}
The dynamics of the solutions to a class of conservative SPDEs are analysed from two perspectives: Firstly, a probabilistic construction of a corresponding random dynamical system is given for the first time. Secondly, the existence and uniqueness of invariant measures, as well as mixing for the associated Markov process is shown. \end{abstract}
\tableofcontents

\allowdisplaybreaks
\section{Introduction}

In this work, we consider the dynamical behaviour of solutions to conservative SPDEs of the general type
\begin{equation}\label{eq:spde-intro}
\partial_{t}\rho=\Delta\Phi(\rho)-\nabla\cdot\nu(\rho)-\nabla\cdot(B(\rho))-f(\rho)-\sqrt{\eps}\nabla\cdot(\sigma(\rho)\xi),
\end{equation}
where $\nu$ is local function of $\rho$, $B$ a nonlocal function, $\xi$ is spatially correlated noise and under assumptions on $\Phi,\nu,B,f,\sigma$ specified below. We introduce a new probabilistic construction of an associated random dynamical system, and we prove the ergodicity and mixing properties of the associated Markovian semigroup. 

Conservative SPDEs of this type appear, for example, as fluctuating continuum models in non-equilibrium statistical mechanics, see, for example \cite{Ot2005,Gi.Le.Pr1999,FG21,FehGesDir2020,Fe.Ge2022}. 

The theory of random dynamical systems (RDS) aim at merging techniques from stochastic analysis with those from dynamical systems theory, as an approach to the systematic study of dynamical properties of solutions to stochastic differential equations. Once it has been set up correctly, a random dynamical system provides us with several useful tools, for example, the multiplicative ergodic theorem, the concept of Lyapunov exponents, and invariant manifolds. This motivates the use of RDS techniques in the dynamical analysis of solutions to conservative SPDEs such as~\eqref{eq:spde-intro}. Before methods from RDS may be used in order to analyse the dynamics of SPDEs, it has to be shown that appropriate modifications can be selected among their solutions satisfying, for example, the stochastic flow property. 

In contrast to the deterministic case, the stochastic flow property is not an immediate consequence of the uniqueness of solutions. While in finite dimensions satisfactory answers to this problem have been found, based on the Kolmogorov continuity theorem (see \cite{K90,Arn98,AS95}), the infinite-dimensional case is intrinsically more difficult and counterexamples are known to exist \cite{M84}. The vast majority of established results rely either on a transformation to a (random) PDE that can be solved path-by-path \cite{CF94,IL01,G13-2,G13}, a method restricted to affine-linear noise, or on analytic methods like rough path theory and its generalizations \cite{Ku.Ne.So2020,GT10,He.Ne2019}, see in particular \cite{FehGes2019} for conservative SPDEs, which, however, require restrictive assumptions on the regularity of the coefficients, since roughly speaking they rely on higher order Taylor expansions. 

Only few results on the existence of stochastic flows based on probabilistic methods exist. The most general result so far was developed in \cite{F96}. One key obstacle that one encounters on the way to proving the existence of stochastic flows is the selection of a modification of the set of solutions that is continuous with respect to the initial condition. This obstacle was successfully overcome in \cite{F96} for semilinear SPDEs of the form\footnote{The method introduced in the present work  strictly contains the SPDEs considered in \cite{F96}, since the additional $0$-order multiplicative noise in \eqref{eq:intro_flanoli} compared to \eqref{eq:spde-intro} can be handled by a standard transformation method, see e.g.\ \cite{G13-2}.} 
\begin{equation}\label{eq:intro_flanoli}
  \partial_t \rho= D^2 : (A(x)\rho) - \nabla\cdot(a(x)\rho) - f(\rho) - \nabla\cdot(\rho \xi) - \rho (\tilde\xi+b(x)).
\end{equation}
The present work extends \cite{F96} in two regards: First, the method developed in \cite{F96} is restricted to linear noise and linear higher order operators as in \eqref{eq:intro_flanoli}. Second,  in \cite{F96} it is shown that for fixed times $s,t$ a modification of the solutions to \eqref{eq:intro_flanoli} that is pathwise continuous with respect to the initial condition can be chosen. This only implies an imperfect form of the stochastic flow property, leaving open the proof of the (perfect) semi-flow property, as well as of the (perfect) cocycle property, which are key properties of an RDS. Both of these open problems are solved in the present work. We present the first purely probabilistic construction of an RDS for a nonlinear SPDE of the type \eqref{eq:spde-intro}.

The key to this construction is a pathwise contraction estimate for \eqref{eq:spde-intro}, which was first established in \cite{FG21} and is generalized in this work to possibly unstable dynamics with nontrivial and possible nonlocal drift $B$, and nonlinearity $f$. More precisely, we prove that there is a finite $C=C(T)>0$, so that each two solutions $\rho^1,\rho^2$ to \eqref{eq:spde-intro} satisfy a.s. 
\begin{equation}\label{eq:intro-pathwise-contraction}
\sup_{t \in [0,T]}\|\rho(t,\rho_{0,1},\omega)-\rho(t,\rho_{0,2},\omega)\|_{L^1(\T^d)}\le C \|\rho_{0,1}-\rho_{0,2}\|_{L^1(\T^d)} \, .
\end{equation}
Roughly speaking, it is this pathwise continuity in the initial condition that allows us to select good modifications of the solutions $\rho$ to \eqref{eq:spde-intro}, which are then shown to form a stochastic flow. The so-called perfection of this stochastic flow into an RDS strongly relies on the continuity in the initial time of solutions \cite{AS95,Ka.Sc1997}. In the case of SPDEs, it is challenging to establish this continuity. In the present work, we establish new regularity estimates for SPDEs of the type \eqref{eq:spde-intro} which, in particular, are shown to imply initial time continuity. This leads to the first main result of this work of which we state an informal version below.
\begin{theorem}[see Theorem~\ref{thm:rds}]
There exists an $\eps_0 \in (0,1)$, such that for all $\eps < \eps_0$ and under appropriate assumptions on $\Phi,\sigma,\nu,B,f$ (see Assumptions~\ref{assumption_1}, \ref{assumption_2},  \ref{assumption_3}, \ref{assdrift}, and ~\ref{assume_grad}), the solutions $\rho(t,\rho_0,\omega)$ of~\eqref{eq:spde-intro} generate a random dynamical system.
\end{theorem}
This result opens the way for an analysis of the qualitative behaviour of solutions to \eqref{eq:spde-intro} by means of RDS techniques. Notably, this includes dynamically challenging cases for which the deterministic dynamics
\begin{equation}
\partial_{t}\rho=\Delta\Phi(\rho)-\nabla\cdot\nu(\rho)-\nabla\cdot(B(\rho))-f(\rho)\label{eq:spde-intro-1}
\end{equation}
are unstable. We provide a particular example of a system where the results of our theorem apply.
\begin{example}[Parabolic sine-Gordon equation with conservative noise]
Consider the following SPDE
\[
\partial_t \rho = \Delta \rho - \kappa f(\rho)- \sqrt{\eps} \nabla\cdot (\sigma(\rho)\dd\xi) \, ,
\]
where $f(x)=\sin{x}$, $\ve$ is small, and $\sigma(\cdot)$ is chosen to be sufficiently regular.  With the particular choice of $f$ we have made and with $\kappa$ sufficiently large, the deterministic part of the equation is bistable, i.e. it has multiple steady states.
\end{example}

The second main result of this work addresses the existence and the uniqueness of invariant measures, and mixing for conservative SPDEs. While the ergodicity for SPDEs with semilinear noise, that is, noise that depends only on the values of the solution, not on its derivatives, has attracted much interest in the literature, only a few results are known in the case of conservative noise, see, for example, \cite{EKMS00,ESR12,Ge.So2017,GS16-2,Ga.Ge2019} and the references therein. Notably, these known results are restricted to the case of spatially homogeneous noise. In this case, the spatially constant state is invariant, and, therefore, the unique invariant measure is simply given by the corresponding Dirac measure. Here, we study the dynamically richer case of spatially inhomogeneous noise. 

We introduce a general framework, based on the pathwise stability property \eqref{eq:intro-pathwise-contraction}, which implies the uniqueness of invariant measures for the associated  Markovian semigroups. We prove our ergodicity result in a more abstract setting (see Theorem~\ref{thm:abstractergodicity}) for a Markov process taking values in a metric space and satisfying certain properties, i.e. the strong Markov property, an abstract dissipation estimate, and the existence of a distinguished object in the support of its path space measure which is asymptotically contractive (see Assumption~\ref{ass:support}). The last property is symbolic of the contractivity of the deterministic dynamics of~\eqref{eq:spde-intro} in the model case we wish to treat. Indeed, the theory is developed with an eye on the case of the following conservative SPDE
\begin{equation}
\partial_{t}\rho=\Delta\Phi(\rho) - \nabla \cdot \nu(\rho)-\sqrt{\eps}\nabla\cdot(\sigma(\rho)\dd\xi),\label{eq:spde-intro-2}
\end{equation}
where $\Phi$ and $\nu$ are such that the non-perturbed case ($\eps=0$) is asymptotically globally stable.

\begin{theorem}[see Theorems~\ref{thm:abstractergodicity} and~\ref{thm:semiabstract}]
Assume that $f, B \equiv 0$. Then, under appropriate assumptions  on $\Phi,\sigma,\nu$ (see Assumptions~\ref{assumption_1}, \ref{assumption_2},  \ref{assumption_3}, and \ref{assdrift}), the associated Markovian semigroup has a unique invariant measure and is strongly mixing. 
\end{theorem}

We apply this theorem to a particular SPDE of the type \eqref{eq:spde-intro-2}. 
\begin{example}[Regularized Dean--Kawasaki equation]
Consider the following SPDE
\[
\partial_t \rho = \Delta \rho - \sqrt{\eps} \nabla\cdot (\sigma(\rho)\dd\xi) \, ,
\]
where $\ve$ is small and $\sigma(\cdot)$ is chosen to be a sufficiently nice approximation of $(\cdot)^{\frac{1}{2}}$. The above equation is a regularized version of the so-called Dean--Kawasaki equation which describes the macroscopic fluctuations of $N$ i.i.d Brownian motions on $\T^d$, see \cite{Dea96}. 
\end{example}

\subsection{Overview of the literature}

The generation of random dynamical systems for (finite-dimensional) SDE has been thoroughly understood in a series of works. The existence of stochastic flows is answered in detail in \cite{K90}, fundamentally relying on the Kolmogorov continuity theorem. This left open the problem of the existence of perfect cocycles. This aspect was discussed in detail in the works \cite{AS95,Ka.Sc1997,Arn98}. In particular, in \cite{Ka.Sc1997} a general perfection result for crude cocycles has been shown, relying on the initial time continuity of solutions. Again, in the case of (finite-dimensional) SDE, this continuity could be obtained via Kolmogorov's continuity theorem. The infinite-dimensional case is intrinsically more difficult and counter-examples are known \cite{M84}. Despite much effort \cite{MZZ08,F95,BF92,FS90,F96}, no probabilistic construction of RDS for nonlinear SPDEs is known, previous to the present work. In fact, the existence of random dynamical systems is one of the major obstacles in developing a dynamical systems approach to SPDEs. The vast majority of existence results for SPDEs relies on transformations into random PDE, which can then be solved path by path, for example, \cite{CF94,MZZ08,IL01,G13-2,G13} but there are many more. 

Under restrictive assumptions on the coefficients, the existence of stochastic flows for linear SPDEs has been shown in \cite{F95,BF92,FS90}. Only few probabilistic methods to construct stochastic flows for nonlinear SPDEs are known. The most general result in this direction was shown in \cite{F96} considering SPDEs of the type \eqref{eq:intro_flanoli}. The proof in \cite{F96} was based on the derivation of a weighted $L^2$ norm, for which the noise operator in \eqref{eq:intro_flanoli} becomes contractive. As a result, the stochastic integral in the corresponding It\^o formula vanishes, leading to a pathwise stability estimate. This method is strictly limited to linear higher order operators, since nonlinear operators like all of the terms in \eqref{eq:spde-intro}, are not contractive in weighted $L^2$ norms. Therefore, the method developed in the present work can be applied to a far reaching extension SPDEs.

A rough-path based well-posedness approach to conservative SPDEs was developed in recent years in the works \cite{LS98,LPS13,LPS14,GS14,GS16-2,Ho.Ka.Ri.St2018,FG21,Ho.Ka.Ri.St2020}, leading to the work \cite{FehGes2019}, which proved the existence of RDS for conservative parabolic-hyperbolic SPDEs with spatially inhomogeneous noise. Since this approach relies on rough paths techniques, it requires restrictive assumptions on the coefficients. 

More precisely, in \cite{FehGes2019} it was required that $\sigma \in C^5(\R)$ and that $\xi$ takes values in $C^5(\TT^d)$. 

An alternative approach to the existence of random dynamical systems based on a rough path, and, thus, path-by-path well-posedness theory for SPDEs has been initiated in \cite{GT10} and partially extended in \cite{He.Ne2019}. As usual in rough path based approaches, these results rely on high order regularity of the coefficients. 

The ergodicity for SPDEs with semilinear noise has attracted much interest in the literature and a complete account would go much beyond the scope of this work. We refer to the monographs \cite{DPZ14,Ku.Sh2012} and the references therein. Much less is known for the case of conservative SPDEs. The analysis of the long-time behaviour of conservative stochastic conservation laws has been initiated in \cite{Ge.So2017}, where a quantitative rate of convergence to the spatially constant state has been proved. In the subsequent work \cite{GS16-2}, this result was generalized to the case of parabolic-hyperbolic SPDEs with spatially homogeneous conservative noise. These works rely on Fourier analytic methods, and, thus, are restricted to the case of spatially homogeneous noise. 

The long-time behaviour for particular cases of stochastic Hamilton-Jacobi equation, has been analysed in \cite{Ga.Ge2019}. Notably, in 1+1 dimension, the derivative of these equations satisfies a conservative SPDE. These results have been significantly extended in the recent work \cite{GGLS22}, where also cases of spatially inhomogeneous noise have been treated. In the case of spatially homogeneous noise, in \cite{GGLS22} it has been shown that noise increases the speed of convergence to the infinite time limit point in the stochastic mean curvature flow.

The pathwise stability property \eqref{eq:intro-pathwise-contraction} can be interpreted as a strong, pathwise form of the $e$-property introduced in \cite{KPS10}, also known under the names of $e$-chain, $d$-equicontinuity \cite{Ku.Sc2018}. The abstract framework developed in \cite{KPS10} has been subsequently applied in several instances of SPDEs, e.g.\ stochastic mean curvature flows \cite{ESR12}, singular-degenerate SPDEs \cite{GT14,GT15}, and models from self-organized criticality \cite{Ne2019a}.  An alternative approach via generalized couplings has been developed in \cite{Ku.Sc2018}, providing conditions for strong mixing. A quantification of the rate of convergence has subsequently been given in \cite{Bu.Ku.Sc2020}. For a discussion of the relation of the $e$-property to the concept of asymptotic strong Feller property introduced in \cite{Ha.Ma2006a} we refer to \cite{Ja2013a}. Assuming the stronger pathwise stability property \eqref{eq:intro-pathwise-contraction}, the results in the present work show  strong mixing for the Markovian semigroup. This is stronger than the weak$^*$-mean ergodicity which has been shown in \cite{KPS10}.

The ergodicity for the conservative, stochastic Burgers equation has been proven in \cite{EKMS00}, with several subsequent extensions in \cite{Ba2016a,Bo2016b,Bo.Ku2021}. Notably, while these works focus on the case of additive noise, they consider SPDEs posed on the full space, for which the methods developed in the present work do not apply. Extensions of \cite{EKMS00}, to multi-dimensional (non-conservative) stochastic scalar conservation laws with general flux and semilinear, multiplicative noise have been given in \cite{DV15}, with further extension to parabolic-hyperbolic SPDEs in \cite{Ge.Ho2018}.

\section{Pathwise stability for SPDEs with conservative noise} ~\label{sec:stability}

In this section, we prove the uniqueness and pathwise contraction of stochastic kinetic solutions to the It\^o equation
\begin{equation}\label{it_2}\partial_t\rho = \Delta\Phi(\rho)-\nabla\cdot\nu(\rho)-\nabla\cdot (\lambda(\rho) B(\rho))-f(\rho)-\nabla\cdot(\sigma(\rho)\dd\xi).\end{equation}
The definition of a solution in Definition~\ref{def_sol_new} below is based on the equation's kinetic form and  \cite[Definitions~3.4, 6.2]{FG21}. Further motivation for the argument leading to the uniqueness of solutions can be found in the introduction to \cite[Section~4]{FG21}.  We introduce the assumptions on the noise and coefficients in Assumptions \ref{assumption_1}, \ref{ass1}, \ref{ass2}, and \ref{assdrift} below.  We prove the pathwise $L^1$-contraction of solutions in Theorem~\ref{thm_unique}.

 \begin{assumption}\label{assumption_1} \label{assnoise} Let $(\Omega,\mathcal{F},\mathbb{P})$ be a probability space, let $(\mathcal{F}_t)_{t\in[0,\infty)}$ be a filtration on $\Omega$, let $\{e_k\}_{k\in\N}$ be a smooth orthonormal $L^2(\TT^d)$-basis, and let $(\lambda_k)_{k\in\N}$ be real numbers.  Assume that the noise $\xi$ is defined by
 \[\xi = \sum_{k=1}^\infty \lambda_k e_k B^k_t,\]
 for independent $\mathcal{F}_t$-adapted Brownian motions $(B^k)_{k\in\N}$, and assume that $\xi$ is probabilistically stationary in space.  Precisely, the quadratic variation of the noise and that of its spatial gradient
 \[F_1 = \sum_{k=1}^\infty \lambda_k^2 e_k^2\;\;\textrm{and}\;\; F_3 = \sum_{k=1}^\infty \lambda_k^2\abs{\nabla e_k}^2,\]
 are finite and constant on $\TT^d$. \end{assumption}
 \begin{remark}  Assumption~\ref{assumption_1} is satisfied by the standard approximations of space-time white noise on $\TT^d$:
 \begin{equation}\label{trig_noise}\xi=\sum_{k\in\Z^{d}}\lambda_{k}(\sqrt{2}\sin(k\cdot x)B^k_t+\sqrt{2}\cos(k\cdot x)W^k_t),\end{equation}
for independent, $\mathcal{F}_t$-adapted Brownian motions $(B^k,W^k)_{k\in\Z^d}$, and for every smooth spatial convolution of space-time white noise.  In the specific case of \eqref{trig_noise}, we have that $F_{1}=\sum_{k\in\Z^{d}}\lambda_{k}^{2}$ and $F_{3} = \sum_{k\in\Z^{d}}\abs{k}^{2}\lambda_{k}^{2}$, and we have that Assumption~\ref{assumption_1} is satisfied whenever the latter sum is finite.  \end{remark}

 \begin{assumption}[{\cite[Assumption 4.1]{FG21}}] Assume that $\Phi,\sigma\in C([0,\infty)$ and $\nu\in C([0,\infty);\R^{d})$ satisfy the following assumptions: 
\begin{enumerate}
\item We have that $\Phi,\sigma\in C_{{\rm loc}}^{1,1}((0,\infty))$ and $\nu\in C_{{\rm loc}}^{1}((0,\infty);\R^{d})$. 
\item We have that $\Phi(0)=0$ with $\Phi'>0$ on $(0,\infty)$. 
\item There exists $c\in(0,\infty)$ such that 
\[
\limsup_{\xi\to0^{+}}\frac{\sigma^{2}(\xi)}{\xi}\leq c.
\]
\item There exists $c\in[1,\infty)$ such that 
\[
\sup_{\xi'\in[0,\xi]}\sigma^{2}(\xi')\leq c(1+\xi+\sigma^{2}(\xi))\textrm{ for }\xi\in[0,\infty)\,.
\]
\item There exists $c\in[1,\infty)$ such that 
\[
\sup_{\xi'\in[0,\xi]}|\nu|(\xi')\leq c(1+\xi+|\nu|(\xi))\textrm{ for }\xi\in[0,\infty)\,.
\]
\end{enumerate}
\label{ass1}
\label{assumption_2}
 \end{assumption}
 
 \begin{remark}  The following assumption is used only to guarantee the existence of stochastic kinetic solutions in Section~\ref{apptospde} below.  The assumption plays no role in the proof of the $L^1$-contraction in Theorem~\ref{thm_unique}.  \end{remark}
 
 \begin{assumption}[{\cite[Assumption 5.2]{FG21}}]\label{assumption_3} Let $\Phi,\sigma\in C([0,\infty)\cap C_{{\rm loc}}^{1}((0,\infty))$, $\nu\in C([0,\infty);\R^{d})\cap C_{{\rm loc}}^{1}((0,\infty);\R^{d})$, and $p\in[2,\infty)$ satisfy the following assumptions: 
\begin{enumerate}
\item There exists a $m\in[1,\infty)$ and $c\in(0,\infty)$ such that 
\[
\Phi(\xi)\leq c(1+\xi^{m})\textrm{ for }\xi\in[0,\infty)\,.
\]
\item Let $\Theta_{\Phi,p}\in C([0,\infty))\cap C_{{\rm loc}}^{1}((0,\infty))$ be the unique function satisfying $\Theta_{\Phi,p}(0)=0$ and $\Theta_{\Phi,p}'(\xi)=\xi^{\frac{p-2}{2}}[\Phi'(\xi)]^{\frac{1}{2}}$. Then, there exists $c\in(0,\infty)$ such that
\[
|\nu|(\xi)+\Phi'(\xi)\leq c(1+\xi+\Theta_{\Phi,p}^{2}(\xi))\textrm{ for }\xi\in(0,\infty)\,.
\]
\item Either there exists $c\in(0,\infty)$ and $\theta\in[0,1/2]$ such that 
\[
\xi^{-\frac{p-2}{2}}\Phi'(\xi)^{-\frac{1}{2}}\leq c\xi^{\theta}\textrm{ for }\xi\in(0,\infty)\,,
\]
or there exists $c\in(0,\infty)$ and $q\in[1,\infty)$ such that 
\[
|\xi-\xi'|^{q}\leq c|\Theta_{\Phi,p}(\xi)-\Theta_{\Phi,p}(\xi')|^{2}\textrm{ for }\xi,\xi'\in[0,\infty)\,.
\]
\item For $\Theta_{\Phi,2}$ and $\Theta_{\Phi,p}$ as above, there exists a $c\in(0,\infty)$ such that 
\[
\sigma^{2}(\xi)\leq c(1+\xi+\Theta_{\Phi,2}^{2}(\xi))\textrm{ and }\xi^{p-2}\sigma^{2}(\xi)\leq c(1+\xi+\Theta_{\Phi,p}^{2}(\xi))\textrm{ for }\xi\in[0,\infty)\,.
\]
\item For every $\delta\in(0,1)$, there exists $c_{\delta}\in(0,\infty)$ such that 
\[
\frac{\sigma'(\xi)^{4}}{\Phi'(\xi)}+[\sigma\sigma'(\xi)]^{2}+\Phi'(\xi)\leq c_{\delta}(1+\xi+\Theta_{\Phi,p}^{2}(\xi))\;\;\textrm{for every \ensuremath{\xi\in(\delta,\infty)}.}
\]
\end{enumerate}
\label{ass2} \end{assumption}

 \begin{assumption}\label{assume_B} Let $\lambda\in \C([0,\infty))\cap\C^1((0,\infty))$ have a bounded first derivative and assume that $B\colon L^{1}(\TT^{d})\rightarrow W^{1,\infty}(\TT^{d};\R^d)$ is Lipschitz continuous: there exists $\norm{B}_{{\Lip}}\in(0,\infty)$ such that 
\[
\norm{B(\rho^{1})-B(\rho^{2})}_{W^{1,\infty}(\TT^{d};\R^d)}\leq \norm{B}_{{\Lip}}\norm{\rho^{1}-\rho^{2}}_{L^{1}(\TT^{d})}\;\;\textrm{for every}\;\;\rho^{1},\rho^{2}\in L^{1}(\TT^{d}).
\]
\label{assdrift}
We further assume that $f\in \C([0,\infty))$, that $f(0)=0$, and that $f$ satisfies a one-sided Lipschitz estimate and a linear growth estimate:  there is a constant $\|f\|_{\Lip}<\infty$ such that, for every $\xi,\xi'\in[0,\infty)$
\[-(f(\xi)-f(\xi'))(\xi-\xi')\leq\|f\|_{\Lip}|\xi-\xi'|^2\;\;\textrm{and}\;\;|f(\xi)|\leq \|f\|_{\Lip}(1+|\xi|).\]
 \end{assumption}

\begin{definition}\label{def_sol_new}  Let $\rho_0\in L^1(\TT^d)$ be nonnegative and $\cF_0$-measurable.  A \emph{stochastic kinetic solution} of \eqref{it_2} is a continuous $L^1(\TT^d)$-valued, $\mathcal{F}_t$-predictable process $\rho\in L^1([0,T]\times\Omega;L^1(\TT^d))$ that satisfies the following properties.
\begin{enumerate}
\item \emph{Preservation of mass}:  almost surely for every $t\in[0,T]$,
\begin{equation}\label{it_3}\norm{\rho(\cdot,t)}_{L^1(\TT^d)}=\norm{\rho_0}_{L^1(\TT^d)}-\int_0^t \int_{\TT^d}f(\rho(\cdot,s))ds.\end{equation}
\item \emph{Integrability of the flux}:  we have that
\[\sigma(\rho)\in L^2(\Omega;L^2(\TT^d\times[0,T]))\;\;\textrm{and}\;\;\nu(\rho)\in L^1(\Omega;L^1(\TT^d\times[0,T];\R^d)).\]
\item \emph{Regularity of the reaction term}:  we have that a.s. 
\[\lambda(\rho)\in L^1([0,T];W^{1,1}(\TT^d)).\]
\item \emph{Local regularity}:  for every $K\in\N$,
\begin{equation}\label{it_4} [(\rho\wedge K)\vee\nicefrac{1}{K}]\in L^2(\Omega;L^2([0,T];H^1(\TT^d))).\end{equation}
\end{enumerate}
Furthermore, there exists a kinetic measure $q$ that satisfies the following three properties.
\begin{enumerate}
\setcounter{enumi}{3}
\item \emph{Regularity}: almost surely as nonnegative measures,
\begin{equation}\label{it_5}\delta_0(\xi-\rho)\Phi'(\xi)\abs{\nabla\rho}^2\leq q\;\;\textrm{on}\;\;\TT^d\times(0,\infty)\times[0,T].\end{equation}
\item \emph{Vanishing at infinity}:  we have that
\begin{equation}\label{it_6}\lim_{M\rightarrow\infty}\E\left[q(\TT^d\times[M,M+1]\times[0,T])\right]=0.\end{equation}
\item \emph{The equation}: for every $\psi\in C^\infty_c(\TT^d\times(0,\infty))$, almost surely for every $t\in[0,T]$,
\begin{align}\label{it_7}
& \int_{\R}\int_{\TT^d}\chi(x,\xi,t)\psi(x,\xi) = \int_{\R}\int_{\TT^d}\overline{\chi}(\rho_0)\psi(x,\xi)-\int_0^t\int_{\TT^d}\Phi'(\rho)\nabla\rho\cdot(\nabla\psi)(x,\rho)
\\ \nonumber & \quad -\int_0^t\int_{\R}\int_{\TT^d}\partial_\xi\psi(x,\xi)\dd q+\frac{F_1}{2}\int_0^t\int_{\TT^d}(\sigma'(\rho))^2\abs{\nabla\rho}^2(\partial_\xi\psi(x,\rho))
\\ \nonumber & \quad +\frac{F_3}{2}\int_0^t\int_{\TT^d}\sigma^2(\rho)(\partial_\xi\psi)(x,\rho)-\int_0^t\int_{\TT^d}\psi(x,\rho)\nabla\cdot\nu(\rho)\dt
\\ \nonumber & \quad -\int_0^t\int_{\TT^d}\psi(x,\rho)(\nabla\cdot(\lambda(\rho) B(\rho))+f(\rho))-\int_0^t\int_{\TT^d}\psi(x,\rho)\nabla\cdot\left(\sigma(\rho)\dd\xi\right).
\end{align}
\end{enumerate}
 \end{definition}

\begin{remark}
We remark here that the notion of solution introduced above coincides with the notion of rough paths solution introduced in~\cite{FehGes2019} as long as the assumptions stated above and in~\cite[Section 2.1]{FehGes2019} are both satisfied. We refer the reader to ~\cite[Remark 2.6]{DarGes2020} where the equivalence between different notions of solutions to such conservative SPDEs is briefly discussed.
  \label{rem:twosolutions}
\end{remark} \begin{definition}\label{def_s}  For every $\ve,\delta\in(0,1)$ let $\kappa^\ve_d\colon\TT^d\rightarrow[0,\infty)$ and $\kappa^\delta_1\colon\R\rightarrow[0,\infty)$ be standard convolution kernels of scales $\ve$ and $\delta$ on $\TT^d$ and $\R$ respectively, and let $\kappa^{\ve,\delta}$ be defined by
\[\kappa^{\ve,\delta}(x,y,\xi,\eta)=\kappa^\ve_d(x-y)\kappa^\delta_1(\xi-\eta)\;\;\textrm{for every}\;\;(x,y,\xi,\eta)\in(\TT^d)^2\times\R^2.\]
For every $\beta\in(0,1)$ let $\varphi_\beta\colon\R\rightarrow[0,1]$ be the unique nondecreasing piecewise linear function that satisfies
\begin{equation}\label{3_7}\varphi_\beta(\xi)=1\;\;\textrm{if}\;\;\xi\geq \beta,\;\;\varphi_\beta(\xi)=0\;\; \textrm{if}\;\;\xi\leq \nicefrac{\beta}{2},\;\;\textrm{and}\;\;\varphi'_\beta=\frac{2}{\beta}\mathbf{1}_{\{\nicefrac{\beta}{2}<\xi<\beta\}},\end{equation}
and for every $M\in\N$ let $\zeta_M\colon\R\rightarrow [0,1]$ be the unique nonincreasing piecewise linear function that satisfies
\[\zeta_M(\xi)=0\;\; \textrm{if}\;\;\xi\geq M+1,\;\;\zeta_M(\xi)= 1\;\;\textrm{if}\;\;\xi\leq M,\;\;\textrm{and}\;\;\zeta'_M=-\mathbf{1}_{\{M<\xi<M+1\}}.\]
\end{definition}

\begin{theorem}\label{thm_unique}  Let $\xi$, $\Phi$, $\sigma$, $\nu$, $\lambda$, $f$, and $B$ satisfy Assumptions~\ref{assumption_1}, ~\ref{assumption_2}, and \ref{assume_B}, let $\rho^1_0,\rho^2_1\in L^1(\Omega;L^1(\TT^d))$ be $\cF_0$-measurable and let $\rho^1,\rho^2$ be stochastic kinetic solutions of \eqref{it_2} in the sense of Definition~\ref{def_sol_new} with initial data $\rho_0^1,\rho_0^2$.  Then, almost surely,
\begin{equation}\begin{split}\label{eq:pathwise_stability}
& \norm{\rho^1(\cdot,t)-\rho^2(\cdot,t)}_{L^1(\TT^d)}
\\ & \leq \norm{\rho_{0,1}-\rho_{0,2}}_{L^1(\TT^d)}\exp\left(t(\norm{B}_{{\Lip}}\min_{i\in\{1,2\}}\norm{\lambda(\rho^i)}_{L^1([0,t];W^{1,1}(\TT^d))}+\|f\|_{\Lip})\right),
\end{split}\end{equation}
for $\norm{B}_{{\Lip}}$ the Lipschitz constant of the map $B\colon L^1(\TT^d)\rightarrow W^{1,\infty}(\TT^d;\R^d)$.
\end{theorem}

\begin{proof}   Let $\rho^1$ and $\rho^2$ be solutions of \eqref{it_2} with initial data $\rho^1_0$ and $\rho^2_0$ respectively, and let $\chi^1$ and $\chi^2$ be the kinetic functions of $\rho^1$ and $\rho^2$.   For every $\ve,\delta\in(0,1)$ and $i\in\{1,2\}$ let $\chi^{\ve,\delta}_{t,i}(y,\eta)=(\chi^i(\cdot,\cdot,t)*\kappa^{\ve,\delta})(y,\eta)$ for the convolution kernel $\kappa^{\ve,\delta}$ defined in Definition~\ref{def_s}.  It follows from Definition~\ref{def_sol_new} and the Kolmogorov continuity criterion (see, for example, Revuz and Yor \cite[Chapter~1, Theorem~2.1]{RevYor1999}) that for every $\ve,\delta\in(0,1)$ there exists a subset of full probability such that, for every $i\in\{1,2\}$, $(y,\eta)\in\TT^d\times(\nicefrac{\delta}{2},\infty)$, and $t\in[0,T]$,
\begin{align}\label{it_20}
& \left.\chi^{\ve,\delta}_{s,i}(y,\eta)\right|_{s=0}^t=\nabla_y\cdot \left(\int_0^t\int_{\TT^d}\Phi'(\rho^i)\nabla\rho^i \kappa^{\ve,\delta}(x,y,\rho^i,\eta)\right)
\\ \nonumber & \quad + \partial_\eta\left(\int_0^t\int_{\R}\int_{\TT^d}\kappa^{\ve,\delta}(x,y,\xi,\eta)\dd q^i\right)
\\ \nonumber & \quad -\int_0^t\int_{\TT^d}\kappa^{\ve,\delta}(x,y,\rho^i,\eta)(\nabla\cdot(\lambda(\rho^i)B(\rho^i,x))+f(\rho^i))
\\ \nonumber & \quad -\partial_\eta\left(\frac{1}{2}\int_0^t\int_{\TT^d}\left(F_1(\sigma'(\rho^i)^2\abs{\nabla\rho^i}^2+F_3(x)\sigma^2(\rho^i) \right)\kappa^{\ve,\delta}(x,y,\rho^i,\eta) \right)
\\ \nonumber & \quad -\int_0^t\int_{\TT^d}\kappa^{\ve,\delta}(x,y,\rho^i,\eta)\nabla\cdot\nu(\rho) - \int_0^t\int_{\TT^d}\kappa^{\ve,\delta}(x,y,\rho^i,\eta)\nabla\cdot\left(\sigma(\rho^i)\cdot\dd\xi\right).
\end{align}
For the cutoff functions defined in Definition~\ref{def_s}, it follows almost surely from \eqref{3_7} and \eqref{it_20} that, for every $\ve,\beta\in(0,1)$, $M\in\N$, and $\delta\in(0,\nicefrac{\beta}{4})$, for every $t\in[0,T]$ and $i\in\{1,2\}$,
\begin{equation}\label{it_38}
\left.\int_{\R}\int_{\TT^d}\chi^{\ve,\delta}_{s,i}(y,\eta)\varphi_\beta(\eta)\zeta_M(\eta)\dy\deta\right|^t_{s=0} = I^{i,\textrm{cut}}_t+I^{i,\textrm{mart}}_t+I^{i,\textrm{cons}}_t+I^{i,\textrm{int}}_t,
\end{equation}
for the cutoff term defined by
\begin{align*}
& I^{i,\textrm{cut}}_t =-\int_0^t\int_{\R^2}\int_{(\TT^d)^2}\kappa^{\ve,\delta}(x,y,\xi,\eta)\partial_\eta(\varphi_\beta(\eta)\zeta_M(\eta))\dd q^i(x,\xi,s)
 \\  & +\frac{1}{2}\int_0^t\int_{\R}\int_{(\TT^d)^2}F_1(\sigma'(\rho^i(x,s))^2\kappa^{\ve,\delta}(x,y,\rho^i(x,s),\eta)\partial_\eta(\varphi_\beta(\eta)\zeta_M(\eta))
 \\  & +\frac{1}{2}\int_0^t\int_{\R}\int_{(\TT^d)^2}F_3(x)\sigma^2(\rho^i(x,s))\kappa^{\ve,\delta}(x,y,\rho^i(x,s),\eta)\partial_\eta(\varphi_\beta(\eta)\zeta_M(\eta)),
 \end{align*}
for the martingale term defined by
\[I^{i,\textrm{mart}}_t = -\int_0^t\int_{\R}\int_{(\TT^d)^2}\kappa^{\ve,\delta}(x,y,\rho^i(x,s),\eta)\nabla\cdot(\sigma(\rho^i(x,s))\dd\xi)\varphi_\beta(\eta)\zeta_M(\eta),\]
for the conservative term defined by
\[I^{i,\textrm{cons}}_t=-\int_0^t\int_{\R}\int_{(\TT^d)^2}\kappa^{\ve,\delta}(x,y,\rho^i,\eta)\nabla\cdot\nu(\rho)\varphi_\beta(\eta)\zeta_M(\eta),\]
and for the interaction term defined by
\begin{equation}\label{itt_0}I^{i,\textrm{int}}_t = -\int_0^t\int_{\R}\int_{(\TT^d)^2}\kappa^{\ve,\delta}(x,y,\rho^i,\eta)(\nabla\cdot(\lambda(\rho^i)B(\rho^i,x))+f(\rho^i))\varphi_\beta(\eta)\zeta_M(\eta),\end{equation}
where we emphasize that the terms $I^{i,\textrm{cut}}_t$, $I^{i,\textrm{mart}}_t$, $I^{i,\textrm{cons}}_t$, and $I^{i,\textrm{int}}_t$ depend on $\ve,\delta,\beta\in(0,1)$ and $M\in\N$.

We will now handle the mixed term.  In the following, we will write $(x,\xi)\in\TT^d\times\R$ for the arguments of $\chi^1$ and all related quantities, and  $(x',\xi')\in\TT^d\times\R$ for the arguments of $\chi^2$ and all related quantities.  We define the convolution kernels
\[\overline{k}^{\ve,\delta}_{s,1}(x,y,\eta)=\kappa^{\ve,\delta}(x,y,\rho^1(x,s),\eta)\]
and
\[\overline{k}^{\ve,\delta}_{s,2}(x',y,\eta)=\kappa^{\ve,\delta}(x',y,\rho^2(x',s),\eta).\]
It follows from \eqref{3_7}, \eqref{it_20}, and the stochastic product rule that we almost surely have, for every $\ve,\beta\in(0,1)$, $M\in\N$, and $\delta\in(0,\nicefrac{\beta}{4})$, for every $t\in[0,T]$,
\begin{align}\label{it_09}
& \left.\int_{\R}\int_{\TT^d}\chi^{\ve,\delta}_{s,1}(y,\eta)\chi^{\ve,\delta}_{s,2}(y,\eta)\varphi_\beta(\eta)\zeta_M(\eta)\dy\deta\right|_{s=0}^t
\\ \nonumber & = \int_0^t\int_{\R}\int_{\TT^d}\left(\chi^{\ve,\delta}_{s,2}(y,\eta)\dd \chi^{\ve,\delta}_{s,1}(y,\eta)+\chi^{\ve,\delta}_{s,1}(y,\eta)\dd \chi^{\ve,\delta}_{s,2}(y,\eta)\right)\varphi_\beta(\eta)\zeta_M(\eta)\dy\deta
\\ \nonumber & \quad +\int_0^t\int_{\R}\int_{\TT^d} \dd\langle \chi^{\ve,\delta}_2,\chi^{\ve,\delta}_1\rangle_s(y,\eta)\varphi_\beta(\eta)\zeta_M(\eta)\dy\deta.
\end{align}
It follows from the integration by parts formula of \cite[Lemma~2.2]{FG21}, \eqref{it_20}, the definition of $\varphi_\beta$, $\delta\in(0,\nicefrac{\beta}{4})$, and the distributional equalities involving the kinetic function that
\begin{align}\label{it_0009}
& \int_0^t\int_{\R}\int_{\TT^d}\chi^{\ve,\delta}_{s,2}(y,\eta)\dd \chi^{\ve,\delta}_{s,1}(y,\eta)\varphi_\beta(\eta)\zeta_M(\eta)\dy\deta\ds
\\ & = I^{2,1,\textrm{err}}_t+I^{2,1,\textrm{meas}}_t+I^{2,1,\textrm{cut}}_t+I^{2,1,\textrm{mart}}_t+I^{2,1,\textrm{cons}}_t+I^{2,1,\textrm{int}}_t
\end{align}
where, after adding the second term of \eqref{it_000009} below and subtracting it in \eqref{it_0000009} below, the error term is
\begin{align}\label{it_000009}
I^{2,1,\textrm{err}}_t & = -\int_0^t\int_{\R}\int_{(\TT^d)^3}\Phi'(\rho^1)\nabla\rho^1\cdot\nabla\rho^2\overline{\kappa}^{\ve,\delta}_{s,1}\overline{\kappa}^{\ve,\delta}_{s,2}\varphi_\beta(\eta)\zeta_M(\eta)
\\ \nonumber & \quad +\int_0^t\int_{\R}\int_{(\TT^d)^3}[\Phi'(\rho^1)]^\frac{1}{2}[\Phi'(\rho^2)]^\frac{1}{2}\nabla\rho^1\cdot\nabla\rho^2\overline{\kappa}^{\ve,\delta}_{s,1}\overline{\kappa}^{\ve,\delta}_{s,2}\varphi_\beta(\eta)\zeta_M(\eta)
\\ \nonumber & \quad - \frac{1}{2}\int_0^t\int_{\R}\int_{(\TT^d)^3}\left(F_1(\sigma'(\rho^1))^2\right)\overline{\kappa}^{\ve,\delta}_{s,1}\overline{\kappa}^{\ve,\delta}_{s,2}\varphi_\beta(\eta)\zeta_M(\eta)
\\ \nonumber & \quad - \frac{1}{2}\int_0^t\int_{\R}\int_{(\TT^d)^3}\left(F_3(x)\sigma^2(\rho^1))\right)\overline{\kappa}^{\ve,\delta}_{s,1}\overline{\kappa}^{\ve,\delta}_{s,2}\varphi_\beta(\eta)\zeta_M(\eta),
\end{align}
the measure term is
\begin{align}\label{it_0000009}
& I^{2,1,\textrm{meas}}_t  = \int_0^t\int_{\R^2}\int_{(\TT^d)^3}\kappa^{\ve,\delta}(x,y,\xi,\eta)\overline{\kappa}^{\ve,\delta}_{s,2}\varphi_\beta(\eta)\zeta_M(\eta)\dd q^1(x,\xi,s)
\\ \nonumber & \quad - \int_0^t\int_{\R}\int_{(\TT^d)^3}[\Phi'(\rho^1)]^\frac{1}{2}[\Phi'(\rho^2)]^\frac{1}{2}\nabla\rho^1\cdot\nabla\rho^2\overline{\kappa}^{\ve,\delta}_{s,1}\overline{\kappa}^{\ve,\delta}_{s,2}\varphi_\beta(\eta)\zeta_M(\eta),
\end{align}
the cutoff term is
\begin{align*}
& I^{2,1,\textrm{cut}}_t  = -\int_0^t\int_{\R^2}\int_{(\TT^d)^2}\kappa^{\ve,\delta}(x,y,\xi,\eta)\chi^{\ve,\delta}_{s,2}(y,\eta)\partial_\eta(\varphi_\beta(\eta)\zeta_M(\eta))\dd q^1(x,\xi,s)
\\  & \quad +\frac{1}{2}\int_0^t\int_{\R}\int_{(\TT^d)^2}\left(F_1(\sigma'(\rho^1))^2+F_3(x)\sigma^2(\rho^1)\right)\overline{\kappa}^{\ve,\delta}_{s,1}\chi^{\ve,\delta}_{s,2}\partial_\eta(\varphi_\beta(\eta)\zeta_M(\eta)),
\end{align*}
the martingale term is
\[I^{2,1,\textrm{mart}}_t=-\int_0^t\int_{\R}\int_{(\TT^d)^2}\overline{\kappa}^{\ve,\delta}_{s,1}\chi^{\ve,\delta}_{s,2}\varphi_\beta(\eta)\zeta_M(\eta)\nabla\cdot(\sigma(\rho^1)\dd\xi(x)),\]
the conservative term is
\[I^{2,1,\textrm{cons}}_t = -\int_0^t\int_{\R}\int_{(\TT^d)^2}\overline{\kappa}^{\ve,\delta}_{s,1}\chi^{\ve,\delta}_{s,2}\varphi_\beta(\eta)\zeta_M(\eta)\nabla\cdot\nu(\rho^1),\]
and the interaction term is
\begin{equation}\label{itt_1}I^{2,1,\textrm{int}}_t = -\int_0^t\int_{\R}\int_{(\TT^d)^2}\overline{\kappa}^{\ve,\delta}_{s,1}\chi^{\ve,\delta}_{s,2}\varphi_\beta(\eta)\zeta_M(\eta)(\nabla\cdot(\lambda(\rho^1)B(\rho^1))+f(\rho^1)).\end{equation}
The analogous formula holds for the second term on the right hand side of \eqref{it_09}, with the decomposition $I^{1,2,\textrm{err}}_t$ and similarly for the remaining five parts.  For the final term of \eqref{it_09}, it follows from \eqref{it_20} and the definition of $\xi$ that
\begin{align}\label{3_009}
& \int_0^t\int_{\R}\int_{\TT^d}\dd\langle\chi^{\ve,\delta}_1,\chi^{\ve,\delta}_{s,2}\rangle_s(y,\eta)\varphi_\beta(\eta)\zeta_M(\eta)
\\ \nonumber & = \sum_{k=1}^\infty\int_0^t\int_{\R}\int_{(\TT^d)^3}\lambda_k^2e_k(x)e_k(x')\sigma'(\rho^1)\sigma'(\rho^2)\nabla\rho^1\cdot\nabla\rho^2\overline{\kappa}^{\ve,\delta}_{s,1}\overline{\kappa}^{\ve,\delta}_{s,2}\varphi_\beta(\eta)\zeta_M(\eta)
\\ \nonumber & \quad + \sum_{k=1}^\infty\int_0^t\int_{\R}\int_{(\TT^d)^3}\lambda_k^2\nabla e_k(x)\cdot\nabla e_k(x')\sigma(\rho^1)\sigma(\rho^2)\overline{\kappa}^{\ve,\delta}_{s,1}\overline{\kappa}^{\ve,\delta}_{s,2}\varphi_\beta(\eta)\zeta_M(\eta)
\\ \nonumber &  \quad + \sum_{k=1}^\infty \int_0^t\int_{\R}\int_{(\TT^d)^3}\lambda_k^2\sigma'(\rho^1)\sigma(\rho^2)e_k(x)\nabla e_k(x')\cdot\nabla\rho^1\overline{\kappa}^{\ve,\delta}_{s,1}\overline{\kappa}^{\ve,\delta}_{s,2}\varphi_\beta(\eta)\zeta_M(\eta)
\\ \nonumber &  \quad + \sum_{k=1}^\infty \int_0^t\int_{\R}\int_{(\TT^d)^3}\lambda_k^2\sigma(\rho^1)\sigma'(\rho^2)e_k(x')\nabla e_k(x)\cdot\nabla\rho^2\overline{\kappa}^{\ve,\delta}_{s,1}\overline{\kappa}^{\ve,\delta}_{s,2}\varphi_\beta(\eta)\zeta_M(\eta).
\end{align}
It follows from \eqref{it_09}, \eqref{it_0009}, and \eqref{3_009} that
\begin{align*}
& \left.\int_{\R}\int_{\TT^d}\chi^{\ve,\delta}_{s,1}(y,\eta)\chi^{\ve,\delta}_{s,2}(y,\eta)\varphi_\beta(\eta)\zeta_M(\eta)\dy\deta\right|^t_{s=0}
\\ & = I^{\textrm{err}}_t+I^{\textrm{meas}}_t+I^{\textrm{mix},\textrm{cut}}_t+I^{\textrm{mix},\textrm{mart}}_t+I^{\textrm{mix},\textrm{cons}}_t+I^{\textrm{mix},\textrm{int}}_t,
\end{align*}
where the error terms \eqref{it_000009} and \eqref{3_009} combine to form, using Einstein's summation convention over repeated indices,\small
\begin{align*}
& I^{\textrm{err}}_t = -\int_0^t\int_{\R}\int_{(\TT^d)^3}\left([\Phi'(\rho^1)]^\frac{1}{2}-[\Phi'(\rho^2)]^\frac{1}{2}\right)^2\nabla\rho^1\cdot\nabla\rho^2\overline{\kappa}^{\ve,\delta}_{s,1}\overline{\kappa}^{\ve,\delta}_{s,2}\varphi_\beta\zeta_M
\\ \nonumber & -\frac{1}{2} \int_0^t\int_{\R}\int_{(\TT^d)^3}(F_1[\sigma'(\rho^1)]^2+F_1[\sigma'(\rho^2)]^2)
\\ \nonumber & +\int_0^t\int_{\R}\int_{(\TT^d)^3}\lambda_k^2e_k(x)e_k(x')\sigma'(\rho^1)\sigma'(\rho^2))\overline{\kappa}^{\ve,\delta}_{s,1}\overline{\kappa}^{\ve,\delta}_{s,2}\varphi_\beta\zeta_M
\\ \nonumber & -\frac{1}{2} \int_0^t\int_{\R}\int_{(\TT^d)^3}(F_3(x)\sigma^2(\rho^1)+F_3(x')\sigma^2(\rho^2)
\\ \nonumber & + \int_0^t\int_{\R}\int_{(\TT^d)^3}\lambda_k^2\nabla e_k(x)\cdot\nabla e_k(x')\sigma(\rho^1)\sigma(\rho^2))\overline{\kappa}^{\ve,\delta}_{s,1}\overline{\kappa}^{\ve,\delta}_{s,2}\varphi_\beta\zeta_M
\\ \nonumber & + \int_0^t\int_{\R}\int_{(\TT^d)^3}\sigma'(\rho^1)\sigma(\rho^2)(\lambda_ke_k)(x)\nabla (\lambda_ke_k)(x'))\cdot\nabla\rho^1\overline{\kappa}^{\ve,\delta}_{s,1}\overline{\kappa}^{\ve,\delta}_{s,2}\varphi_\beta\zeta_M
\\ \nonumber & +\int_0^t\int_{\R}\int_{(\TT^d)^3}\sigma(\rho^1)\sigma'(\rho^2)(\lambda_ke_k)(x')\nabla (\lambda_ke_k)(x)\cdot\nabla\rho^2\overline{\kappa}^{\ve,\delta}_{s,1}\overline{\kappa}^{\ve,\delta}_{s,2}\varphi_\beta\zeta_M,
\end{align*}
\normalsize where the measure terms \eqref{it_0000009} combine to form
\begin{align*}
I^{\textrm{meas}}_t & =  \int_0^t\int_{\R^2}\int_{(\TT^d)^3}\kappa^{\ve,\delta}(x,y,\xi,\eta)\overline{\kappa}^{\ve,\delta}_{s,2}\varphi_\beta(\eta)\zeta_M(\eta)\dd q^1(x,\xi,s)
\\ \nonumber & \quad +\int_0^t\int_{\R^2}\int_{(\TT^d)^3}\kappa^{\ve,\delta}(x',y,\xi',\eta)\overline{\kappa}^{\ve,\delta}_{s,1}\varphi_\beta(\eta)\zeta_M(\eta)\dd q^2(x',\xi',s)
\\ \nonumber & \quad -2\int_0^t\int_{\R}\int_{(\TT^d)^3}[\Phi'(\rho^1)]^\frac{1}{2}[\Phi'(\rho^2)]^\frac{1}{2}\nabla\rho^1\cdot\nabla\rho^2\overline{\kappa}^{\ve,\delta}_{s,1}\overline{\kappa}^{\ve,\delta}_{s,2}\varphi_\beta(\eta)\zeta_M(\eta),
\end{align*}
and where the interaction terms \eqref{itt_0} and \eqref{itt_1} combine to form
\begin{align*}
I^{\textrm{int}}_t &  = \int_0^t\int_{\R}\int_{(\TT^d)^2}(2\chi^{\ve,\delta}_{s,2}-1)(\nabla\cdot(\lambda(\rho^1)B(\rho^1))+f(\rho^1))\overline{\kappa}^{\ve,\delta}_{s,1}\varphi_\beta(\eta)\zeta_M(\eta)
\\ & \quad + \int_0^t\int_{\R}\int_{(\TT^d)^2}(2\chi^{\ve,\delta}_{s,1}-1)(\nabla\cdot(\lambda(\rho^2)B(\rho^2))+f(\rho^2))\overline{\kappa}^{\ve,\delta}_{s,2}\varphi_\beta(\eta)\zeta_M(\eta).
\end{align*}
Then, for the cut-off, martingale, and conservative terms defined respectively by
\[I^{\textrm{cut,...}}_t= I^{1,\textrm{cut,...}}_t+I^{2,\textrm{cut,...}}_t-2(I^{2,1,\textrm{cut,...}}_t+I^{1,2,\textrm{cut,...}}_t),\]
we have from \eqref{it_38} and \eqref{it_0009} that, almost surely for every $t\in[0,T]$,
\begin{align}\label{it_21}
& \left.\int_{\R}\int_{\TT^d}\left(\chi^{\ve,\delta}_{s,1}+\chi^{\ve,\delta}_{s,2}-2\chi^{\ve,\delta}_{s,1}\chi^{\ve,\delta}_{s,2}\right)\varphi_\beta\zeta_M\right|_{s=0}^t
\\ \nonumber & =-2I^{\textrm{err}}_t-2I^{\textrm{meas}}_t+I^{\textrm{mart}}_t+I^{\textrm{cut}}_t+I^{\textrm{cons}}_t+I^{\textrm{int}}_t.
\end{align}
We will handle the five terms on the right hand side of \eqref{it_21} separately.

\textit{The measure term}.  It follows from property \eqref{it_5} of the kinetic measure and H\"older's inequality that the measure term almost surely satisfies, for every $t\in[0,T]$,
\begin{equation}\label{3_22} I^{\textrm{meas}}_t\geq 0.\end{equation}

\textit{The error term}.  The argument leading from \cite[Equation~(4.24)]{FG21} to \cite[Equation~(4.26]{FG21} proves almost surely for every $t\in[0,T]$ that
\begin{equation}\label{3_12}\limsup_{\delta\rightarrow 0}\left(\limsup_{\ve\rightarrow 0}\abs{I^{\textrm{err}}_t}\right)=0.\end{equation}

\textit{The martingale term}.  The argument leading from \cite[Equation~(4.27)]{FG21} to \cite[Equation~(4.36)]{FG21} proves almost surely for every $t\in[0,T]$ that
\begin{equation}\label{3_29}  \lim_{M\rightarrow\infty}\left(\lim_{\beta\rightarrow 0}\left(\lim_{\delta\rightarrow 0}\left(\lim_{\ve\rightarrow 0} I^\textrm{mart}_t\right)\right)\right) =0.\end{equation}

\textit{The conservative term.}  The argument leading to \cite[Equation~(4.37)]{FG21} proves proves almost surely along subsequences that, for every $t\in[0,T]$,
\begin{equation}\label{3_29292929}  \lim_{M\rightarrow\infty}\left(\lim_{\beta\rightarrow 0}\left(\lim_{\delta\rightarrow 0}\left(\lim_{\ve\rightarrow 0} I^\textrm{cons}_t\right)\right)\right) =0.\end{equation}

\textit{The cut-off term.}  The argument leading to \cite[Equation~(4.38)]{FG21} proves almost surely for every $t\in[0,T]$ that
\begin{equation}\label{3_18} \lim_{M\rightarrow\infty}\left(\lim_{\beta\rightarrow 0}\left(\lim_{\delta\rightarrow 0}\left(\lim_{\ve\rightarrow 0}I^{\textrm{cut}}_t\right)\right)\right)=0.\end{equation}

\textit{The interaction term.}  The new elements of the proof involve the treatment of the interaction term defined for every $t\in[0,T]$ by
\begin{align*}
I^{\textrm{int}}_t & = \int_0^t\int_{\R}\int_{(\TT^d)^2}(2\chi^{\ve,\delta}_{s,2}-1)(\nabla\cdot(\lambda(\rho^1)B(\rho^1))+f(\rho^1))\overline{\kappa}^{\ve,\delta}_{s,1}\varphi_\beta(\eta)\zeta_M(\eta)
\\ & \quad + \int_0^t\int_{\R}\int_{(\TT^d)^2}(2\chi^{\ve,\delta}_{s,1}-1)(\nabla\cdot(\lambda(\rho^2)B(\rho^2))+f(\rho^2))\overline{\kappa}^{\ve,\delta}_{s,2}\varphi_\beta(\eta)\zeta_M(\eta).
\end{align*}
It follows from the boundedness of the kinetic function that, almost surely for every $t\in[0,T]$,
\begin{align*}
\lim_{\ve\rightarrow 0} I^{\textrm{int}}_t & = \int_0^t\int_{\R}\int_{\TT^d}(2\chi^{\delta}_{s,2}-1)(\nabla\cdot(\lambda(\rho^1)B(\rho^1))+f(\rho^1))\overline{\kappa}^{\delta}_{s,1}\varphi_\beta(\eta)\zeta_M(\eta)
\\ & \quad + \int_0^t\int_{\R}\int_{\TT^d}(2\chi^{\delta}_{s,1}-1)(\nabla\cdot(\lambda(\rho^2)B(\rho^2))+f(\rho^2))\overline{\kappa}^{\delta}_{s,2}\varphi_\beta(\eta)\zeta_M(\eta),
\end{align*}
for $\chi^\delta_{s,i} = (\chi*\kappa^\delta_1)$ and $\overline{\kappa}^\delta_{s,i} = \kappa^\delta_1(\xi-\rho^i)$.  A repetition of the arguments leading from \cite[Equation~(4.28)]{FG21} to \cite[Equation~(4.31)]{FG21} proves that, almost surely for every $t\in[0,T]$,
\begin{align}\label{itt_3}
\lim_{\delta\rightarrow 0}(\lim_{\ve\rightarrow 0} I^\textrm{int}_t) & = \int_0^t\int_{\TT^d} \sign(\rho_2-\rho_1)(\nabla\cdot(\lambda(\rho^1)B(\rho^1))+f(\rho^1))\varphi_\beta(\rho^1)\zeta_M(\rho^1)
\\ \nonumber & \quad - \int_0^t\int_{\TT^d} \sign(\rho_2-\rho_1)(\nabla\cdot(\lambda(\rho^2)B(\rho^2))+f(\rho^2))\varphi_\beta(\rho^2)\zeta_M(\rho^2)
\\& \quad =: I^\textrm{int,B}_t +  I^\textrm{int,f}_t.
\end{align}
For the second term on the righthand side of \eqref{itt_3}, we first note using the boundedness of $f$, the $L^1$-integrability of the solutions, and the dominated convergence theorem that
\[\lim_{M\rightarrow\infty}\left(\lim_{\beta\rightarrow 0} I^\textrm{int,f}_t\right) = \int_0^t\int_{\TT^d} \sign(\rho_2-\rho_1)(f(\rho^1)-f(\rho^2)),\]
and it then follows from the one-sided Lipschitz continuity of $f$ that
\begin{equation}\label{itt_4}\lim_{M\rightarrow\infty}\left(\lim_{\beta\rightarrow 0} I^\textrm{int,f}_t\right) \leq \norm{f}_{{\Lip}}\int_0^t\int_{\TT^d}\abs{\rho^1-\rho^2}.\end{equation}
For the first term on the righthand side of \eqref{itt_3}, for every $\beta\in(0,1)$ and $M\in\N$ let $\Theta_{\beta,M}\colon[0,\infty)\rightarrow\R$ be the unique function defined by $\Theta_{\beta,M}(0)=0$ and $\Theta'_{\beta,M}(\xi)=\lambda'(\xi)\varphi_\beta(\xi)\zeta_M(\xi)$.  We then have that
\begin{align}\label{it_40}
& \lim_{\delta\rightarrow 0}(\lim_{\ve\rightarrow 0} I^\textrm{int,B}_t)
\\ \nonumber & = \int_0^t\int_{\TT^d} \sign(\rho_2-\rho_1)\nabla\cdot(\Theta_{\beta,M}(\rho^1)B(\rho^1)-\Theta_{\beta,M}(\rho^2)B(\rho^2))
\\ \nonumber & \quad + \int_0^t\int_{\TT^d} \sign(\rho_2-\rho_1)(\nabla\cdot B(\rho^1))(\lambda(\rho^1)\varphi_\beta(\rho^1)\zeta_M(\rho^1)-\Theta_{\beta,M}(\rho^1))
\\ \nonumber & \quad - \int_0^t\int_{\TT^d} \sign(\rho_2-\rho_1)(\nabla\cdot B(\rho^2))(\lambda(\rho^2)\varphi_\beta(\rho^1)\zeta_M(\rho^2)-\Theta_{\beta,M}(\rho^2)).
\end{align}
We will first treat the first term on the right hand side of \eqref{it_40}.  We form the decomposition
\begin{align}\label{it_41}
& \int_0^t\int_{\TT^d} \sign(\rho_2-\rho_1)\nabla\cdot(\Theta_{\beta,M}(\rho^1)B(\rho^1)-\Theta_{\beta,M}(\rho^2)B(\rho^2))
\\ \nonumber & = \int_0^t\int_{\TT^d} \sign(\rho_2-\rho_1)\nabla\cdot(\Theta_{\beta,M}(\rho^1)B(\rho^1)-\Theta_{\beta,M}(\rho^2)B(\rho^1))
\\ \nonumber & \quad + \int_0^t\int_{\TT^d} \sign(\rho_2-\rho_1)\nabla\cdot(\Theta_{\beta,M}(\rho^2)(B(\rho^1)-B(\rho^2))).
\end{align}
The first term on the right hand side of \eqref{it_41} is handled analogously to \cite[Equation~(4.33)]{FG21}, which uses the the Lipschitz continuity of $\Theta_{\beta,M}$ on $[0,\infty)$, the boundedness of $B(\rho^1)$ given by the assumption that
\[\norm{B(\rho^1)}_{L^\infty(\TT^d\times[0,T];\R^d)}\leq c\norm{\rho^1}_{L^\infty([0,T];L^1(\TT^d)}\leq c\norm{\rho_{0,1}}_{L^1(\TT^d)},\]
the local regularity of $\rho^1$ and $\rho^2$, and an approximation of the $\sign$ function to make rigorous the formal equality that, almost surely for every $t\in[0,T]$,
\begin{align*}
& \int_0^t\int_{\TT^d} \sign(\rho_2-\rho_1)\nabla\cdot(\Theta_{\beta,M}(\rho^1)B(\rho^1)-\Theta_{\beta,M}(\rho^2)B(\rho^1))
\\ & -\int_0^t\int_{\TT^d}\delta_0(\rho^2-\rho^1)(\Theta_{\beta,M}(\rho^1)-\Theta_{\beta,M}(\rho^2))(\nabla\rho^1-\nabla\rho^2)\cdot B(\rho^1)=0.
\end{align*}
For the second term on the right hand side of \eqref{it_41}, it follows from the $W^{1,1}$-regularity of $\lambda(\rho^2)$ and the definition of the cut-off functions $\varphi_\beta$ and $\zeta_M$ that, almost surely as $\beta\rightarrow 0$ and $M\rightarrow\infty$,
\[\Theta_{\beta,M}(\rho^2)\rightarrow \lambda(\rho^2)\;\;\textrm{strongly in}\;\;L^1([0,T];W^{1,1}(\TT^d)).\]
It then follows from the Lipschitz continuity of $B$ from $L^1(\TT^d)$ to $W^{1,\infty}(\TT^d;\R^d)$ that, almost surely for every $t\in[0,T]$,
\begin{align*}
& \lim_{M\rightarrow\infty}\left(\lim_{\beta\rightarrow 0}\left|\int_0^t\int_{\TT^d} \sign(\rho_2-\rho_1)\nabla\cdot(\Theta_{\beta,M}(\rho^2)(B(\rho^1)-B(\rho^2))\right|\right)
\\ & = \left|\int_0^t\int_{\TT^d} \sign(\rho_2-\rho_1)\nabla\cdot(\lambda(\rho^2)(B(\rho^1)-B(\rho^2)))\right|
\\ & \leq \norm{B}_{\Lip}\int_0^t\norm{\rho^1(\cdot,s)-\rho^2(\cdot,s)}_{L^1(\TT^d)}\int_{\TT^d}\abs{\lambda(\rho^2(\cdot,s))}+\abs{\nabla\lambda(\rho^2(\cdot,s))},
\end{align*}
for $ \norm{B}_{\Lip}$ the Lipschitz constant of the map $B\colon L^1(\TT^d)\rightarrow W^{1,\infty}(\TT^d)$.  This completes the analysis of \eqref{it_41}.  Returning to the final two terms of \eqref{it_40}, it follows from the $L^1$-integrability of $\lambda(\rho^i)$, the assumption that $\lambda(0)=0$, and the definitions of the cutoff functions that almost surely, as $\beta\rightarrow 0$ and $M\rightarrow\infty$,
\[\lambda(\rho^i)\varphi_\beta(\rho^i)\zeta_M(\rho^i)\rightarrow\lambda(\rho^i)\;\;\textrm{strongly in}\;\;L^1(\TT^d\times[0,T]),\]
and
\[\Theta_{M,\beta}(\rho^i)\rightarrow\lambda(\rho^i)\;\;\textrm{strongly in}\;\;L^1(\TT^d\times[0,T]).\]
It then follows from the boundedness of $\nabla \cdot B(\rho^i)$ that, almost surely for every $t\in[0,T]$, as $\beta\rightarrow 0$ and $M\rightarrow\infty$,
\[\left|\int_0^t\int_{\TT^d} \sign(\rho_2-\rho_1)(\nabla\cdot B(\rho^i))(\lambda(\rho^1)\varphi_\beta(\rho^i)\zeta_M(\rho^i)-\Theta_{\beta,M}(\rho^i))\right|\rightarrow 0.\]
Returning to \eqref{itt_3}, it follows from \eqref{itt_4} and \eqref{it_40} that, almost surely that for every $t\in[0,T]$,
\begin{align*}
& \lim_{M\rightarrow\infty}(\lim_{\beta\rightarrow 0}(\lim_{\delta\rightarrow 0}(\lim_{\ve\rightarrow 0} I^\textrm{int}_t)))
\\ & \leq \left( \norm{B}_{\Lip}\int_{\TT^d}\abs{\lambda(\rho^2(\cdot,s))}+\abs{\nabla\lambda(\rho^2(\cdot,s))} + \|f\|_{\Lip}\right)\int_0^t\norm{\rho^1(\cdot,s)-\rho^2(\cdot,s)}_{L^1(\TT^d)},
\end{align*}
which completes the analysis of the interaction term.

\textit{Conclusion.}  Returning to \eqref{it_21}, it follows from properties of the kinetic function that that there almost surely exist random subsequences $\ve,\delta,\beta\rightarrow 0$ and $M\rightarrow\infty$ such that, for every $t\in[0,T]$,
\begin{align*}
& \left.\int_{\R}\int_{\TT^d}\abs{\chi^1_t-\chi^2_t}^2\right|_{s=0}^{s=t}
\\& = \lim_{M\rightarrow\infty}\left(\lim_{\beta\rightarrow 0}\left(\lim_{\delta\rightarrow 0}\left(\lim_{\ve\rightarrow 0}\left.\int_{\R}\int_{\TT^d}\abs{\chi^{\ve,\delta}_{t,1}-\chi^{\ve,\delta}_{t,2}}^2\varphi_\beta\zeta_M\right|_{s=0}^{s=t}\right)\right)\right)
\\ & \leq \lim_{M\rightarrow\infty}\left(\lim_{\beta\rightarrow 0}\left(\lim_{\delta\rightarrow 0}\left(\lim_{\ve\rightarrow 0}\left(-2I^{\textrm{err}}_t-2I^{\textrm{meas}}_t+I^{\textrm{mart}}_t+I^{\textrm{cut}}_t+I^{\textrm{cons}}_t+I^{\textrm{int}}_t\right)\right)\right)\right)
\\ & \leq \left(\norm{B}_{\Lip}\int_{\TT^d}\abs{\lambda(\rho^2(\cdot,s))}+\abs{\nabla\lambda(\rho^2(\cdot,s))} + \|f\|_{\Lip}\right)\int_0^t\norm{\rho^1(\cdot,s)-\rho^2(\cdot,s)}_{L^1(\TT^d)}.
\end{align*}
It then follows from Gr\"onwall's inequality and properties of the kinetic function that, almost surely for every $t\in[0,T]$,
\begin{align*}
& \norm{\rho^1(\cdot,t)-\rho^2(\cdot,t)}_{L^1(\TT^d)}
\\ & \leq \norm{\rho_{0,1}-\rho_{0,2}}_{L^1(\TT^d)}\exp\left(t(\norm{B}_{{\Lip}}\norm{\lambda(\rho^2)}_{L^1([0,t];W^{1,1}(\TT^d))}+\|f\|_{\Lip})\right),
\end{align*}
where, due to the symmetry of the argument with respect to $\rho^1$ and $\rho^2$, we have that
\begin{align*}
& \norm{\rho^1(\cdot,t)-\rho^2(\cdot,t)}_{L^1(\TT^d)}
\\ & \leq \norm{\rho_{0,1}-\rho_{0,2}}_{L^1(\TT^d)}\exp\left(t(\norm{B}_{{\Lip}}\min_{i\in\{1,2\}}\norm{\lambda(\rho^i)}_{L^1([0,t];W^{1,1}(\TT^d))}+\|f\|_{\Lip})\right),
\end{align*}
which completes the proof.  \end{proof}

\section{Gradient estimates for SPDEs with conservative noise}

\label{sec:gradient} In this section, we establish a stable a priori estimate for certain nonlinear gradients of the solution to the SPDE 
\begin{equation}
\partial_{t}\rho=\Delta\Phi(\rho)-\nabla\cdot(\nu(\rho)+\lambda(\rho)B(\rho))-f(\rho)-\nabla\cdot\left(\sigma(\rho)\dd\xi\right),\label{ge_1}
\end{equation}
with spatially stationary It\^o-noise $\xi=\sum\lambda_{k}e_{k} B_{t}^{k}$ satisfying Assumption~\ref{assnoise}.  

Let $\{e_{k}\}_{k\in\N}$ be a smooth orthonormal $L^{2}(\TT^{d})$-basis and for every $N\in\N$ let $L_{T}^{2}(\TT^{d})=\Pi_{N}(L^{2}(\TT^{d}))$ where $\Pi_{N}$ denotes the Fourier projection 
\[
\Pi_{N}f=\sum_{k=1}^{N}\hat{f}_{k}e_{k}\;\;\textrm{for}\;\;\hat{f}_{k}=\int_{\TT^{d}}f(x)e_{k}(x).
\]

Due to the nonlinear structure of the diffusion in \eqref{ge_1}, the estimates derived in this section are not relying on standard $H^1$-based apriori estimates, but instead on estimating $\|\nabla \Phi(\rho)\|_{L^2}^2$. We observe, however, that such nonlinear a priori estimates are incompatible with standard Galerkin approximations.  In fact, it is unclear if $\|\nabla \Phi(\Pi_{N}\rho_0)\|_{L^2}^2$ is uniformly bounded in $N$, see, for example \cite{Gess2012}. We here develop a new argument bypassing these issues: Instead of considering a Galerkin approximation of the SPDE \eqref{ge_1}, we first, informally, apply It\^o's formula to derive an SPDE for $v:=\Phi(\rho)$, and project to the first $N$ Fourier modes only afterwards. Precisely, we consider
\begin{equation}\begin{split}
dv_{N}&=d\Phi(\rho_{N})=\Pi_{N}\Phi'(\rho_{N})\Big(\Delta\Phi(\rho_{N})-\nabla\cdot(\nu(\rho_{N})+\lambda(\rho_{N})B(\rho_{N}))\\
&-f(\rho_{N})-\nabla\cdot\left(\sigma(\rho_{N})\dd\xi\right)\Big)+\frac{1}{2}\Pi_{N}\Phi''(\rho_{N})\sum_{k=1}^\infty|\nabla\cdot\left(\sigma(\rho_{N})\lambda_{k}e_{k}\right)|^{2},\label{ge_2}
\end{split}\end{equation}
where $\rho_{N}:=\Phi^{-1}(v_{N})$, and we extend $\Phi$ to all of $\R$ by setting $\Phi(r)=-\Phi(-r)$. 
In this section we first derive uniform apriori estimates on $v_N$, and prove that after passing to a limit $v_N\to v$, these estimates can be transferred back to the original solution $\rho$ to \eqref{ge_1} as desired. Notably, the conservative structure of the noise in  \eqref{ge_1} leads to quartic error terms, which have to be treated carefully in the derivation of uniform estimates on $v_N$.

\begin{prop}
Let $\xi$, $\Phi$, $\sigma$, $\nu$, $\lambda$, $f$, and $B$ satisfy Assumptions~\ref{assumption_1}, \ref{assumption_2}, \ref{assumption_3}, and \ref{assume_B} and let $\rho_0\in L^1(\TT^d)$.  Then, for every $N\in\N$, there exists a unique strong solution $v_{N}\in L^{2}([0,T];L_{N}^{2}(\TT^{d}))$ of \eqref{ge_2}. 
\end{prop}

\begin{proof}
For an element $v_N=\sum_{k=1}^{N}\lambda_{k}(t)e_{k}(x)$ of $L^{2}([0,T];L_{N}^{2}(\TT^{d}))$ equation \eqref{ge_2} is equivalent to the system of SDEs, for every $k\in\{1,\ldots,N\}$, 
\begin{align*}
  \dd\lambda_{k}
  &=-\int_{\TT^{d}}\nabla e_{k}\cdot(\Phi'(\rho_N)\nabla\rho_N-\nu(\rho_N)-\lambda(\rho_N)B(\rho_N))\dt\\
  &-\int_{\TT^{d}} e_{k} f(\rho_N)\dt+\int_{\TT^{d}}\sigma(\rho_N)\nabla e_{k}\cdot\dd\xi,
\end{align*}
with $\rho_N := \Phi^{-1}(v_N)$, which, due to the assumptions on $\Phi$, $\nu$, $\lambda$, $B$, $f$, and $\sigma$, has a unique strong solution, which completes the proof. 
\end{proof}
\begin{assumption}\label{assume_grad}   Let $\sigma,\Phi\in \C([0,\infty))\cap\C^{2}((0,\infty))$ have bounded first and second derivatives, let $\xi=\sum_{k=1}^{\infty}\lambda_{k}e_{k} B_{t}^{k}$ satisfy Assumption~\ref{assnoise}, and assume that $\sigma$, $\Phi$, and $\xi$ satisfy the following properties. 
\begin{enumerate}
\item Stochastic coercivity: there exists $c_{1}\in(0,\infty)$ such that 
\[
\Phi'(\xi)-\frac{F_{1}}{2}(\sigma'(\xi))^{2}\geq c_{1}\quad\forall \xi\in(0,\infty).
\]
\item Decay of the second derivatives: there exists $c\in(0,\infty)$ such that 
\[
\abs{\sigma''(\xi)}+\abs{\Phi''(\xi)}\leq c(1+\Phi(\xi))^{-2}\quad\forall \xi\in(0,\infty).
\]
\item Regularity of the noise: we have that 
\[
F_4:=\sum_{k=1}^{\infty}\lambda_{k}^{2}\left\|\nabla^2e_{k}\right\|_\infty^{2}<\infty.
\]
\end{enumerate}
\end{assumption}
\begin{prop}
\label{prop_grad_est} Let $\xi$, $\Phi$, $\sigma$, $\nu$, $\lambda$, $f$, and $B$ satisfy Assumptions~\ref{assumption_1}, \ref{assumption_2}, \ref{assumption_3}, \ref{assume_B}, and \ref{assume_grad}. Let $N\in\N$, let $T\in(0,\infty)$, let $v_{N}\in L^{2}([0,T];L_{N}^{2}(\TT^{d}))$ be the unique strong solution of \eqref{ge_2}, and let $\rho_{N}:=\Phi^{-1}(v_{N}).$ Then, for all $p\in[2,\infty)$, for all $F_1$ small enough, there exists $c\in(0,\infty)$ depending on the $L^m(\Omega;L^1(\TT^d))$-norm of the initial data, $p$, and $T$ but independent of $N\in \N$ such that 
\begin{align}
 & \sup_{t\in[0,T]}\E\left(\int_{\TT^{d}}\abs{\nabla\Phi(\rho_{N}(t))}^{2}+1\right)^{\frac{p}{2}}+\E\int_{0}^{T}\left(\int_{\TT^{d}}\abs{\nabla\Phi(\rho_{N})}^{2}+1\right)^{\frac{p-2}{2}}\int_{\TT^{d}}(\Delta\Phi(\rho_{N}))^{2}\label{grad_est}\\
 & \leq c\left(\E\left(\int_{\TT^{d}}\abs{\nabla\Phi(\rho_{N}(0))}^{2}+1\right)^{\frac{p}{2}}+1\right).\nonumber
\end{align}
\end{prop}

\begin{proof}

It follows from the $\C^2$-regularity of $\sigma$ and $\Phi$ and the smoothness of the $e_{k}$ that, almost surely for every $x\in\TT^{d}$, the gradient of $v_N=\Phi(\rho_{N})$ solves the SDE obtained by formally differentiating \eqref{ge_2}. We then obtain from It\^o's formula that 
\begin{align}
\dd\left(\frac{1}{2}\int_{\TT^{d}}\abs{\nabla v_{N}}^{2}\right) & =\dd\left(\frac{1}{2}\int_{\TT^{d}}\abs{\nabla\Phi(\rho_{N})}^{2}\right)=-\int_{\TT^{d}}\Phi'(\rho_{N})(\Delta\Phi(\rho_{N}))^{2}\label{ge_4}\\
 & -\int_{\TT^{d}}\Delta\Phi(\rho_{N})\Phi'(\rho_{N})\nabla\cdot(\sigma(\rho_{N})\dd\xi)\nonumber \\
 & -\int_{\TT^{d}}\Delta\Phi(\rho_{N})\Phi'(\rho_{N})(\nabla\cdot(\nu(\rho_{N})+\lambda(\rho_{N})B(\rho_{N}))-f(\rho_N))\\
 & -\frac{1}{2}\int_{\TT^{d}}\Delta\Phi(\rho_{N})\Phi''(\rho_{N})(F_{1}\abs{\sigma'(\rho_{N})\nabla\rho_{N}}^{2}+F_{3}\sigma(\rho_{N})^{2})\nonumber \\
 & +\frac{1}{2}\int_{\TT^{d}}\sum_{k=1}^{\infty}\sum_{i=1}^{d}\abs{\nabla\left(\Phi'(\rho_{N})\partial_{x_i}\left(\sigma(\rho_{N})(\lambda_{k}e_{k}\right)\right)}^{2}.\nonumber 
\end{align}
For $p\ge2$, we may again apply It\^o's formula to obtain that
\begin{align}
 & \dd\E\left(\frac{1}{2}\int_{\TT^{d}}\abs{\nabla\Phi(\rho_{N})}^{2}+1\right)^{\frac{p}{2}}\label{ge_4-1}\\
 & =\frac{p}{2}\E\left(\frac{1}{2}\int_{\TT^{d}}\abs{\nabla\Phi(\rho_{N})}^{2}+1\right)^{\frac{p-2}{2}}\Bigg(-\int_{\TT^{d}}\Phi'(\rho_{N})(\Delta\Phi(\rho_{N}))^{2}\nonumber \\
 & -\int_{\TT^{d}}\Delta\Phi(\rho_{N})\Phi'(\rho_{N})(\nabla\cdot(\nu(\rho_{N})+\lambda(\rho_{N})B(\rho_{N}))-f(\rho_N))\nonumber \\
 & -\frac{1}{2}\int_{\TT^{d}}\Delta\Phi(\rho_{N})\Phi''(\rho_{N})(F_{1}\abs{\sigma'(\rho_{N})\nabla\rho_{N}}^{2}+F_{3}\sigma(\rho_{N})^{2})\nonumber \\
 & +\frac{1}{2}\int_{\TT^{d}}\sum_{k=1}^{\infty}\sum_{i=1}^{d}\abs{\nabla\left(\Phi'(\rho_{N})\partial_{x_i}\left(\sigma(\rho_{N})(\lambda_{k}e_{k}\right)\right)}^{2}\Bigg)\nonumber \\
 & +\frac{p(p-2)}{8}\E\left(\frac{1}{2}\int_{\TT^{d}}\abs{\nabla\Phi(\rho_{N})}^{2}+1\right)^{\frac{p-4}{2}}\sum_{k=1}^\infty\Big|\int_{\TT^{d}}\Delta\Phi(\rho_{N})\Phi'(\rho_{N})\nabla(\sigma(\rho_{N})\lambda_{k}e_{k})\Big|^{2}.\nonumber 
\end{align}
We first treat the last term on the right hand side of \eqref{ge_4-1}, for which we have
\begin{align*}
&\sum_{k=1}^\infty\Big|\int_{\TT^{d}}\Delta\Phi(\rho_{N})\Phi'(\rho_{N})\nabla\cdot(\sigma(\rho_{N})\lambda_{k}e_{k})\Big|^{2} \\
&\le\sum_{k=1}^\infty\Big(\int_{\TT^{d}}\Phi'(\rho_{N})(\Delta\Phi(\rho_{N}))^{2}\Big)\Big(\int_{\TT^{d}}\Phi'(\rho_{N})|\nabla(\sigma(\rho_{N})\lambda_{k}e_{k}|^{2}\Big).
\end{align*}
It then follows from Assumption~\ref{assume_grad} that, for some $c\in(0,\infty)$,
\begin{align*}
&\sum_{k=1}^\infty\int_{\TT^{d}}\Phi'(\rho_{N})|\nabla(\sigma(\rho_{N})\lambda_{k}e_{k}|^{2} \\
& \le 2\sum_{k=1}^\infty \int_{\TT^{d}}\Phi'(\rho_{N})(|\nabla\rho_{N}|^{2}(\sigma'(\rho_{N}))^{2}|\lambda_{k}e_{k}|^{2}+\sigma^{2}(\rho_{N})|\lambda_{k}\nabla e_{k}|^{2})\\
 & \le c\left(F_1\int_{\TT^{d}}F_1|\nabla\Phi(\rho_{N})|^{2}(\sigma'(\rho_{N}))^{2}+F_3\int_{\TT^{d}}\Phi'(\rho_{N})\sigma^{2}(\rho_{N})|\lambda_{k}\nabla   e_{k}|^{2}\right)\\
 & \le c\left(F_1\int_{\TT^{d}}|\nabla\Phi(\rho_{N})|^{2}+F_3\Big(\int_{\TT^{d}}\Phi^{2}(\rho_{N})+1\Big)\right)\\
 & \leq c(F_{1}+F_{3})\Big(\int_{\TT^{d}}|\nabla\Phi(\rho_{N})|^{2}+\left(\int_{\TT^d}\rho_N(0)\right)^m+1\Big),
\end{align*}
where in the final step we have used the $L^1$-integrability of $\rho_N$, Assumption~\ref{assumption_3}, and the interpolation estimate \cite[Lemma~5.4]{FG21} to deduce that, for some $c\in(0,\infty)$,
\begin{equation}\label{interpolate}\int_{\TT^d}\Phi(\rho_N)^2\leq c\left(\int_{\TT^d}\abs{\nabla\Phi(\rho_N)}^2+\left(\int_{\TT^d}\rho_n(0)\right)^m+1\right).\end{equation}
Hence, for some $c\in(0,\infty)$,
\begin{align*}
&\Big|\int_{\TT^{d}}\Delta\Phi(\rho_{N})\Phi'(\rho_{N})\nabla(\sigma(\rho_{N})\lambda_{k}e_{k})\Big|^{2} \\
& \le c(F_{1}+F_{3})\Big(\int_{\TT^{d}}|\nabla\Phi(\rho_{N})|^{2}+\left(\int_{\TT^d}\rho_N(0)\right)^m+1\Big),
\end{align*}
which implies that, for some $c\in(0,\infty)$ depending on $p$ and the norm of the initial data in $L^m(\Omega;L^1(\TT^d))$,
\begin{align*}
&\frac{p(p-2)}{8}\E\left(\frac{1}{2}\int_{\TT^{d}}\abs{\nabla\Phi(\rho_{N})}^{2}+1\right)^{\frac{p-4}{2}}\sum_{k=1}^\infty\Big|\int_{\TT^{d}}\Delta\Phi(\rho_{N})\Phi'(\rho_{N})\nabla\cdot(\sigma(\rho_{N})\lambda_{k}e_{k})\Big|^{2} \\
 & \leq c(F_{1}+F_{3})\E\left(\frac{1}{2}\int_{\TT^{d}}\abs{\nabla\Phi(\rho_{N})}^{2}+1\right)^{\frac{p-2}{2}}\Big(\int_{\TT^{d}}\Phi'(\rho_{N})(\Delta\Phi(\rho_{N}))^{2}\Big).
\end{align*} 

For the second term on the right hand side of \eqref{ge_4-1}, we have using Assumption~\ref{assume_grad} and the $L^{1}$-estimate for $\rho$ that there exists $c\in(0,\infty)$ depending on the Lipschitz norms of $\nu$, $B,f$, and $\lambda$, and on the lower bound of $\Phi'$, that 
\begin{align}
 & \abs{\int_{\TT^{d}}\Phi'(\rho_{N})(\nabla\cdot(\nu(\rho_{N})+\lambda(\rho_{N})B(\rho_{N}))+f(\rho_N))\Delta\Phi(\rho_{N})}\label{ge_50}\\
 & \leq c\left(\int_{\TT^{d}}\abs{\nabla\Phi(\rho_{N})}^{2}+|\Phi(\rho_{N})|^2 \right)^{\frac{1}{2}}\left(\int_{\TT^{d}}\Phi'(\rho_{N})(\Delta\Phi(\rho_{N}))^{2}\right)^{\frac{1}{2}}.\nonumber 
\end{align}
For the last term of \eqref{ge_4} we have that, for every $k\in\N$ and $i\in\{1,\ldots,d\}$, 
\begin{align*}
\sum_{i=1}^{d}|\nabla\left(\Phi'(\rho_{N})\partial_{x_i}(\sigma(\rho_{N})(\lambda_{k}e_{k}))\right)|^2
&=\sum_{i,j=1}^{d}|\partial_{x_j}\left(\Phi'(\rho_{N})\partial_{x_i}(\sigma(\rho_{N})(\lambda_{k}e_{k}))\right)|^2
\\&\le 4 (V_{1}+V_{2}+V_{3}+V_{4}),
\end{align*}
for $V_{i}$ depending on $k$ defined by 
\begin{align*}
V_{1}
&:=\sum_{i,j=1}^{d} |\partial_{x_j}(\Phi'(\rho_{N})\partial_{x_i}(\sigma(\rho_{N})))\cdot(\lambda_{k}e_{k})|^2\\
&=\sum_{i,j=1}^{d}|\partial_{x_j}(\partial_{x_i}\Phi(\rho_{N})(\sigma'(\rho_{N})))\cdot(\lambda_{k}e_{k})|^2,
\end{align*}
and for
\begin{align*}
& V_{2}:=\sum_{i,j=1}^{d}|\partial_{x_j}(\Phi'(\rho_{N})\sigma(\rho_{N}))\partial_{x_i}(\lambda_{k}e_{k})|^2,\;\;V_{3}:=\sum_{i,j=1}^{d}|\Phi'(\rho_{N})\partial_{x_j}(\sigma(\rho_{N}))\partial_{x_i}(\lambda_{k}e_{k})|^2\\
& \;\;\textrm{and}\;\;V_{4}:=\sum_{i,j=1}^{d}|\Phi'(\rho_{N})\sigma(\rho_{N})\partial_{x_j}\partial_{x_i}(\lambda_{k}e_{k})|^2.
\end{align*}
We observe that, 
\begin{align*}
\sum_{k=1}^{\infty}\int_{\TT^{d}}V_{1}& =F_{1}\sum_{i,j=1}^{d}\int_{\TT^{d}}\abs{\partial_{x_{j}}\partial_{x_{i}}\Phi(\rho_{N})}^{2}(\sigma'(\rho_{N}))^{2}\\
 & \quad+2F_{1}\int_{\TT^{d}}\sigma'(\rho_{N})\sigma''(\rho_{N})\partial_{x_{j}}\partial_{x_{i}}\Phi(\rho_{N})\partial_{x_{j}}\rho_{N}\\
 & \quad+F_{1}\int_{\TT^{d}}(\Phi'(\rho_{N}))^{2}(\sigma''(\rho_{N}))^{2}(\partial_{x_{i}}\rho_{N})^{2}(\partial_{x_{j}}\rho_{N})^{2}.
\end{align*}
We then have that 
\begin{align}
\sum_{k=1}^{\infty}\int_{\TT^{d}}V_{1} & =F_{1}\sum_{i,j=1}^{d}\int_{\TT^{d}}(\sigma'(\rho_{N}))^{2}(\partial_{x_{i}}\partial_{x_{j}}\Phi(\rho_{N}))^{2}\label{ge_6}\\
 & \quad+\sum_{i,j=1}^{d}2F_{1}\int_{\TT^{d}}\sigma'(\rho_{N})\sigma''(\rho_{N})\partial_{x_{j}}\rho_{N}\partial_{x_{i}}\Phi(\rho_{N})\partial_{x_{i}x_{j}}^{2}\Phi(\rho_{N})\nonumber \\
 & \quad+F_{1}\int_{\TT^{d}}(\Phi'(\rho_{N}))^{2}(\sigma''(\rho_{N}))^{2}\abs{\nabla\rho_{N}}^{4}\nonumber .
\end{align}
For the first term on the right hand side of \eqref{ge_6}, it follows from Assumption~\ref{assume_grad} that, after integrating by parts, 
\begin{align}
 & \sum_{i,j=1}^{d}\int_{\TT^{d}}(\sigma'(\rho_{N}))^{2}(\partial_{x_{i}}\partial_{x_{j}}\Phi(\rho_{N}))^{2}=\int_{\TT^{d}}(\sigma'(\rho_{N}))^{2}(\Delta(\Phi(\rho_{N})))^{2}\label{ge_8}\\
 & \quad-2\sum_{i,j=1}^{d}\int_{\TT^{d}}\sigma'(\rho_{N})\sigma''(\rho_{N})\partial_{x_{j}}\rho_{N}\partial_{x_{i}}\Phi(\rho_{N})\partial_{x_{i}x_{j}}^{2}\Phi(\rho_{N})\\
 & \quad+2\int_{\TT^{d}}\Phi'(\rho_{N})\sigma'(\rho_{N})\sigma''(\rho_{N})\abs{\nabla\rho_{N}}^{2}\Delta\Phi(\rho_{N}).
\end{align}
Returning to \eqref{ge_6}, we have from \eqref{ge_8} that 
\begin{align}
 \sum_{k=1}^{\infty}\int_{\TT^{d}} V_{1} & =2F_{1}\int_{\TT^{d}}\Phi'(\rho_{N})\sigma'(\rho_{N})\sigma''(\rho_{N})\abs{\nabla\rho_{N}}^{2}\Delta\Phi(\rho_{N})\label{ge_9_0}\\
 & \quad+F_{1}\int_{\TT^{d}}(\sigma'(\rho_{N}))^{2}(\Delta(\Phi(\rho_{N})))^{2}\nonumber \\
 & \quad+F_{1}\int_{\TT^{d}}(\Phi'(\rho_{N}))^{2}(\sigma''(\rho_{N}))^{2}\abs{\nabla\rho_{N}}^{4}.\nonumber 
\end{align}
It remains to estimate the terms involving $V_{i}$ for $i\in\{2,3,4\}$. For the case $i=2$ we have using Assumption~\ref{assume_grad} that, for some $c\in(0,\infty)$, 
\begin{align}
{ \sum_{k=1}^{\infty}\int_{\TT^{d}}V_{2}} & \leq cF_3\int_{\TT^{d}}\left(\frac{\Phi''(\rho_{N})\sigma(\rho_{N})}{\Phi'(\rho_{N})}+\sigma'(\rho_{N})\right)^{2}\abs{\nabla\Phi(\rho_{N})}^{2}\label{ge_9}\\
 & \leq cF_3\int_{\TT^{d}}\abs{\nabla\Phi(\rho_{n})}^{2}.\nonumber 
\end{align}
For the term involving $V_{3}$, it follows from Assumption~\ref{assume_grad} that, for some $c\in(0,\infty)$,
\begin{equation}
{ \sum_{k=1}^{\infty}\int_{\TT^{d}}  V_{3}}\leq cF_3\int_{\TT^{d}}(\sigma'(\rho_{N}))^{2}\abs{\nabla\Phi(\rho_{N})}^{2}\leq cF_3\int_{\TT^{d}}\abs{\nabla\Phi(\rho_{N})}^{2}.\label{ge_10}
\end{equation}
Finally, for the term involving $V_{4}$, we have from Assumption~\ref{assume_grad} and \eqref{interpolate} that, for some $c\in(0,\infty)$ depending on the $L^m(\Omega;L^1(\TT^d))$-norm of the initial data,
\begin{equation} \sum_{k=1}^{\infty}\int_{\TT^{d}} V_{4}\leq c\int_{\TT^{d}}(\Phi'(\rho_{N})\sigma(\rho_{N}))^{2}\leq c\int_{\TT^{d}}\Phi(\rho_{N})^{2}\leq c\left(\int_{\TT^d}\abs{\nabla\Phi(\rho_N)}^2+1\right).\label{ge_11}
\end{equation}
Returning to \eqref{ge_4}, it follows from \eqref{ge_6}, \eqref{ge_8}, \eqref{ge_9}, \eqref{ge_10}, and \eqref{ge_11}, the stochastic coercivity of Assumption~\ref{assume_grad}, and Young's inequality that, for some $c\in(0,\infty)$, 
\begin{align}
 & \dd\E\left(\frac{1}{2}\int_{\TT^{d}}\abs{\nabla\Phi(\rho_{N})}^{2}\right)^{\frac{p}{2}}\label{ge_12}\\
 & \leq c\E\left(\frac{1}{2}\int_{\TT^{d}}\abs{\nabla\Phi(\rho_{N})}^{2}\right)^{\frac{p-2}{2}}\Bigg(-c_1\int_{\TT^{d}}\Phi'(\rho_N)(\Delta\Phi(\rho_{N}))^{2}\nonumber \\
 & \quad-\frac{1}{2}\int_{\TT^{d}}\Phi''(\rho_{N})(F_{1}\abs{\sigma'(\rho_{N})\nabla\rho_{N}}^{2}+F_{3}\sigma(\rho_{N})^{2})\Delta\Phi(\rho_{N})\nonumber \\
 & \quad+F_{1}\int_{\TT^{d}}\Phi'(\rho_{N})\sigma'(\rho_{N})\sigma''(\rho_{N})\abs{\nabla\rho_{N}}^{2}\Delta\Phi(\rho_{N})\nonumber \\
 & \quad+\frac{F_{1}}{2}\int_{\TT^{d}}(\Phi'(\rho_{N}))^{2}(\sigma''(\rho_{N}))^{2}\abs{\nabla\rho_{N}}^{4}+c\left(\int_{\TT^{d}}\abs{ \nabla\Phi(\rho_{n})}^{2}+1\right)\Bigg).\nonumber
\end{align}
For the second term on the right hand side of \eqref{ge_12}, it follows from Assumption~\ref{assume_grad}, H\"older's inequality, and Young's inequality that, for some $c\in(0,\infty)$, 
\begin{align}
 & \abs{F_{1}\int_{\TT^{d}}\Phi''(\rho_{N})(\sigma'(\rho_{N}))^{2}\abs{\nabla\rho_{N}}^{2}\Delta\Phi(\rho_{N})}\label{ge_30}\\
 & \leq F_{1}\int_{\TT^{d}}\Phi''(\rho_{N})\abs{\nabla\Phi(\rho_{n})}^{2}\Delta\Phi(\rho_{N})\nonumber \\
 & \leq cF_{1}\left(\int_{\TT^{d}}\Phi''(\rho_{N})^{2}\abs{\nabla\Phi(\rho_{N})}^{4}\right)^{\frac{1}{2}}\left(\int_{\TT^{d}}\Phi'(\rho_N)(\Delta\Phi(\rho_{N}))^{2}\right)^{\frac{1}{2}}.\nonumber 
\end{align}
For the second term on the right hand side of \eqref{ge_12}, we have using Assumption~\ref{assume_grad} and H\"older's inequality that, for some $c\in(0,\infty)$, 
\begin{align}
 & \abs{\frac{1}{2}\int_{\TT^{d}}F_{3}\Phi''(\rho_{N})\sigma(\rho_{N})^{2}\Delta\Phi(\rho_{N})}\label{ge_14}\\
 & \leq cF_3\left(\int_{\TT^{d}}\Phi(\rho_{N})^{2}\right)^{\frac{1}{2}}\left(\int_{\TT^{d}}\Phi'(\rho_N)(\Delta(\Phi(\rho_{N})))^{2}\right)^{\frac{1}{2}}.\nonumber 
\end{align}
For the third term on the right hand side of \eqref{ge_12}, it follows from Assumption~\ref{assume_grad} and H\"older's inequality that, for some $c\in(0,\infty)$,
\begin{align}
 & \abs{F_{1}\int_{\TT^{d}}\Phi'(\rho_{N})\sigma'(\rho_{N})\sigma''(\rho_{N})\abs{\nabla\rho_{N}}^{2}\Delta\Phi(\rho_{N})}\label{ge_31}\\
 & =\abs{F_{1}\int_{\TT^{d}}\sigma''(\rho_{N})\abs{\nabla\Phi(\rho_{N})}^{2}\Delta\Phi(\rho_{N})}\\
 & \leq cF_{1}\left(\int_{\TT^{d}}\sigma''(\rho_{N})^{2}\abs{\nabla\Phi(\rho_{N})}^{4}\right)^{\frac{1}{2}}\left(\int_{\TT^{d}}\Phi'(\rho_N)(\Delta\Phi(\rho_{N}))^{2}\right)^{\frac{1}{2}}.\nonumber 
\end{align}
For the fourth term on the right hand side of \eqref{ge_12}, we have that 
\begin{equation}
\int_{\TT^{d}}(\Phi'(\rho_{N}))^{2}(\sigma''(\rho_{N}))^{2}\abs{\nabla\rho_{N}}^{4}\leq c\int_{\TT^{d}}(\sigma''(\rho_{N}))^{2}\abs{\nabla\Phi(\rho_{N})}^{4}.\label{ge_15}
\end{equation}
It follows from Young's inequality, \eqref{ge_50}, \eqref{ge_12}, \eqref{ge_30}, \eqref{ge_14}, \eqref{ge_31}, and \eqref{ge_15} that, for some $c\in(0,\infty)$,
\begin{align}
 & \dd\E\left(\frac{1}{2}\int_{\TT^{d}}\abs{\nabla\Phi(\rho_{N})}^{2}+1\right)^{\frac{p}{2}}\label{ge_32}\\
 & \leq c\E\left(\frac{1}{2}\int_{\TT^{d}}\abs{\nabla\Phi(\rho_{N})}^{2}+1\right)^{\frac{p-2}{2}}\Bigg(-\frac{c_1}{2}\int_{\TT^{d}}\Phi'(\rho_N)(\Delta\Phi(\rho_{N}))^{2}\nonumber \\
 & \quad + F_{1}\E\int_{\TT^{d}}(\sigma''(\rho_{N})^{2}+\Phi''(\rho_{N})^{2})\abs{\nabla\Phi(\rho_{N})}^{4} +\E\int_{\TT^{d}}\abs{\nabla\Phi(\rho_{n})}^{2}+1\Bigg)\nonumber.
\end{align}
We will now use Assumption~\ref{assume_grad} to treat the first two terms on the final line of \eqref{ge_32}. Precisely, we have using the decay of the second derivatives that, for some $c\in(0,\infty)$, after integrating by parts, 
\begin{align}
 & \int_{\TT^{d}}(\sigma''(\rho_{N})^{2}+\Phi''(\rho_{N})^{2})\abs{\nabla\Phi(\rho_{N})}^{4}\leq c\int_{\TT^{d}}(1+\Phi(\rho_{N}))^{-2}\abs{\nabla\Phi(\rho_{N})}^{4}\label{ge_55}\\
 & =c\int_{\TT^{d}}(1+\Phi(\rho_{N}))^{-1}\abs{\nabla\Phi(\rho_{N})}^{2}\Delta\Phi(\rho_{N})\nonumber \\
 & \quad+2c\sum_{i,j=1}^{d}\int_{\TT^{d}}(1+\Phi(\rho_{N}))^{-1}\partial_{x_{i}}\Phi(\rho_{N})\partial_{x_{j}}\Phi(\rho_{N})\partial_{x_{i}x_{j}}^{2}\Phi(\rho_{N}).\nonumber 
\end{align}
It then follows from H\"older's inequality that, for some $c\in(0,\infty)$, 
\begin{align*}
 & \int_{\TT^{d}}(1+\Phi(\rho_{N}))^{-2}\abs{\nabla\Phi(\rho_{N})}^{4}\\
 & \leq c\left(\int_{\TT^{d}}(1+\Phi(\rho_{N}))^{-2}\abs{\nabla\Phi(\rho_{N})}^{4}\right)^{\frac{1}{2}}\left(\left(\sum_{i,j=1}^{d}\int_{\TT^{d}}(\partial_{x_{i}x_{j}}^{2}\Phi(\rho_{N})\right)^{2}\right)^{\frac{1}{2}}.
\end{align*}
After integrating by parts and using the nondegeneracy of $\Phi$ in Assumption~\ref{assume_grad} for the final term, we have that, for some $c\in(0,\infty)$, 
\[
\int_{\TT^{d}}(1+\Phi(\rho_{N}))^{-2}\abs{\nabla\Phi(\rho_{N})}^{4}\leq c\int_{\TT^{d}}(\Delta\Phi(\rho_{N}))^{2}\leq c\int_{\TT^d}\Phi'(\rho_N)(\Delta(\Phi(\rho_N)))^2.
\]
In combination with \eqref{ge_32} and \eqref{ge_55}, we have that, for some $c_2,c\in(0,\infty)$,
\begin{align*}
 & \dd\E\left(\frac{1}{2}\int_{\TT^{d}}\abs{\nabla\Phi(\rho_{N})}^{2}+1\right)^{\frac{p}{2}}
 \\ & \leq \left(-\frac{c_1}{2}+c_2F_1\right)\E\left(\frac{1}{2}\int_{\TT^{d}}\abs{\nabla\Phi(\rho_{N})}^{2}+1\right)^{\frac{p-2}{2}}\int_{\TT^{d}}\Phi'(\rho_N)(\Delta\Phi(\rho_{N}))^{2}\nonumber 
 \\
  &\quad +c\left(\E\left(\int_{\TT^{d}}\abs{\nabla\Phi(\rho_{N})}^{2}+1\right)^{\frac{p}{2}}+1\right).
\end{align*}
For all $F_1\leq\frac{c_1}{4c_2}$ the claim is then a consequence of Gr\"onwalls inequality and the nondegeneracy of $\Phi$ in Assumption~\ref{assume_grad}.  \end{proof}
\begin{prop}
\label{prop_exist} Let $\xi$, $\Phi$, $\sigma$, $\nu$, $\lambda$, $f$, and $B$ satisfy Assumptions~\ref{assumption_1}, ~\ref{assumption_2}, \ref{assumption_3}, \ref{assume_B}, and \ref{assume_grad}, let $\rho_{0}\in L^m(\Omega;L^{1}(\TT^{d}))$ be nonnegative and $\mathcal{F}_{0}$-measurable, and let $T\in(0,\infty)$. Then, there exists a unique stochastic kinetic solution $\rho$ of \eqref{ge_1} in the sense Definition~\ref{def_sol_new}. Further, if $\Phi(\rho_{0})\in L^{2}(\Omega;H^{1}(\TT^{d}))$, $p\in[2,\infty)$ and $F_1$ is small enough then, for some $c\in(0,\infty)$ depending on the $L^m(\Omega;L^1(\TT^d))$-norm of $\rho_0$, $p$, and $T$,
\begin{equation}\label{eq:grad_est_2}
 \sup_{t\in[0,T]}\E\left( \int_{\TT^{d}}\abs{\nabla\Phi(\rho(t))}^{2}\right)^{\frac{p}{2}}\leq c\left(\E\left(\int_{\TT^{d}}\abs{\nabla\Phi(\rho(0))}^{2}\right)^{\frac{p}{2}}+1\right).
\end{equation}
 
\end{prop}

\begin{proof}
The proof is a simplified version of of \cite[Theorem~5.29]{FG21}, since in this case we have the stronger estimates of Proposition~\ref{prop_grad_est}.  The system \eqref{ge_2} used to approximate solutions is the same as that used in \cite[Proposition~5.20]{FG21}.
We first consider initial data that is bounded in $L^2(\Omega;L^2(\TT^d))\cap L^m(\Omega;L^1(\TT^d))$ and pass to the limit $N\rightarrow\infty$ using Assumptions~\ref{assumption_3} and \ref{assume_grad} and the estimates of \cite[Propositions~5.9, 5.14]{FG21} as in \cite[Theorem~5.29]{FG21}.  In fact, for this case, the proof can be considerably simplified using Assumption~\ref{assume_grad}.  We then establish the existence of solutions for general initial data in $L^m(\Omega;L^1(\TT^d))$ using the $L^1$-contraction of Theorem~\ref{thm_unique} as in the conclusion of \cite[Theorem~5.29]{FG21}.  Estimate \eqref{eq:grad_est_2} is then a consequence of the methods of \cite[Theorem~5.29]{FG21}, where it is shown that the laws of the $\Phi(\rho_N)$ are tight in the weak topology of $L^2([0,T];H^1(\TT^d))$, and the weak lower semicontinuity of the Sobolev norm.  This completes the proof.  \end{proof}

\section{Existence of stochastic flows and RDS}~\label{sec:rds}
In this section, we argue that, under the appropriate assumptions, stochastic kinetic solutions of~\eqref{eq:DK} generate a random dynamical system. We will follow the approach of~\cite{AS95} and~\cite[Chapter 6]{Arn98}. We start by introducing some preliminary notions. Let $(\Omega,\mathcal{F}^0,\mathbb{P})$ be a probability space and $(\Omega,\mathcal{F}^0, \mathbb{P},(\theta_t)_{t\in \R})$ a metric dynamical system (MDS) as defined in~\cite[Definition 3]{AS95}. We now proceed to construct an MDS associated to a two-side cylindrical Wiener process on $\ell^2(\Z^d)$ which one can obtain by gluing two independent cylindrical Wiener processes to each other at $t=0$. We set $\Omega = C_0(\R; H)$  where 
\[
C_0(\R;H) := \left\{\omega \in C(\R;H)  : \omega(0)=0\right\},
\]
and $H \supset \ell^2(\Z^d)$ is a separable Hilbert space such that the embedding $\iota: \ell^2(\Z^d) \to H$ is Hilbert--Schmidt. We equip $\Omega$ with the compact-open topology, i.e. the topology of uniform convergence on compacts, and denote by $\mathcal{F}^0$ the corresponding Borel sigma algebra (which is not complete). Let $\mathbb{P}$ be the law of a two-sided cylindrical Wiener process on $\Omega$ which is supported on $C_0(\R;H)$. Define the maps $\theta_{t}:\Omega \to \Omega, t \in \R$ by
\begin{equation}
\theta_t (\omega): = \omega (\cdot +t ) -\omega(t) \, . 
\end{equation} 
We then have the following result.
\begin{prop}
The tuple $(\Omega,\mathcal{F}^0, \mathbb{P}, (\theta_t)_{t\in \R})$ forms an MDS in the sense of ~\cite[Definition 3]{AS95}.
\end{prop}
\begin{proof}
Properties 2 and 3 of the definition can be checked easily. For property 1, note that we can metrise the topology of $\Omega$ with the following complete metric
\begin{equation}
d_{\Omega}(\omega,\omega'):= \sum_{n=1}^\infty \frac{1}{2^n}\frac{\norm{\omega-\omega'}_n}{1+\norm{\omega-\omega'}_n}\, , \quad \norm{\omega-\omega'}_{n}:= \sup_{t \in [-n ,n]}\norm{\omega(t)-\omega'(t)}_H  \, .
\end{equation}
We then have that
\begin{align}
d_{\Omega}(\theta_t(\omega),\theta_s(\omega')) \leq& d_{\Omega}(\theta_t(\omega),\theta_t(\omega')) + d_{\Omega}(\theta_t(\omega'),\theta_s(\omega')) \\
\leq &  \bra*{1 + 2^{\lceil \abs{t}\rceil-1}} d_{\Omega}(\omega,\omega') + 2 d_{\Omega}(\omega'(\cdot+s),\omega'(\cdot+t)) \, ,  
\end{align}
from which it follows that $(t,\omega)\mapsto \theta_t(\omega)$ is continuous with respect to the product topology on $\Omega \times \R$. It follows that it is $(\mathcal{B}(\R)\otimes\mathcal{F}^0,\mathcal{F}^0)$-measurable and thus property 1 holds true. Finally, property 4 of~\cite[Definition 3]{AS95} follows from the fact that if $W(\cdot)$ is a cylindrical Wiener process on $\ell^2(\Z^d)$, then so is $W(\cdot +t)-W(t)$ for all $t \in \mathbb{R}$. 
\end{proof}
\begin{remark}
We remark here that, as in the case of the canonical MDS associated with a standard Brownian motion on $\R^k$ (see~\cite[Appendix A]{Arn98} for the construction), it is important to work with the uncompleted filtration $\mathcal{F}^0$. While $\theta$ is $(\mathcal{B}(\R)\otimes\mathcal{F}^0,\mathcal{F}^0)$-measurable it is not $(\mathcal{B}(\R)\otimes\mathcal{F},\mathcal{F})$-measurable, where $\mathcal{F}$ is the completion of $\mathcal{F}^0$ with respect to $\mathbb{P}$.
\end{remark}
Before, proceeding, we now introduce some more useful notions following~\cite[Chapter 2]{Arn98}. 
\begin{definition}[Two-parameter filtration]
Let $(\Omega,\mathcal{F},\mathbb{P})$ be a complete probability space. A two-parameter family $(\mathcal{F}_{t,s})_{ s\leq t}$, $s,t \in \R$ of sub-$\sigma$-algebras of $\mathcal{F}$ is said to be a \emph{two-parameter filtration} if
\begin{enumerate}
\item $\mathcal{F}_{t,s} \subset \mathcal{F}_{v,u}$ for $u \leq s \leq t \leq v$
\item $\mathcal{F}_{t+,s}:= \bigcap_{v>t}\mathcal{F}_{v,s}=\mathcal{F}_{t,s}$ and $\mathcal{F}_{t,s-}:= \bigcap_{s>u}\mathcal{F}_{t,u}=\mathcal{F}_{t,s}$  for all $s \leq t$
\item $\mathcal{F}_{t,s}$ contains all the $\mathbb{P}$-null sets of $\mathcal{F}$ for all $s \leq t$.
\end{enumerate}
\label{def:2pf}
\end{definition}
Given the notion of a two-parameter filtration, we can then go on to define the corresponding notion of semimartingales. Note that we adopt the approach of ~\cite{KS97}, where the authors define the underlying dynamical system for $t \in \R$ but perform any stochastic analysis only with nonnegative time.
\begin{definition}[$\mathcal{F}_{t,s}$-semimartingale]
Let $(\mathcal{F}_{t,s})_{ s \leq t}$ be a two-parameter filtration on a complete probability space $(\Omega,\mathcal{F},\mathbb{P})$, in the sense of Definition~\ref{def:2pf}. We say that a measurable function $F: \mathbb{R}_{\geq 0}\times \mathbb{R}_{\geq 0} \times \Omega \to H, (t,s,\omega) \mapsto F(t,s,\omega)$ (with $H$ a separable Hilbert space), jointly continuous in  $s,t$, is a \emph{ $\mathcal{F}_{t,s}$-semimartingale} if for all fixed $s\geq 0$, $(F(s+t,s,\omega))_{t\geq 0}$  is an $\mathcal{F}_{s+t,s}$-semimartingale.
\end{definition}
We also introduce the notion of a filtered dynamical system following~\cite{Arn98}.
\begin{definition}[Filtered dynamical system (FDS)]
Let $(\Omega,\mathcal{F}^0, \mathbb{P}, (\theta_t)_{t \in \R})$ be an MDS, $\mathcal{F}$ be the $\mathbb{P}$-completion of $\mathcal{F}^0$, and $(\mathcal{F}_{t,s})_{ s \leq t}$ be a two-parameter filtration on $(\Omega,\mathcal{F},\mathbb{P})$ in the sense of Definition~\ref{def:2pf}. Then, we say that $(\Omega,\mathcal{F}^0, \mathbb{P}, (\theta_t)_{t \in \R},(\mathcal{F}_{t,s})_{ s \leq t})$ is a \emph{filtered dynamical system (FDS)} if for all $u \in \R$ and $ s \leq t$, we have
\[
\theta_u^{-1} \mathcal{F}_{t,s} = \mathcal{F}_{t+u,s+u} \, .
\]
\end{definition}

Given these definitions, we now elucidate how one can go from the MDS associated to the cylindrical Wiener process to an FDS. We start by considering the $\mathbb{P}$-completion of $\mathcal{F}^0$ which we denote by $\mathcal{F}$ and then define an appropriate two-parameter complete filtration $(\mathcal{F}_{t,s})_{s \leq t}$ with $s,t \in \R$ as is done in~~\cite[Remark 2.3.5]{Arn98},~\cite[Remark 9]{AS95}, or~\cite[page 4]{dSLM75}, i.e. let $\mathcal{F}^-$ and $\mathcal{F}^+$ be sub-$\sigma$-algebras of $\mathcal{F}$ containing all $\mathbb{P}$-null sets of $\mathcal{F}$ such that
\[
\theta_t^{-1}(\mathcal{F}^-) \subset \mathcal{F}^- \quad \textrm{ for all } t\leq 0
\]
and 
\[
\theta_t^{-1}(\mathcal{F}^+) \subset \mathcal{F}^+ \quad \textrm{ for all } t\geq 0 \, .
\]
We then define 
\[
 \mathcal{F}_t:=\theta_t^{-1}(\mathcal{F}^-) \, ,\quad \mathcal{F}_s:=\theta_t^{-1}(\mathcal{F}^+) \, ,\quad \mathcal{F}_{t,s}:= \mathcal{F}_s \cap \mathcal{F}_t  \quad \textrm{ for } s \leq t \, .
\]
One can then check that $\mathcal{F}_{t,s}$, as defined above, is indeed a two-parameter filtration on $(\Omega,\mathcal{F},\mathbb{P})$ in the sense of Definition~\ref{def:2pf} and that $(\Omega,\mathcal{F}^0, \mathbb{P}, (\theta_t)_{t \in \R},(\mathcal{F}_{t,s})_{ s \leq t})$ is an FDS. In the specific case of the MDS associated to a cylindrical Wiener process, we choose $\mathcal{F}^+$ (resp. $\mathcal{F}^-$) to be the $\sigma$-algebra generated by the $\mathbb{P}$-null sets of $\mathcal{F}$ and $\sigma(\omega(t):t \geq 0)$ (resp. $\sigma(\omega(t):t \leq 0)$). It follows that $\mathcal{F}_{t,s}$ is given by the $\mathbb{P}$-completion of $\sigma(\omega(v)-\omega(u): s \leq u\leq v \leq t)$.

We now introduce the another notion that will play an important role in the discussion that follows.
\begin{definition}[Helix and semimartingale helix]
Let $(\Omega,\mathcal{F}^0, \mathbb{P}, (\theta_t)_{t \in \R})$ be an MDS and $(G,\circ)$ a group. A map $F :[0,\infty) \times \Omega\to G$ is called a \emph{(perfect) helix} if
\begin{equation}
F(t+s,\omega)=F(t,\theta_s\omega)\circ F(s,\omega) \, ,
\end{equation}
for all $s,t \in \R$ and all $\omega \in \Omega$. 

Let $(\Omega,\mathcal{F}^0, \mathbb{P}, (\theta_t)_{t \in \R},(\mathcal{F}_{t,s})_{ s \leq t})$ be an FDS. Then, an $(H,+)$-valued helix (for some Hilbert space $H$) is called a \emph{semimartingale helix} if
\begin{equation}
F(t,s,\omega):= F(t,\omega)-F(s,\omega) \, ,
\end{equation} 
is an $\mathcal{F}_{t,s}$-semimartingale.
\end{definition}
Finally, we conclude by presenting the definition of a random dynamical system.
\begin{definition}[Random dynamical system (RDS)]
A \emph{measurable (one-sided) random dynamical system on a measure space $(Y,\mathcal{B})$ over an MDS $(\Omega,\mathcal{F}^0,\mathbb{P},(\theta_t)_{t\in \R})$} is a map $\varphi:[0,\infty) \times Y\times \Omega \to Y$ which satisfies the following properties:
\begin{enumerate}
  \item $\varphi$ is $(\mathcal{B}([0,\infty))\otimes \mathcal{B}\otimes \mathcal{F}^0,\mathcal{B})$-measurable
\item For all $t,s \in [0,\infty)$, $y \in Y$, and $\omega \in \Omega$ we have
\[
\varphi(0,y,\omega)=y \, ,\quad 
\varphi(t+s,y,\omega)= \varphi(t,\varphi(s,y,\omega),\theta_s \omega) \, .
\]

\end{enumerate}
We say that $\varphi$ is a \emph{continuous RDS} if in addition $Y$ is a topological space with Borel sigma algebra $\mathcal{B}$ and, for all $\omega \in \Omega$, $\varphi(\cdot,\cdot,\omega):[0,\infty)\times Y \to Y$  is continuous.
\label{def:rds}
\end{definition}

We now define our candidate for a semimartingale helix, namely
\begin{equation}
W(t,\omega) := \omega(t) \, ,
\end{equation}
for all $t \geq 0$. We then have the following result.
\begin{prop}
Let $(\Omega,\mathcal{F}^0, \mathbb{P}, (\theta_t)_{t \in \R},(\mathcal{F}_{t,s})_{ s \leq t})$ be the FDS associated to a cylindrical Wiener process as constructed above. Then, $W$ is a semimartingale helix with respect to the group $(H,+)$.
\label{prop:helix}
\end{prop}
\begin{proof}
We start by checking that it is a (perfect) $(H,+)$-helix. We have
\begin{align}
W(t,\theta_s \omega) =& (\theta_s \omega)(t) \\
=& \omega(s+t) - \omega(s) \\
=& W(t+s,\omega)-W(s,\omega) \, .
\end{align}
Note that, the way we have defined it, for fixed $s\geq 0$ and all $t \geq 0$, $\mathcal{F}^{s+t,t}$ is just the completed natural filtration of $(W(t,\omega)-W(s,\omega))_{t\geq0}$. Since $W(t,\omega)-W(s,\omega)$ is a continuous martingale with respect to its natural filtration, this completes the proof.
\end{proof}

We define for all $k \in \Z^d$
\begin{equation}
B_t^k := \langle e_k , W(t,\omega)\rangle = \langle e_k , \omega(t) \rangle \, ,
\end{equation}
where $e_k , k \in \Z^d$ are the standard unit vectors in $\ell^2(\Z^d)$. It is clear that $(B_t^k)_{k \in \Z^d}$ are independent standard Brownian motions on $\Omega$.

Consider now the following It\^o equation
\begin{equation}
\partial_t \rho = \Delta \Phi(\rho) -\nabla \cdot\nu(\rho)- \nabla \cdot ( B(\rho)) -f(\rho)-\sqrt{\eps}\nabla \cdot (\sigma(\rho)  \dd\xi) \, ,
\label{eq:DK}
\end{equation}
which was analysed in Sections~\ref{sec:stability} and~\ref{sec:gradient}. The reader should note that for the purposes of this section we set $\lambda \equiv 1$ which ensures that the stability estimate from Theorem~\ref{thm_unique} holds with a deterministic constant. We rely on the solution theory and assumptions presented in Section~\ref{sec:stability}. We will however adapt the notion of solution presented in Section~\ref{sec:gradient} so that we can start at arbitrary initial times $s \geq 0$. Additionally, we will also allow for a small constant $\eps>0$ that models the intensity of the noise. We restate the main existence result in this setting for the readers convenience. $\Delta$ denotes the set $\{(t,s)\in [0,\infty)^2: s \leq t \}$.
\begin{theorem}
Let $(\Omega,\mathcal{F}^0,(\theta_t)_{t\in \R},(\mathcal{F}_{s,t})_{s\leq t})$ be the FDS defined above and let Assumptions \ref{assumption_1}, \ref{assumption_2}, \ref{assumption_3}, and \ref{assdrift} hold true for some $p \in[2,\infty), m \in [1,\infty)$, and $B_t^k, k \in \N$ as defined previously. Fix $s\geq0$ and let $\rho_0 \in L^p(\Omega;L^p(\T^d)) \cap L^{m+p-1}(\Omega;L^1(\T^d))$ be nonnegative. Then, there exists a unique stochastic kinetic solution of~\eqref{eq:DK} starting at time $s\geq 0$ which is $(\mathcal{F}_{t,s})_{s \leq t}$-predictable and almost-surely continuous in $L^1(\T^d)$.

 Additionally, if $\rho(t,s,\rho_{0,1},\omega)$ and $\rho(t,s,\rho_{0,2},\omega)$ are two stochastic kinetic solutions with initial data $\rho_{0,1}$ and $\rho_{0,2}$,  then, there exists a $\mathbb{P}$-null set $N_{s,\rho_{0,1},\rho_{0,2}}$ such that for all $\omega \in N_{s,\rho_{0,1},\rho_{0,2}}^c$ and for all $T \in [s,\infty)$,   we have
\begin{align}
&\sup_{t \in[s,T] }\lVert \rho(t,s,\rho_{0,1},\omega) -\rho(t,s,\rho_{0,2},\omega)\rVert_{L^1(\T^d)} \\ & \leq \exp\left(T(\norm{B}_{\Lip} + \norm{f}_{\Lip}) \right) \lVert \rho_{0,1}(\omega) -\rho_{0,2}(\omega)\rVert_{L^1(\T^d)}  \\
& =: C(T) \lVert \rho_{0,1}(\omega) -\rho_{0,2}(\omega)\rVert_{L^1(\T^d)} \, .
\label{eq:contraction}
\end{align}
\label{thm:sks}
\end{theorem}
\begin{remark}
Note that we have extended the well-posedness result and stability estimate of Section~\ref{sec:stability}, i.e. Theorem~\ref{thm_unique}, to arbitrary time horizons. It is clear that the null set on which the stability estimate~\eqref{eq:contraction} and continuity in the initial time~\eqref{eq:conts} are valid can be made independent of the time horizon $T \geq 0$ by restricting to time horizons $n \in \N$ and using the uniqueness of stochastic kinetic solutions.
\end{remark}

\subsection{Existence of a stochastic flow}
The first step in the construction of a random dynamical system is constructing a stochastic semiflow. For all $p \in [1,\infty]$, we will use $L^p_+(\T^d)$ to denote the space of nonnegative $L^p(\T^d)$ functions. For every fixed $s \in [0,\infty),\rho_0 \in L^p_+(\T^d)$, we can use Theorem~\ref{thm:sks} to obtain an $\mathcal{F}_{t,s}$-predictable stochastic kinetic solution which we denote by $\rho=\rho(t,s,\rho_0,\omega)\in L^1(\Omega; C([s,\infty);L^1_+(\T^d)))$ defined on the probability space $(\Omega,\mathcal{F}, \mathbb{P})$. Thus, for all $s \geq 0$ and $\rho_0 \in L^p_+(\T^d)$, there exists a $\mathbb{P}$-null set $N_{s,\rho_0}$, such that for all $\omega \in N_{s,\rho_0}^c$ the map $[s,\infty) \ni t \mapsto \rho(t,s,\rho_0,\omega) \in L^1_+(\T^d)$ is continuous.

Similarly, for every fixed $s \in [0,\infty),\rho_0 \in \rho_0 \in L^p(\Omega;L^p_+(\T^d)) \cap L^{m+p-1}(\Omega;L^1_+(\T^d))$, we can use Theorem~\ref{thm:sks} to obtain an $\mathcal{F}_{t,s}$-predictable stochastic kinetic solution which we denote by $\bar \rho=\bar \rho(t,s,\rho_0,\omega)\in L^1(\Omega; C([s,\infty);L^1_+(\T^d)))$ defined on the probability space $(\Omega,\mathcal{F}, \mathbb{P})$. Note that for fixed $t,s \in \Delta$ and $\omega \in \Omega$, $\bar \rho$ depends on the entire random variable $\rho_0$ and not just on its evaluation $\rho_0(\omega)$. Again, for all $s \geq 0$ and $\rho_0 \in L^p(\Omega;L^p_+(\T^d))$, there exists a $\mathbb{P}$-null set $N_{s,\rho_0}$, such that for all $\omega \in N_{s,\rho_0}^c$ the map $[s,\infty) \ni t \mapsto \bar \rho(t,s,\rho_0,\omega) \in L^1_+(\T^d)$ is continuous. We remark here that the difference between $\bar \rho$ and $\rho$ as they are defined is slightly artificial. The only reason we make a distinction between them is that we want to think about $\rho$ as the solution map for deterministic initial data $\rho_0\in L^p_+(\T^d)$ and $\bar \rho $ as the solution map for random initial data  $\rho_0 \in L^p(\Omega;L^p_+(\T^d)) \cap L^{m+p-1}(\Omega;L^1_+(\T^d))$. Clearly, for all $\rho_0 \in L^p_+(\T^d)$ and $t,s \in \Delta$, $\rho(t,s,\rho_0,\omega)=\bar \rho(t,s,\rho_0,\omega)$.

We start by establishing a very crude version of the semiflow property.
\begin{lemma}[Very crude semiflow property]
Fix $s\geq0$ and let $\rho$ and $\bar \rho$ be as defined above. Then, for any $\rho_0 \in L^p_+(\T^d)$ and $s_1 \in[s,\infty)$, there exists a $\mathbb{P}$-null set $N_{s,s_1,\rho_0}$, such that for all $t \in [s_1,\infty) $ and $\omega \in N_{s,s_1,\rho_0}^c $, we have
\begin{equation}
\rho(t,s,\rho_0,\omega)= \bar \rho(t,s_1,\rho(s_1,s,\rho_0,\cdot),\omega) \, .
\label{eq:vcrude}
\end{equation}
\label{lem:vcrude}
\end{lemma}
\begin{proof}
Consider the equation satisfied by $\rho(t,s,\rho_0,\omega)$ for some  $\rho_0 \in L^p_+(\T^d)$, i.e.  for any $\varphi \in C_c^\infty(\T^d \times (0,\infty))$, we have
\begin{align*}
		&\int_{\R}\int_{\T^d} \chi(x,\xi,t) \varphi(x,\xi) \dx{x} \dx{\xi} =  \int_{\R} \int_{\T^d} \bar{\chi}(\rho_0) \varphi(x,\xi) \dx{x}\dx{\xi}\\& -\int_s^t \int_{\T^d}\Phi'(\rho(r,s,\rho_0,\omega)) \nabla \rho(r,s,\rho_0,\omega) \cdot \nabla \varphi(x,\rho) \dx{x} \dx{r} \\&+\frac{F_1}{2}\int_s^t \int_{\T^d} [\sigma'(\rho(r,s,\rho_0,\omega))]^2\nabla\rho(r,s,\rho_0,\omega) \cdot \nabla \varphi(x,\rho(r,s,\rho_0,\omega)) \dx{x}\dx{r} \\
		&- \int_s^t \int_{\R} \int_{\T^d}(\partial_{\xi} \varphi(x,\xi)) \dx{q} +\frac{F_3}{2}\int_s^t \int_{\T^d} \sigma^2(\rho(r,s,\rho_0,\omega)) (\partial_\xi \varphi)(x,\rho(r,s,\rho_0,\omega)) \dx{x}\dx{r} \\
		&- \int_s^t\int_{\T^d} \varphi(x,\rho(r,s,\rho_0,\omega))\nabla \cdot \nu(\rho(r,s,\rho_0,\omega)) \dx{x}\dx{r} \\
		&- \int_s^t\int_{\T^d} \varphi(x,\rho(r,s,\rho_0,\omega))(\nabla \cdot  B(\rho(r,s,\rho_0,\omega)) + f(\rho(r,s,\rho_0,\omega))) \dx{x}\dx{r}\\
		&- \int_{\T^d}\int_s^t \varphi(x,\rho(r,s,\rho_0,\omega)) \nabla \cdot(\sigma(\rho(r,s,\rho_0,\omega)) \dx{\xi}^F) \dx{x}  \, .
	\end{align*}
Given $s_1\in[s,\infty)$, we split the integrals from $s$ to $t$  into the sum of integrals from $s$ to $s_1$ and $s_1$ to $t$, to obtain
\allowdisplaybreaks
\begin{align*}
		&\int_{\R}\int_{\T^d} \chi(x,\xi,t) \varphi(x,\xi) \dx{x} \dx{\xi} =  \int_{\R} \int_{\T^d} \bar{\chi}(\rho(s_1,s,\rho_0,\omega)) \varphi(x,\xi) \dx{x}\dx{\xi}\\
		&-\int_{s_1}^t \int_{\T^d}\Phi'(\rho(r,s,\rho_0,\omega)) \nabla \rho(r,s,\rho_0,\omega) \cdot \nabla \varphi(x,\rho) \dx{x} \dx{r} \\&+\frac{F_1}{2}\int_{s_1}^t \int_{\T^d} [\sigma'(\rho(r,s,\rho_0,\omega))]^2\nabla\rho(r,s,\rho_0,\omega) \cdot \nabla \varphi(x,\rho(r,s,\rho_0,\omega)) \dx{x}\dx{r} \\
		&- \int_{s_1}^t \int_{\R} \int_{\T^d}(\partial_{\xi} \varphi(x,\xi)) \dx{q} +\frac{F_3}{2}\int_{s_1}^t \int_{\T^d} \sigma^2(\rho(r,s,\rho_0,\omega)) (\partial_\xi \varphi)(x,\rho(r,s,\rho_0,\omega)) \dx{x}\dx{r} \\
		&- \int_{s_1}^t\int_{\T^d} \varphi(x,\rho(r,s,\rho_0,\omega))\nabla \cdot \nu(\rho(r,s,\rho_0,\omega)) \dx{x}\dx{r} \\
		&- \int_{s_1}^t\int_{\T^d} \varphi(x,\rho(r,s,\rho_0,\omega))(\nabla \cdot  B(\rho(r,s,\rho_0,\omega)) + f(\rho(r,s,\rho_0,\omega))) \dx{x}\dx{r}\\
		&- \int_{\T^d}\int_{s_1}^t \varphi(x,\rho(r,s,\rho_0,\omega)) \nabla \cdot(\sigma(\rho(r,s,\rho_0,\omega)) \dx{\xi}^F) \dx{x}  \, .
\end{align*}
It follows that the curve $[s_1,\infty) \ni t \mapsto \rho(t,s,\rho_0,\cdot)$ is a stochastic kinetic solution of~\eqref{eq:DK} with initial datum $\rho (s_1,s,\rho_0,\cdot)$. Applying the stability estimate~\eqref{eq:contraction}, the result follows.
\end{proof}
We now use the above crude version of the semiflow property along with gradient estimates from Section~\ref{sec:gradient} to obtain continuity of the solution map with respect to the initial time.

\begin{prop}
Let $\xi^F$, $\Phi$, $\sigma$, $\nu$, $\lambda$, $f$, and $B$ satisfy Assumptions~\ref{assumption_1}, ~\ref{assumption_2}, \ref{assumption_3}, \ref{assume_B}, and \ref{assume_grad}, let $\rho_{0}\in H^1(\T^d)$ be nonnegative, and $\rho(t,s,\rho_0,\omega)$ be the unique stochastic kinetic solution of \eqref{eq:DK} with initial datum $\rho_{0}$. Then, there exist $\eta ,\eps_0\in (0,1) $, such that for all $\eps < \eps_0$, there exists a $\mathbb{P}$-null set $N_{\rho_0}$, such that for all $T \geq 0$, there exists a random variable $X_{\rho_0,T}\in L^p(\Omega)$, $\forall p \ge 1$, such that for all $\omega \in N_{\rho_0}^c$, $s_1,s_2 \in [0,T]\cap \mathbb{Q}$, and $t \in [s_1,T]\cap [s_2,T]$, we have
\begin{equation}
\norm{\rho(t,s_1,\rho_0,\omega)- \rho(t,s_2,\rho_0,\omega)}_{L^1(\T^d)} \leq X_{\rho_0,T} (\omega) \abs{s_1-s_2}^\eta \, .
\label{eq:conts}
\end{equation}
\label{prop:conts}
\end{prop}

\begin{proof}
Due to the pathwise stability estimate \eqref{eq:pathwise_stability} and the (crude) semiflow property proved in Lemma \ref{lem:vcrude}, we have that, for all $s\le s'$, $\mathbb{P}$-a.s., for all $t$, 
\begin{align}
&\|\rho(t,s,\rho_{0},\omega)-\rho(t,s',\rho_{0},\omega)\|_{C([0,T];L^{1}(\T^d))} \\  = &\sup_{t\in[0,T]}\|\bar \rho(t,s',\rho(s',s,\rho_0,\cdot),\omega)-\rho(t,s',\rho_{0},\omega)\|_{L^1(\T^d)}\\
   \le & \|\rho(s',s,\rho_{0},\omega)-\rho_{0}\|_{L^1(\T^d)} \, ,
\end{align}
where we extend $\rho(t,s,\rho_0,\omega)$ by $\rho_0$ on $[0,s)$. For every $\delta\in(0,\infty)$, let $\kappa^{\delta}\colon\TT^{d}\rightarrow[0,1]$ be a smooth convolution kernel at scale $\delta$ on $\TT^{d}$ and let $\rho^{\delta}=(\rho*\kappa^{\delta})$. We then have, $\mathbb{P}$-a.s., for all $s\le s'$, 
\begin{align*}
& \|\rho(s',s,\rho_0,\omega)-\rho_{0}\|_{L^{1}(\T^d)}  \le \|\rho^{\delta}(s',s,\rho_0,\omega)-\rho_{0}^{\delta}\|_{L^{1}(\T^d)}
\\& \quad +\|\rho(s',s,\rho_{0},\omega)-\rho^{\delta}(s',s,\rho_{0},\omega)\|_{L^{1}(\T^d)}+\|\rho_{0}-\rho_{0}^{\delta}\|_{L^{1}(\T^d)}.
\end{align*}
It follows from the standard convolution estimate that, $\mathbb{P}$-a.s., for all $s\le s'$,
\[
\|\rho(s',s,\rho_{0},\omega)-\rho^{\delta}(s',s,\rho_{0},\omega)\|_{L^{1}(\T^d)}\leq C\delta\norm{\nabla\rho(s',s,\rho_{0},\omega)}_{L^{2}(\TT^{d})}.
\]
We thus have
\begin{align*}
&\E\|\rho(t,s,\rho_{0},\omega)-\rho(t,s',\rho_{0},\omega)\|_{C([0,T];L^{1}(\T^d))}^{p} \\
& \le\E \|\rho(s',s,\rho_{0},\omega)-\rho_{0}\|_{L^1(\T^d)}^{p}\\
 & \lesssim \E\|\rho^{\delta}(s',s,\rho_{0},\omega)-\rho_{0}^{\delta}\|_{L^{1}(\T^d)}^{p}+\delta^{p}(\E\norm{\nabla\rho(s',s,\rho_{0},\omega)}_{L^{2}(\TT^{d})}^{p}+\E\norm{\nabla\rho_{0}}_{L^{2}(\TT^{d})}^{p}).
\end{align*}
We next notice that
\begin{align}
&\E\|\rho^{\delta}(s',s,\rho_{0},\omega)-\rho_{0}^{\delta}\|_{L^{1}_x}^{p}\\
\lesssim &  \E\left\|\int_{s}^{s'}\kappa^{\delta}*(\nabla\cdot\Phi'(\rho)\nabla\rho)\right\|_{L^{1}(\T^d)}^{p}+\E\left\|\int_{s}^{s'}\kappa^{\delta}*(\nabla\cdot\nu(\rho))\right\|_{L^{1}(\T^d)}^{p}\label{pt_2-1-1}\\
 & +\E\left \|\int_{s}^{s'}\kappa^{\delta}*(\nabla\cdot(B(\rho)))\right\|_{L^{1}(\T^d)}^{p} + \E\left \|\int_{s}^{s'}\kappa^{\delta}*f(\rho)\right\|_{L^{1}(\T^d)}^{p}
 \\&+\E\left\|\int_{s}^{s'}\sqrt{\eps}\kappa^{\delta}*(\nabla\cdot(\sigma(\rho)\dd\xi^F)\right\|_{L^{1}(\T^d)}^{p}.\nonumber 
\end{align}
So, by It\^o's formula and Sobolev embedding, we have that
\begin{align*}
&\E\left\|\int_{s}^{s'}\sqrt{\eps}\kappa^{\delta}*(\nabla\cdot(\sigma(\rho)\dd\xi^F)\right\|_{L^{1}(\T^d)}^{p} 
\\ & \le \E\left\|\sqrt{\eps}\int_{s}^{s'}\kappa^{\delta}*(\nabla\cdot(\sigma(\rho)\dd\xi^F)\right\|_{L^{2}(\T^d)}^{p}\\
  & \le \eps^{\frac p2}\sum_{k}\E\left\|\mathbf{1}_{[s,s']}\kappa^{\delta}*(\nabla\cdot(\sigma(\rho)e_{k}))\right\|_{L^{2}([0,T];L^{2}(\T^d))}^{p}\\
  & \lesssim \eps^{\frac p2}|s'-s|^{p/2}\sum_{k}\E\|\kappa^{\delta}*(\nabla\cdot(\sigma(\rho)e_{k}))\|_{L^{\infty}([0,T];L^{2}(\T^d))}^{p}\\
  & \lesssim \eps^{\frac p2} F_1^{\frac p2} |s'-s|^{p/2}\delta^{-p(1+\frac{d}{2})}\E\|\sigma(\rho)\|_{L^{\infty}([0,T];L^{1}(\T^d))}^{p}.
\end{align*} 
We thus obtain that
\begin{align}
&\E\|\rho^{\delta}(s',s,\rho_{0},\omega)-\rho_{0}^{\delta}\|_{L^{1}_{x}}^{p}\\
& \lesssim   \delta^{-2}|s'-s|^{p}\E\|\Phi(\rho)\|_{L^{\infty}([0,T];L^{1}(\T^d))}^{p}\\& \quad +\delta^{-1}|s'-s|^{p}\E(\|\nu(\rho)\|_{L^{\infty}([0,T];L^{1}(\T^d))}^{p}+\|B(\rho)\|_{L^{\infty}([0,T];L^{1}(\T^d))}^{p})\label{pt_2-1-1-1}\\
& \quad +|s'-s|^{p}\E\|f(\rho)\|_{L^{\infty}([0,T];L^{1}(\T^d))}^{p} \\
 & \quad +|s'-s|^{p/2}\delta^{-p(1+\frac{d}{2})}\E\|\sigma(\rho)\|_{L^{\infty}([0,T];L^{1}(\T^d))}^{p}.\nonumber 
\end{align}
Using the fact that all coefficients have sublinear growth, since $\sigma,\Phi\in\C_{b}^{2}([0,\infty))$, all expectations on the right hand side are finite, and we obtain that
\begin{align*}
\E\|\rho^{\delta}(s',s,\rho_{0},\omega)-\rho_{0}^{\delta}\|_{L^{1}(\T^d)}^{p}\lesssim & |s'-s|^{p/2}\delta^{-p(1+\frac{d}{2})}.
\end{align*}
As a result we have shown that
\begin{align*}
&\E\|\rho(t,s,\rho_{0},\omega)-\rho(t,s',\rho_{0},\omega)\|_{C([0,T];L^1(\T^d))}^{p} \\
& \lesssim |s'-s|^{p/2}\delta^{-p(1+\frac{d}{2})}+\delta^{p}(\E\norm{\nabla\rho(s',s,\rho_{0},\omega)}_{L^{2}(\TT^{d})}^{p}+\E\norm{\nabla\rho_{0}}_{L^{2}(\TT^{d})}^{p}). 
\end{align*}
Equilibrating via $|s'-s|^{p/2}\delta^{-p(1+\frac{d}{2})}=\delta^{p}$ gives
\[
|s'-s|^{\frac{p}{d+4}}=:\delta^{p}
\]
and thus
\begin{align*}
&\E\|\rho(t,s,\rho_{0},\omega)-\rho(t,s',\rho_{0},\omega)\|_{C([0,T];L^1(\T^d))}^{p} \\
& \lesssim|s'-s|^{\frac{p}{d+4}}(\E\norm{\nabla\rho(s',s,\rho_{0},\omega)}_{L^{2}(\TT^{d})}^{p}+\E\norm{\nabla\rho_{0}}_{L^{2}(\TT^{d})}^{p}+1).
\end{align*}
We note that by Assumption \ref{assume_grad} and Proposition \ref{prop_grad_est}, the right hand side is finite for all $p\ge1$ if $\eps$ is sufficiently small\footnote{This is equivalent to $F_1$ being sufficiently small}. Since $p>d+4$, we can apply Kolmogorov's continuity theorem, which implies that there is an $\eta>0$, a random variable $X_{\rho_0,T}$ with $X_{\rho_0,T}\in L^{p}(\Omega)$, and a modification of $(\omega,s)\mapsto\rho(t,s,\rho_{0},\omega)$ (which we continue to denote by $\rho(t,s,\rho_0,\omega)$), so that, for all $s\le s' \leq T $, $\mathbb{P}$-a.s.,
\begin{align*}
\|\rho(t,s,\rho_{0},\omega)-\rho(t,s',\rho_{0},\omega)\|_{C([0,T];L_1(\T^d))} & \lesssim X_{\rho_0,T}(\omega)|s'-s|^{\eta}.
\end{align*}
Since the modification agrees with $\rho$ on a null set dependent on $s$ and $\rho_0$, we can restrict to a countable subset $s \in \mathbb{Q}\cap [0,T]$ to obtain the result of the proposition.
\end{proof}

We will now show that there exists a continuous modification in $t,s$ of the solution map $\rho$. 
\begin{lemma}[Continuous modification]
Assume Assumptions~\ref{assumption_1}, \ref{assumption_2}, \ref{assumption_3}, \ref{assdrift}, and \ref{assume_grad} are satisfied and denote by $\rho$ a stochastic kinetic solution of~\eqref{eq:DK}. Let $X:= C(L^1_+(\T^d);L^1_+(\T^d))$ be equipped with the compact-open topology with $\mathcal{B}(X)$ the corresponding Borel sigma algebra. If $\eps < \eps_0$, then there exists a $\mathbb{P}$-null set $N$ and a map $\Delta \times L^1_+(\T^d) \times \Omega \ni ((t,s),\rho_0,\omega) \mapsto \check \rho(t,s,\rho_0,\omega) \in L^1_+(\T^d)$ such that for all $\omega \in N^c$, $\Delta \ni (t,s)\mapsto \check \rho(t,s,\cdot,\omega) \in X$ and $\Delta \times L^1_+(\T^d) \ni ((t,s),\rho_0)\mapsto \check \rho(t,s,\rho_0,\omega) \in L^1_+(\T^d)$  are continuous. Furthermore, for all $\omega \in N^c$, $\rho_0 \in L^1_+(\T^d)$, $s\in [0,\infty)$ and  $T \in [s,\infty)$, we have
\[
\sup_{t \in [s,T] }\norm{\check \rho(t,s,\rho_{0,1},\omega)- \check \rho(t,s,\rho_{0,2},\omega)}_{L^1(\T^d)} \leq C(T) \norm{\rho_{0,1}-\rho_{0,2}}_{L^1(\T^d)} \, .
\]

 Additionally, for all $\rho_0 \in L^p_+(\T^d)$ (with $p$ as in Assumptions~\ref{assnoise} to~\ref{assdrift}) and $s \in [0,\infty)$, there exists a $\mathbb{P}$-null set $N_{s,\rho_0}$ such that for all $\omega \in N_{s,\rho_0}^c$ and $t \in [s,\infty)$, we have
\begin{equation}
\rho(t,s,\rho_0,\omega)=\check \rho(t,s,\rho_0,\omega) \, .
\label{eq:indistinguish}
\end{equation} 
Finally, the map $\Delta \times \Omega\ni ((t,s),\omega)\mapsto \check\rho(t,s,\cdot,\omega) \in X$ is $(\mathcal{B}(\Delta) \otimes \mathcal{F}, \mathcal{B}(X))$-measurable and the map $\Delta\times L^1_+(\T^d) \times \Omega\ni ((t,s),\rho_0,\omega)\mapsto \check\rho(t,s,\rho_0,\omega) \in L^1_+(\T^d)$ is $(\mathcal{B}(\Delta) \otimes \mathcal{B}(L^1_+(\T^d))\otimes \mathcal{F}, \mathcal{B}(L^1_+(\T^d)))$-measurable.
\label{lem:cmod}
\end{lemma}
\begin{proof}
The idea behind the proof of this lemma is to obtain perfect versions (i.e. with a uniform null set) of the estimates~\eqref{eq:contraction} and~\eqref{eq:conts} of Theorem~\ref{thm:sks} and Proposition~\ref{prop:conts}, respectively, for an appropriately constructed modification of $\rho$. We know already that for every $s \in [0,\infty)$ and $\rho_0 \in L^p_+(\T^d)$, there exists a $\mathbb{P}$-null set $N_{s,\rho_0}$, such that for all $\omega \in N_{s,\rho_0}^c$ the map $[s,\infty) \ni t \mapsto \rho(t,s,\rho_0,\omega) \in L^1_+(\T^d)$ is continuous.

We now fix some $\rho_0 \in L^p_+(\T^d)$ and we restrict to countably many $s \in \mathbb{Q}\cap [0,\infty)$ 
. From the results of Proposition~\ref{prop:conts}, we can then find a single $\mathbb{P}$-null set $N_{\rho_0}$ and a  random variable $X_{\rho_0,T} \in \Leb^2(\Omega)$, such that, for all $\omega \in N_{\rho_0}^c$, $X_{\rho_0,T}(\omega) < \infty$ . Furthermore, for all $\omega \in N_{\rho_0}^c$, $T \geq 0$, $s_1,s_2 \in \mathbb{Q} \in [0,T]$, and $t \in [s_1,T] \cap [s_2,T]$, we have
\begin{equation}
\norm{\rho(t,s_1,\rho_{0},\omega)- \rho(t,s_2,\rho_{0},\omega)}_{L^1(\T^d)} \leq X_{\rho_{0},T} (\omega) \abs{s_1-s_2}^\eta \, .
\end{equation}
 Thus, on $N_{\rho_0}^c$, considered as a map from $\mathbb{Q}\cap [0,\infty)  \to C([0,\infty);L^1_+(\T^d))$ (by extending trivially by $\rho_0$ on $[0,s)$), equipped with the compact-open topology,  $\rho(\cdot,s,\rho_0,\omega)$ is uniformly continuous. We can thus extend it uniquely by uniform continuity on $N_{\rho_0}^c$ to some $\check{\rho}(\cdot,s,\rho_0,\omega)$ defined on $[0,\infty)$ which, when considered as a map from $[0,\infty) \to C([0,\infty);L^1_+(\T^d))$, is continuous on $N_{\rho_0}^c$.  It follows that for each $\rho_0 \in L^p_+(\T^d)$ there exists a $\mathbb{P}$-null set $N_{\rho_0}^c$, such that for all $\omega \in N^c_{\rho_0}$, $\check \rho(\cdot,\cdot,\rho_0,\omega) \in C(\Delta;L^1_+(\T^d))$.

 We now fix a set $\cC$ which is countable and dense in $L^p_+(\T^d)$ (and thus in $L^1_+(\T^d)$). Using the result of Theorem~\ref{thm:sks}, we can find a $\mathbb{P}$-null set $N$ such that for all $\omega \in N^c$, $\rho_{0,1},\rho_{0,2} \in \cC$, and $s \in \mathbb{Q}\cap [0,\infty)$,  $\check \rho(\cdot,\cdot,\rho_0,\omega) \in C(\Delta;L^1_+(\T^d)) $, and for all $T \in [s,\infty)$, we have
 \begin{equation}
\sup_{t \in [s,T] }\norm{\check \rho(t,s,\rho_{0,1},\omega)- \check \rho(t,s,\rho_{0,2},\omega)}_{L^1(\T^d)} \leq C(T) \norm{\rho_{0,1}-\rho_{0,2}}_{L^1(\T^d)} \, .
\label{eq:est}
 \end{equation} 
Since $\check \rho(\cdot,\cdot,\rho_0,\omega) \in C(\Delta;L^1_+(\T^d)) $, we can extend the above estimate to all $s \in [0,\infty)$. Due to~\eqref{eq:est}, on $N^c$, considered as a map from $\cC \to C(\Delta;L^1_+(\T^d))$, equipped with the compact-open topology,  $\check \rho(\cdot,\cdot,\rho_0,\omega)$ is uniformly continuous. We can thus extend it uniquely by uniform continuity on $N^c$ to some $\check{\rho}(\cdot,\cdot,\rho_0,\omega)$ (we do not rename the extension for the sake of notational convenience) defined on $L^1_+(\T^d)$ which, for all $\omega \in N^c$, when considered as a map from $L^1_+(\T^d) \to C(\Delta;L^1_+(\T^d))$, is continuous and satisfies the estimate~\eqref{eq:est}.

It is also clear that $\check{\rho}(t,s,\rho_0,\omega)$ is continuous on $N^c$ as a map from $\Delta \times L^1_+(\T^d) \to L^1_+(\T^d)$. Then, it is a standard fact (see, for example,~\cite[Theorem 46.11]{Mun00}) that, for all $\omega \in N^c$, $\check{\rho}(t,s,\rho_0,\omega)$ is continuous as a map from $\Delta \to X$.

What remains to be shown is that $\check{\rho}(t,s,\rho_0,\omega)$ defines a stochastic kinetic solution of~\eqref{eq:DK} for $\rho_0 \in L^p_+(\T^d)$. We first notice that, by construction, $\check \rho$ agrees with $\rho$ for all $s \in \mathbb{Q}\cap [0,\infty)$, $t\in [s,\infty)$, $\rho_0 \in \cC$, and $\omega \in N^c$. Now, given arbitrary $s\in [0,T]$ (for some $T \geq 0$), and $\rho_0 \in L^p_+(\T^d)$, we can find sequences $\mathbb{Q}\cap [0,T ] \ni s_n \to s$ and $\cC \ni\rho_{0,n} \to \rho_0 $ in $L^p_+(\T^d)$.  In addition, we can find a $\mathbb{P}$-null set $N_{s,\rho_0}$, such that, for all $\omega \in N_{s,\rho_0}^c$, both the stability estimate~\eqref{eq:contraction} and~\eqref{eq:conts} hold true, and, thus, 
\begin{align}
&\norm{\check \rho(t,s_n,\rho_{0,n},\omega)- \rho(t,s,\rho_{0},\omega)}_{L^1(\T^d)}  \\
& \leq \norm{ \rho(t,s_n,\rho_{0,n},\omega)- \rho(t,s_n,\rho_{0},\omega)}_{L^1(\T^d)} + \norm{ \rho(t,s_n,\rho_{0},\omega)- \rho(t,s,\rho_{0},\omega)}_{L^1(\T^d)} \\
& \leq C(T)\norm{\rho_{0,n}-\rho_0}_{L^1(\T^d)} + X_{\rho_{0},T}(\omega)\abs{s_n-s}^\eta \, .
\end{align}
This completes the proof of~\eqref{eq:indistinguish}.

We now know that, for fixed $((t,s),\rho_0) \in \Delta \times L^1_+(\T^d)$, $\check \rho(t,s,\rho_0,\omega)$ is  $(\mathcal{F},\mathcal{B}(L^1_+(\T^d)))$-measurable as a map from $\Omega \to L^1_+(\T^d)$. This follows from the fact that for all $\rho_0 \in L^p_+(\T^d)$, $\check \rho$ agrees with $\rho$ almost surely, the measurability of $\rho$, and the fact that measurability is closed under almost sure convergence. Since the cylindrical $\sigma$-algebra is the same as the Borel $\sigma$-algebra of the compact-open topology, we have that for all $(t,s) \in \Delta$, $\check \rho(t,s,\rho_0,\omega)$ is  $(\mathcal{F},\mathcal{B}(X))$-measurable as a map from $\Omega \to X$.

We also know that, for all $\omega \in N^c$, $\check{\rho}(t,s,\rho_0,\omega)$ is continuous when considered as a map from $\Delta \to X$. It follows, by an approximation argument, that $\check \rho(t,s,\rho_0,\omega)$ is $(\mathcal{B}(\Delta) \otimes \mathcal{F},\mathcal{B}(X))$-measurable.
A similar argument tells us that $\check \rho(t,s,\rho_0,\omega)$ is $(\mathcal{B}(\Delta) \otimes\mathcal{B}(L^1_+(\T^d)) \otimes \mathcal{F},\mathcal{B}(L^1_+(\T^d)))$-measurable.
\end{proof}
We will now upgrade the very crude semiflow property of Lemma~\ref{lem:vcrude} to a perfect one (with a uniform null set) for the continuous modification $\check \rho(t,s,\rho_0,\omega)$. 

\begin{lemma}[(Perfect) semiflow property]
Let $\check \rho$ be as constructed in Lemma~\ref{lem:cmod}. Then, there exists a $\mathbb{P}$-null set $N$, such that for all $\omega \in N^c$, for all $\rho_0 \in L^1_+(\T^d)$, and all $(t,s), (t,s_1) \in \Delta$ with $s \leq s_1$, we have
\[
\check\rho(t,s,\rho_0,\omega)= \check\rho(t,s_1,\check\rho(s_1,s,\rho_0,\omega),\omega) \, .
\]
\label{lem:semiflow}
\end{lemma}
\begin{proof}
We know from~\eqref{eq:vcrude} that for all $s,s_1\in [0,\infty), s\leq s_1$, there exists a $\mathbb{P}$-null set $N_{s,s_1,\rho_0}$ such that for all $\omega \in N_{s,s_1,\rho_0}^c$ and all $t \in [s_1,\infty)$, we have
\begin{equation*}
\rho(t,s,\rho_0,\omega)=\bar \rho(t,s_1, \rho(s_1,s,\rho_0,\cdot),\omega) \, .
\end{equation*}
We next show that, for all $\rho^*\in L^p(\Omega;L^p_+(\T^d))\cap L^{m+p-1}(\Omega;L^1_+(\T^d)) $ and all $(t,s), (t,s_1) \in \Delta$ with $s \leq s_1$, $\mathbb{P}$- almost surely,
\begin{equation*}
\bar \rho(t,s, \rho^*(\cdot),\omega) =  \check \rho(t,s, \rho^*(\omega),\omega) \, .
\end{equation*}
  It is easy to check that for any simple $\rho^{*,N}$ of the form
\begin{equation}
\rho^{*,N} =\sum_{i=1}^N \mathbf{1}_{A_i} g_i \, 
\end{equation}
for some partition $\{A_i\}_{i=1,\dots,N} \subseteq \mathcal{F}$ of $\Omega$ and $g_i \in L^p_+(\T^d)$,  there exists a $\mathbb{P}$-null set $N_{s,\rho^{*,N}}$ such that for all $t \in [s,\infty)$ and $\omega \in N_{s,\rho^N}^c$, we have 
	\begin{align*}
    \bar \rho(t,s, \rho^{*,N}(\cdot),\omega) =&\sum_{i=1}^N\mathbf{1}_{A_i}\rho(t,s,g_i,\omega)  \\
    =&\rho(t,s,\rho^{*,N}(\omega),\omega) \, ,
	\end{align*}
	where for the last equality we use~\eqref{eq:contraction}. Now note that any $\rho^* \in L^p(\Omega;L^p_+(\T^d))\cap L^{m+p-1}(\Omega;L^1_+(\T^d))$ can be approximated by a sequence of such simple $(\rho^{*,N})_{N \in \N}$ such that $\rho^{*,N}$ converges to $\rho^*$  in $L^1(\Omega;L^1_+(\T^d))$ and almost surely. Taking the expectation in~\eqref{eq:contraction}, it follows that, for all $T \in [s,\infty)$, 
	\begin{align*}
	&\lim_{N \to \infty }\mathbb{E} \sup_{t \in [s,T ]}\left[\lVert\bar \rho(t,s, \rho^{*,N}(\cdot),\omega) - \bar \rho(t,s, \rho^*(\cdot),\omega)\rVert_{L^1(\T^d)}\right] \\\leq& C(T) \lim_{N \to \infty }\mathbb{E} \left[\lVert\rho^{*,N} -\rho^* \rVert_{L^1(\T^d)}(\omega)\right] =0 \, .
	\end{align*}
	Furthermore, from the results of Lemma~\ref{lem:cmod}, there exists a $\mathbb{P}$-null set $N_{s,\rho^*,\rho^{*,N}}$ such that for all $\omega \in N_{s,\rho^*,\rho^{*,N}}^c$ we have that
	\begin{align*}
	\lim_{N \to \infty} \rho(t,s,\rho^{*,N}(\omega),\omega) =& \lim_{N \to \infty} \check\rho(t,s,\rho^{*,N}(\omega),\omega)\\
	=&\check \rho(t,s,\rho^*(\omega),\omega) \, .
	\end{align*}
	It follows that for any $s \in [0,\infty)$ and $\rho^* \in L^p(\Omega; L^p_+(\T^d))\cap L^{m+p-1}(\Omega;L^1_+(\T^d))$ there exists a $\mathbb{P}$-null set $N_{s,\rho^*}$ such that for all $\omega \in N_{s,\rho^*}^c$ and $t \in [s,\infty)$, we have
	\begin{equation*}
     \bar \rho(t,s, \rho^*(\cdot),\omega)=   \check \rho(t,s, \rho^*(\omega),\omega) \, .
	\end{equation*}
    Choosing $\rho^*(\omega)= \check \rho(s_1,s,\rho_0,\omega)$ and using~\eqref{eq:vcrude} along with the results of Lemma~\ref{lem:cmod}, it follows that for all $(t,s),(t,s_1) \in \Delta$ such that $s \leq s_1$, and $\rho_0 \in L^p_+(\T^d)$ there exists a $\mathbb{P}$-null set $N_{s,s_1,\rho_0}$ such that for all $\omega \in N_{s,s_1,t,\rho_0}^c $, we have 
    \begin{equation*}
    	\check \rho(t,s,\rho_0,\omega) =  \check \rho(t,s_1,\check \rho(s_1,s,\rho_0,\omega),\omega) \, .
    \end{equation*}
    By the joint continuity of $\check \rho(t,s,\rho_0,\omega)$ in $t,s,\rho_0$, we can make the $\mathbb{P}$-null set uniform over $s,s_1,\rho_0$ and extend the semiflow property to all $\rho_0 \in L^1_+(\T^d)$. This completes the proof.
\end{proof}

  \subsection{The cocycle property}
  The final step before we can speak about the generation of a random dynamical system is the so-called cocycle property. We start by establishing a crude version of the property in the following lemma.
  \begin{lemma}[Crude cocycle property]
  Assume Assumptions \ref{assumption_1}, \ref{assumption_2}, \ref{assumption_3}, \ref{assdrift}, and \ref{assume_grad} are satisfied and let $\check \rho$ be as constructed in Lemma~\ref{lem:cmod}. Then, for all $s \in [0,\infty)$ there exists a $\mathbb{P}$-null set $N_s$ such that for all $\omega \in N_s^c$, for all $\rho_0 \in L^1_+(\T^d)$, and for all $t \in [0,\infty)$, we have
	\begin{equation}
	\check \rho(t+s,s,\rho_0,\omega) = \check\rho(t,0, \rho_0,\theta_s\omega) \, .
  \label{eq:crudecocyle}
	\end{equation}
	\label{lem:crudecocycle}
  \end{lemma} 
  \begin{proof}
  We only need to show that for all $s \in [0,\infty)$, $\rho_0 \in L^p_+(\T^d)$ there exists a $\mathbb{P}$-null set $N_{s,\rho_0}$ such that for all $t \in [0,\infty)$ and $\omega \in N_{s,\rho_0}^c$,~\eqref{eq:crudecocyle} is valid. By the continuity of $\check \rho$ in the initial condition that follows from Lemma~\ref{lem:cmod}, we can, for every $s\in[0,\infty)$, find a null set $N_s$ such that for all $\omega \in N_s^c$, $t \in [0,\infty)$, and $\rho_0\in L^1_+(\T^d)$,  the crude cocycle property~\eqref{eq:crudecocyle} holds true.

Consider the equation~\eqref{eq:DK} solved by $\rho(t+s,s,\rho_0,\omega)$. We have for all $t \in [0,\infty)$ and $\varphi \in C_c^\infty(\T^d \times (0,\infty))$
\allowdisplaybreaks
\begin{align*}
		&\int_{\R}\int_{\T^d} \bar \chi(\rho(s+t,s,\rho_0,\omega)) \varphi(x,\xi) \dx{x} \dx{\xi} =  \int_{\R} \int_{\T^d} \bar{\chi}(\rho_0) \varphi(x,\xi) \dx{x}\dx{\xi}\\& -\int_0^{t} \int_{\T^d}\Phi'(\rho(r+s,s,\rho_0,\omega)) \nabla \rho(r+s,s,\rho_0,\omega) \cdot \nabla \varphi(x,\rho(r+s,s,\rho_0,\omega)) \dx{x} \dx{r} \\&+\frac{F_1}{2}\int_0^t \int_{\T^d}[\sigma'(\rho(r+s,s,\rho_0,\omega))]^2 \nabla \rho(r+s,s,\rho_0,\omega) \cdot \nabla \varphi(x,\rho(r+s,s,\rho_0,\omega)) \dx{x}\dx{r} \\
		&- \int_0^t \int_{\R} \int_{\T^d}(\partial_{\xi} \varphi(x,\xi)) \dx{q} \\&+\frac{F_3}{2}\int_0^t \int_{\T^d} \sigma^2(\rho(r+s,s,\rho_0,\omega)) (\partial_\xi \varphi)(x,\rho(r+s,s,\rho_0,\omega)) \dx{x}\dx{r} \\
		&- \int_{0}^{t}\int_{\T^d} \varphi(x,\rho(r+s,s,\rho_0,\omega) )\nabla \cdot \nu(\rho(r+s,s,\rho_0,\omega) ) \dx{x}\dx{r} \\
		&- \int_{0}^t\int_{\T^d} \varphi(x,\rho(r+s,s,\rho_0,\omega) )(\nabla \cdot   B(\rho(r+s,s,\rho_0,\omega) ) + f(\rho(r+s,s,\rho_0,\omega))) \dx{x}\dx{r}\\
		&- \sum_{k \in \Z^d} \lambda_k\left(\int_s^{s+t} \varphi(x,\rho(r,s,\rho_0,\cdot)) \nabla \cdot(e_k \sigma(\rho(r,s,\rho_0,\cdot)) dB_r^k) \dx{x} \right) (\omega) \, ,
\end{align*}
where we have used the fact that
\begin{align}
&\int_s^{s+t} \int_{\T^d }G(\rho(r,s,\rho_0,\omega),\nabla\rho(r,s,\rho_0,\omega)) \dx{x}\dx{r} \\=&\int_0^{t}\int_{\T^d} G(\rho(r+s,s,\rho_0,\omega),\nabla\rho(r,s,\rho_0,\omega)))\dx{x} \dx{r} \, ,
\end{align}
for any measurable $G: \R \to \R$. Similarly, we have 
\begin{align*}
		&\int_{\R}\int_{\T^d} \bar \chi(\rho(t,0,\rho_0,\theta_s \omega)) \varphi(x,\xi) \dx{x} \dx{\xi} =  \int_{\R} \int_{\T^d} \bar{\chi}(\rho_0) \varphi(x,\xi) \dx{x}\dx{\xi}\\& -\int_0^{t} \int_{\T^d}\Phi'(\rho(r,0,\rho_0,\theta_s \omega)) \nabla \rho(r,0,\rho_0,\theta_s \omega) \cdot \nabla \varphi(x,\rho(r,0,\rho_0,\theta_s \omega)) \dx{x} \dx{r} \\&+\frac{F_1}{2}\int_0^t \int_{\T^d} [\sigma'(\rho(r,0,\rho_0,\theta_s \omega))]^2\nabla\rho(r,0,\rho_0,\theta_s \omega) \cdot \nabla \varphi(x,\rho(r,0,\rho_0,\theta_s \omega)) \dx{x}\dx{r} \\
		&- \int_0^t \int_{\R} \int_{\T^d}(\partial_{\xi} \varphi(x,\xi)) \dx{q} \\&+\frac{F_3}{2}\int_0^t \int_{\T^d} \sigma^2(\rho(r,0,\rho_0,\theta_s \omega)) (\partial_\xi \varphi)(x,\rho(r,0,\rho_0,\theta_s\omega)) \dx{x}\dx{r} \\
		&- \int_{0}^{t}\int_{\T^d} \varphi(x,\rho(r,0,\rho_0,\theta_s \omega) )\nabla \cdot \nu(\rho(r,0,\rho_0,\theta_s \omega) ) \dx{x}\dx{r} \\
		&- \int_{0}^t\int_{\T^d} \varphi(x,\rho(r,0,\rho_0,\theta_s \omega) )(\nabla \cdot   B(\rho(r,0,\rho_0,\theta_s \omega)) + f(\rho(r,0,\rho_0,\theta_s \omega))) \dx{x}\dx{r}\\
		&- \sum_{k \in \Z^d} \lambda_k\left(\int_0^{t} \varphi(x,\rho(r,0,\rho_0,\cdot)) \nabla \cdot(e_k \sigma (\rho(r,0,\rho_0,\cdot)) dB_r^k) \dx{x} \right) (\theta_s \omega) \, .
\end{align*}
We will now show that for any family of $\mathcal{F}_{t,s}$-predictable processes $g_k:[s,\infty) \times \Omega \to \R$ such that $\sup_{k \in \Z^d}\int_s^T\mathbb{E}(|g_k|^2) \dx{t} <\infty$, for all $s \in [0,\infty)$ and $T \geq 0$,  there exists a $\mathbb{P}$-null set $N_{s,T}$, such that for all $\omega \in N_{s,T}^c$ and $t \in [0,T-s]$, we have
\begin{equation}
\sum_{k \in \Z^d} \lambda_k\int_s^{s+t} g_k(r,\cdot) \, dB^k_r(\cdot) (\omega) = \sum_{k \in \Z^d} \lambda_k\int_0^{t} g_k(r +s,\cdot) \, dB^k_r(\theta_s \cdot)  (\omega) \, .
\label{eq:cocycleint}
\end{equation}
One can check that the above equality is valid on elementary $\mathcal{F}_{t,s}$-predictable processes by using the helix property of Proposition~\ref{prop:helix}. Indeed, for any elementary $g$ and for a fixed $k \in \Z^d$, we have
\begin{align}
\int_s^{s+t} g(r,\cdot) \, dB^k_r(\cdot) (\omega) = & \sum_{j=1}^N g(r_j,\omega)  (B_{r_{j+1}}^k - B_{r_j}^k) (\omega) \\
=& \sum_{j=1}^N g(l_j +s,\omega)  (B_{s+l_{j+1}}^k - B_{s+l_j}^k) (\omega) \\
= & \sum_{j=1}^N g(l_j +s,\omega)  (B_{l_{j+1}}^k - B_{l_j}^k) (\theta_s \omega) \, ,
\end{align}
where $r_j,l_j$ are partitions of $[s,s+t], [0,t]$, respectively. It follows then that we can obtain~\eqref{eq:cocycleint} on all predictable processes by approximating them by elementary processes in $L^2(\Omega \times [0,T])$. Using Assumption~\ref{assumption_1}, the equality holds true for the entire sum.

Thus, both $\Omega \times [0,\infty) \ni (\omega,t) \mapsto \rho(t+s,s,\rho_0,\omega) $ and $\Omega \times [0,\infty) \ni (\omega,t) \mapsto \rho(t,s,\rho_0,\theta_s \omega) $ are stochastic kinetic solutions of~\eqref{eq:DK} with the same initial datum. It follows from~\eqref{eq:contraction} that we can find $\mathbb{P}$-null set $N_{s,\rho_0}$ such that for all $t \in [0,\infty)$ and  $\omega \in N_{s,\rho_0}^c$, we have that
\begin{equation*}
\rho(t+s,s,\rho_0,\omega) = \rho(t,0, \rho_0,\theta_s\omega) \, .
\end{equation*}
It follows then from Lemma~\ref{lem:cmod} that for all $\rho_0\in L^p_+(\T^d)$ and $s \in [0,\infty)$ there exists a $\mathbb{P}$-null set $N_{s,\rho_0}$ such that for all $\omega \in N_{s,\rho_0}^c$ and $t \in [0,\infty)$
\begin{equation*}
\check \rho(t+s,s,\rho_0,\omega) = \check\rho(t,0, \rho_0,\theta_s\omega) \, .
\end{equation*}
We can extend this to all $\rho_0 \in L^1_+(\T^d)$ by using the continuity of $\rho_0$ with respect to $\rho_0$ established in Lemma~\ref{lem:cmod}.  \end{proof}

We will now use the abstract perfection result of Theorem~\ref{thm:ks} to upgrade the above crude cocycle property to a perfect cocycle.
\begin{prop}[Perfection]
Assume Assumptions \ref{assumption_1}, \ref{assumption_2}, \ref{assumption_3}, \ref{assdrift}, and \ref{assume_grad} are satisfied and let $\check \rho(t,s,\rho_0,\omega)$ be as constructed in Lemma~\ref{lem:cmod}. Then, there exists a $\tilde \rho : \Delta \times L^1_+(\T^d)\times \Omega \to L^1_+(\T^d)$ which satisfies the following
\begin{enumerate}
 \item[\textup{(1)}] (Perfect) semiflow property: There exists a $\mathbb{P}$-null set $N$ such that, for all $\omega \in N^c$, $(t,s) ,(t,s_2)\in \Delta$ with $s \leq s_1$, and $\rho_0 \in L^1_+(\T^d)$, we have
 \[
  \tilde\rho(t,s,\rho_0,\omega)= \tilde\rho(t,s_1,\tilde\rho(s_1,s,\rho_0,\omega),\omega) \, .
 \]
 \item[\textup{(2)}] (Perfect) cocycle property: There exists a $\mathbb{P}$-null set $N$ such that, for all $\omega \in N^c$, $(t,s)\in \Delta$, and $\rho_0 \in L^1_+(\T^d)$, we have
 \[
 	\tilde \rho(t+s,s,\rho_0,\omega) = \tilde\rho(t,0, \rho_0,\theta_s\omega) \, .
 \]
 \item[\textup{(3)}] $\tilde \rho$ is $(\mathcal{B}(\Delta)\otimes \mathcal{F}^0,\mathcal{B}(X))$-measurable and $(\mathcal{B}(\Delta)\otimes \mathcal{B}(L^1_+(\T^d))\otimes \mathcal{F}^0,\mathcal{B}(L^1_+(\T^d)))$-measurable.
 \item[\textup{(4)}] There exists a $\mathbb{P}$-null set $N$ such that, for all $\omega \in N^c$, $\tilde \rho(t,s,\rho_0,\omega)$ is continuous as a map from $\Delta \to X$.
 \item[\textup{(5)}] $\mathbb{P}(\{\omega \in \Omega: \check\rho(t,s,\rho_0,\omega) =\tilde \rho(t,s,\rho_0,\omega), \textrm{ for all } (t,s) \in \Delta , \rho_0 \in L^1_+(\T^d)\})=1$.
 \item[\textup{(6)}] There exists a $\mathbb{P}$-null set $N$ such that for all $\omega \in N^c$, $s\in [0,\infty)$, $T \in [s,\infty)$, and $\rho_0 \in L^1_+(\T^d)$, $\tilde \rho$ satisfies
 \[
\sup_{t \in [s,T]}\norm{\tilde \rho(t,s,\rho_{0,1},\omega)- \tilde \rho(t,s,\rho_{0,2},\omega)}_{L^1(\T^d)} \leq C(T) \norm{\rho_{0,1}-\rho_{0,2}}_{L^1(\T^d)} \, .
\]
\end{enumerate}
\label{prop:perfection}
\end{prop}
\begin{proof}
We define $(\mathcal{G},\circ)= X$ with $\circ$ given by composition. All we need to check is that conditions (1), (2), (3), and (4) of Theorem~\ref{thm:ks} are satisfied for the maps  $\phi (t,s,\cdot,\omega):=\check \rho(t,s,\cdot,\omega)$. Conditions (1), (2), and (4) follow from the results of Lemmas~\ref{lem:cmod}, ~\ref{lem:semiflow}, and~\ref{lem:crudecocycle}, respectively. For (3), we only know that $\check \rho(t,s,\rho_0,\omega)$ is $(\mathcal{B}(\Delta) \otimes \mathcal{F}, \mathcal{B}(\mathcal{G}))$-measurable. We would like to replace the $\mathcal{F}$ by $\mathcal{F}^0$. In order to do this, we note that for all fixed $(t,s) \in \Delta $ $\check \rho(t,s, \rho_0,\omega)$ is $(\mathcal{F},\mathcal{B}(\mathcal{G}))$-measurable as a map from $\Omega \to \mathcal{G}$. It follows 
that for each $(t,s) \in \Delta_c$ for some countable $\Delta_c \subset \Delta$, we can find a $\mathbb{P}$-null set $N$  and some $\hat \rho(t,s,\rho_0,\omega)$ such that, for all $(t,s)\in \Delta_c$, $\hat \rho(t,s,\rho_0,\omega)$ is $(\mathcal{F}^0,\mathcal{B}(\mathcal{G}))$-measurable as a map from $\Omega\to \mathcal{G}$. By the continuity of $\check\rho(t,s,\rho_0,\omega)$ as a map from $\Delta \to \mathcal{G}$  (see~Lemma~\ref{lem:cmod}), we have that, for all $\omega \in N^c$ and for all $(t,s)\in \Delta$, $\hat \rho(t,s,\rho_0,\omega) = \check \rho(t,s,\rho_0,\omega)$ and, for all $\omega \in N^c$ and $(t,s)\in \Delta_c$, $\hat \rho(t,s,\rho_0,\omega)= \check \rho(t,s,\rho_0,\omega)$ and $\hat \rho(t,s,\rho_0,\omega)$ is $(\mathcal{F}^0,\mathcal{B}(\mathcal{G}))$-measurable as a map from $\Omega \to \mathcal{G}$ (since it is the almost sure limit of such measurable functions). Furthermore, $\hat \rho$ satisfies the same continuity properties as $\tilde \rho$ including the stability estimate~\eqref{eq:contraction}. Thus, by arguments similar to those in the proof of Lemma~\ref{lem:cmod}, $\hat \rho(t,s,\rho_0,\omega)$ is $(\mathcal{B}(\Delta)\otimes \mathcal{F}^0,\mathcal{B}(\mathcal{G}))$-measurable as a map from $\Delta \times \Omega \to \mathcal{G}$.  We can apply similar arguments to obtain that $\hat \rho(t,s,\rho_0,\omega)$ is $(\mathcal{B}(\Delta)\otimes\mathcal{B}(L^1_+(\T^d))\otimes \mathcal{F}^0,\mathcal{B}(L^1_+(\T^d)))$-measurable as a map from $\Delta \times L^1_+(\T^d) \times \Omega \to L^1_+(\T^d).$

 It is also clear that $\hat \rho$ satisfies conditions (1), (2), and (4) of Theorem~\ref{thm:ks}. Thus, applying  Theorem~\ref{thm:ks}, we obtain a $\tilde \rho$ which satisfies properties (1), (2), (4), and (5). Property (6) simply follows from property (5) and Lemma~\ref{lem:cmod}. The first kind of measurability in property (3) follows from Theorem~\ref{thm:ks} while the second kind follows from the fact that $\tilde \rho$ is indistinguishable from $\hat \rho$. This completes the proof of the proposition.  \end{proof}

\subsection{Generation of a random dynamical system}
We conclude with the main result of this section. 
\begin{theorem}
Assume Assumptions \ref{assumption_1}, \ref{assumption_2}, \ref{assumption_3}, \ref{assdrift}, and \ref{assume_grad} are satisfied. Consider the map $\varphi:[0,\infty) \times L^1_+(\T^d)\times \Omega \to L^1_+(\T^d)$ defined as follows
\begin{align}
\varphi(t,\rho_0,\omega):=
\begin{cases}
\tilde \rho(t,0,\rho_0,\omega) & \omega \in N^c\\
\rho_0 & \omega \in N
\end{cases}
\end{align}
with $N$ and $\tilde \rho$ as in the statement of Proposition~\ref{prop:perfection}. Then, $\varphi$ defines a continuous RDS on $(L^1_+(\T^d),\mathcal{B}(L^1_+(\T^d)))$ over the MDS $(\Omega,\mathcal{F}^0, \mathbb{P}, (\theta_t)_{t \in \R})$ in the sense of Definition~\ref{def:rds}.  Additionally, for all $\omega \in \Omega$, $\rho_{0,1},\rho_{0,2} \in L^1_+(\T^d)$,  and $T \geq 0$, $\varphi$ satisfies the following stability estimate
\begin{equation}
\sup_{t \in [0,T]}\norm{\varphi(t,\rho_{0,1},\omega)- \varphi(t,\rho_{0,2},\omega)}_{L^1(\T^d)} \leq C(T) \norm{\rho_{0,1}-\rho_{0,2}}_{L^1(\T^d)} \, .
\label{eq:stabphi}
\end{equation}
Finally, for all $\rho_0 \in L^p_+(\T^d)$, there exists a $\mathbb{P}$-null set $N_{\rho_0}$, such that for all $\omega \in N_{\rho_0}^c$, and all $t \in [0,\infty)$, we have $\varphi(t,\rho_0,\omega):= \rho(t,0,\rho_0,\omega)$.
\label{thm:rds} 
\end{theorem}
\begin{proof}
Property 2 of Definition~\ref{def:rds} follows simply from the definition of $\varphi$ and the result of Proposition~\ref{prop:perfection}. Similarly, the required continuity and measurability of $\varphi$ also follow from the continuity and measurability of $\tilde \rho$ given by properties (3) and (6) of Proposition~\ref{prop:perfection}.  The stability estimate~\eqref{eq:stabphi} also follows from property (6) of Proposition~\ref{prop:perfection}. The fact that $\varphi$ is a modification of $\rho$ follows from property (6) of Proposition~\ref{prop:perfection} and Lemma~\ref{lem:cmod}.
\end{proof}

\section{Ergodicity}
\subsection{An abstract coupling argument}
Let $(X,d)$ be a separable, metric space which is not necessarily complete and $(\Omega,\mathcal{F},(\mathcal{F}_t)_{t \geq 0}, \mathbb{P})$ be a filtered probability space with right-continuous and complete filtration. We let $ [0,\infty) \times X \times \omega \ni (t,\rho_0,\omega)  \mapsto \rho(t,\rho_0,\omega) \in X$ be a time-homogeneous almost-surely continuous Markov process adapted to $\mathcal F_t$ and started at $\rho_0 \in X$. Furthermore, we let $X \ni \rho_0 \mapsto P_t(\rho_0, \cdot) \in \mathcal{P}(X)$.  By a slight abuse of notation, we let $P_t$ be the corresponding Markov semigroup acting on observables $\varphi \in C_b(X)$ and $P_t^*$ its adjoint acting on probability measures $\mu \in \mathcal{P}(X)$.

We note that the approach of this section is inspired by~\cite{KPS10} but with several key differences due to the stronger, pathwise $e$-property. Before we can proceed to the proof of ergodicity, we state and explain a few natural assumptions we will require for the proof.
\begin{assumption}[Non-expansivity]
Consider the two-point Markov process \\ $(\rho(t,\rho_{0,1},\cdot),\rho(t,\rho_{0,2},\cdot))$  for some $(\rho_{0,1},\rho_{0,2})\in X \times X$. Then, $\mathbb{P}$-almost surely, we have
\begin{equation}
d(\rho(t,\rho_{0,1},\omega),\rho(t,\rho_{0,2},\omega)) \leq d(\rho_{0,1},\rho_{0,2}) \, ,
\end{equation}
for all $t\geq 0$.
\label{ass:nonexpansive}
\end{assumption}
The above assumption will play an important role in the proof of ergodicity as the proof relies on a coupling argument. Furthermore, it follows from Theorem~\ref{thm:sks} (more specifically from~\eqref{eq:contraction}) that it is satisfied by the choice $\nu, B\equiv 0$ and $\Phi,\sigma$ satisfying Assumptions~\ref{ass1} and~\ref{ass2}. 
\begin{assumption}[Support property]
There exists a map $u:[0,T]\times X \to X$ such that:
\begin{tenumerate}
 \item for all $R \geq 0$ and for all compact subsets $\mathcal{M}\subset X$ with $\diam(\mathcal{M})=R$, there exists a bounded and decreasing function $C_{R}:[0,\infty) \to [0,\infty)$ such that, for all $t\geq 0$ and $(\rho_{0,1},\rho_{0,2})\in X \times X$,
\begin{equation}
d(u(t,\rho_{0,1}),u(t,\rho_{0,2})) \leq C_{R}(t) \quad \textrm{and} \quad \lim_{t \to\infty} C_R(t)=0 \, ,
\end{equation} 
and  \label{ass:supporta}
\item  for all $\rho_{0} \in X$, $u(\cdot,\rho_0)$ lies in the support of path-space measure of $\rho(\cdot,\rho_0,\cdot)$, i.e. for all $\delta>0$,  we have
\begin{equation}
\mathbb{P}\bra*{\int_0^T d(\rho(t,\rho_0,\cdot),u(t,\rho_0)) \, \dx{t}\leq \frac{\delta}{2}}>0 \, .
\end{equation}
\label{ass:supportb}
\end{tenumerate}
\label{ass:support}
\label{ass:contract}
\end{assumption}
We will use the above assumption in the proof of ergodicity to show that the process $\rho(t,\rho_0,\cdot)$ gets close to the $u(t,\rho_0)$ with positive probability. 
\begin{assumption}[Strong Markov property]
We assume that the process $\rho(t,\cdot,\cdot)$ satisfies the strong Markov property (see~\cite[Chapter 9, Eq. (9.36)]{DPZ14}), i.e. let $\tau$ be a $\mathcal{F}_t$-stopping time. Then, for all $F \in B_b(X)$ and $\rho_0 \in X$, we have $\mathbb{P}$-almost surely on $\{\omega \in \Omega:\tau <\infty\}$,
\begin{equation}
\mathbb{E}(F(\rho (t+\tau,\rho_0,\cdot ))| \cF_\tau) (\omega) = \mathbb{E}(F(\rho(t, \rho(\tau,\rho_0,\omega),\cdot))) \, ,
\end{equation}
for all $t\geq 0$ where $\cF_\tau$ consists of all events $A \in \cF$ such that $\set{\tau \leq t} \cap A \in \cF_t$ for all $t \geq 0$. We assume the same condition is satisfied the two-point process $[0,T] \times X \times X \times \Omega  \ni (t,(\rho_{0,1},\rho_{0,2}),\omega)\mapsto (\rho(t,\rho_{0,1},\omega),\rho(t,\rho_{0,2},\omega)) \in X \times X$.
\label{ass:Markov}
\end{assumption}
 We will show in Appendix~\ref{app:markov} (Propositions~\ref{prop:markov} and~\ref{prop:strongmarkov}) that the above property is satisfied for stochastic kinetic solutions of~\eqref{eq:DK} under  Assumptions~\ref{assumption_1}, ~\ref{ass1}, ~\ref{ass2}, and~\ref{assdrift}. Note that $X \times X \ni (\rho_{0,1},\rho_{0,2}) \mapsto P_t^{2}((\rho_{0,1},\rho_{0,2}),\cdot)\in \cP(X \times X)$ will denote the two-point transition probabilities, i.e. the transition probability of the two-point process $[0,T] \times X \times X \times \Omega  \ni (t,(\rho_{0,1},\rho_{0,2}),\omega)\mapsto (\rho(t,\rho_{0,1},\omega),\rho(t,\rho_{0,2},\omega)) \in X \times X$.
\begin{assumption}[Dissipation estimate]
We assume that there exist lower semicontinuous functions $\Psi_1,\Psi_2: X \to \R_+$ such that the process $\rho(t,\rho_0,\cdot)$ started at some $\rho_0 \in X$ satisfies a dissipation estimate, i.e. there exists a $C>0$ such that for all $T \geq 0$
\begin{align}
\label{dissest} &\esup_{0 \leq t \leq T}\mathbb{E}\bra*{\Psi_1(\rho(t,\rho_0,\cdot))}+ \mathbb{E}\bra*{ \int_0^T\Psi_2(\rho(t,\rho_0,\cdot)) \dx{t}  } \\ & \leq C (T + \Psi_1(\rho_0)) \, .
\end{align} 
Furthermore, we assume that $\Psi_2$ has compact sublevel sets.
\label{ass:entene}
\end{assumption}

The above assumption will be used to show that the process $\rho(t,\rho_0)$ concentrates with positive probability on compact subsets of $X$.   We will assume for the rest of the section that the Assumptions~\ref{ass:nonexpansive}, \ref{ass:contract}, \ref{ass:Markov}, and \ref{ass:entene} are satisfied. We next proceed to the proof of ergodicity for the process.

We start the proof of ergodicity by showing that the paths of the process $\rho(t,\cdot,\cdot)$ get close with positive probability.
\begin{lemma}
Let  $\mathcal{M} \subset X$ be compact  and set $R:=\diam (\mathcal{M})$. Then, for every $\delta>0$, there exists a $t_0=t_0(R,\delta)\geq 1$ such that for all $\rho_{0,1},\rho_{0,2} \in \mathcal{M}$ and all $t \geq t_0$ there exists a $ \gamma=\gamma(t,\delta)\in (0,1)$ such that
\begin{align}
\mathbb{P}\bra*{ d(\rho(t,\rho_{0,1},\cdot),\rho(t,\rho_{0,2},\cdot))> \delta} < \gamma \, .
\end{align}
\label{lem:close} 
\end{lemma} 
\begin{proof}
Fix $t\geq t_0\geq 1$ and choose
\begin{equation}
t_0 := 1+ \sup\set*{s \in [0,\infty): C_{R}(s)> \frac{\delta}{4} } \, ,
\end{equation}
where $C_{R}$ is as defined as in Assumption~\ref{ass:support}~\ref{ass:supporta}. We will first show that
\begin{align}
&\mathbb{P}\bra*{ d(\rho(t,\rho_{0,1},\cdot),\rho(t,\rho_{0,2},\cdot)) >  \frac\delta 2 } \\\leq&  \mathbb{P}\bra*{\int_{t-1}^t d(\rho(s,\rho_{0,1},\cdot),\rho(s,\rho_{0,2},\cdot)) \dx{s}  >  \frac \delta 2 } \, .
\label{spacetotime}
\end{align}
We start by assuming this is not case. To this end, consider the set $B\subset \Omega$ defined as
\begin{equation}
B:= \set*{\omega \in \Omega:\int_{t-1}^t d(\rho(s,\rho_{0,1},\omega),\rho(s,\rho_{0,2},\omega)) \dx{s} \leq \frac \delta 2 } \, .
\end{equation}
We now choose as a stopping time 
\begin{equation}
\tau(\omega) =\inf\set*{s \in [t-1, t]: d(\rho(s,\rho_{0,1},\omega),\rho(s,\rho_{0,2},\omega))\leq \frac \delta 2} \, . 
\end{equation}

We note that $\tau$ is almost surely finite on $B$ and is an $\mathcal{F}_t$-stopping time. Furthermore, by  Assumption~\ref{ass:Markov}, the strong Markov property is satisfied for the two-point process. Thus, we have
\begin{align}
&\mathbb{P}\bra*{d\bra{\rho(t,\rho_{0,1},\cdot),\rho(t,\rho_{0,2},\cdot)} \leq \frac \delta 2} \\
=&  \mathbb{E}\bra*{\mathbb{P}\bigg(d\bra{\rho(t,\rho_{0,1},\omega),\rho(t,\rho_{0,2},\omega)} \leq \frac \delta 2 \mid \mathcal{F}_\tau \bigg)(\cdot)} \\
\geq &\mathbb{E}\bra*{\mathbf{1}_B(\cdot)\mathbb{P}\bigg(d\bra{\rho(t,\rho_{0,1},\omega),\rho(t,\rho_{0,2},\omega)} \leq \frac \delta 2 \mid \mathcal{F}_\tau \bigg)(\cdot)} \\
=& \mathbb{E}\bra*{\mathbf{1}_B(\cdot)\mathbb{P}\bigg(d\bra{\rho(t-\tau(\cdot),\rho(\tau(\cdot),\rho_{0,1},\cdot),\omega),\rho(t-\tau(\cdot),\rho(\tau(\cdot),\rho_{0,2},\cdot),\omega)} \leq \frac \delta 2  \bigg)}  \\
\geq & \mathbb{E}\bra*{\mathbf{1}_B(\cdot)\mathbb{P}\bigg(d\bra{\rho(\tau(\cdot),\rho_{0,1},\cdot),\rho(\tau(\cdot),\rho_{0,2},\cdot)} \leq \frac \delta 2  \bigg)}  \\
=& \mathbb{P}(B) = \mathbb{P}\bra*{\int_{t-1}^t d(\rho(s,\rho_{0,1},\cdot),\rho(s,\rho_{0,2},\cdot)) \dx{s}  \leq  \frac \delta 2 }\, ,
\end{align}
where in we have used Assumption~\ref{ass:nonexpansive} along with the fact that $\rho$ is $\mathbb{P}$-almost surely continuous. Thus, we have that
\begin{align}
&\mathbb{P}(d\bra{\rho(t,\rho_{0,1},\cdot),\rho(t,\rho_{0,2},\cdot)} > \frac \delta 2 )\\
\leq&  \mathbb{P}\bra*{\int_{t-1}^t d(\rho(s,\rho_{0,1},\cdot),\rho(s,\rho_{0,2},\cdot)) \dx{s}  >  \frac \delta 2 } \\
\leq & \mathbb{P}\left(\int_{t-1}^t d(\rho(s,\rho_{0,1},\cdot),u(s,\rho_{0,1})) \dx{s} + d(\rho(s,\rho_{0,2},\cdot),u(s,\rho_{0,2})) \dx{s} 
 \right.\\
&\left.+ \int_{t-1}^t d(u(s,\rho_{0,1}),u(s,\rho_{0,2})) \dx{s}>  \frac \delta 2 \right) \\
\leq &\mathbb{P} \left(\int_{t-1}^t d(\rho(s,\rho_{0,1},\cdot),u(s,\rho_{0,1})) \dx{s} + d(\rho(s,\rho_{0,2},\cdot),u(s,\rho_{0,2})) \dx{s} + C_{R}(t-1) >\frac \delta 2\right)\,  \\
\leq & \mathbb{P}\bra*{\int_{t-1}^t d(\rho(s,\rho_{0,1},\cdot),u(s,\rho_{0,1})) \dx{s} + d(\rho(s,\rho_{0,2},\cdot),u(s,\rho_{0,2})) \dx{s} +
\frac{\delta}{4} > \frac \delta 2}  \\
\leq &\mathbb{P}\bra*{\int_{t-1}^t d(\rho(s,\rho_{0,1},\cdot),u(s,\rho_{0,1})) \dx{s} + \int_{t-1}^t d(\rho(s,\rho_{0,2},\cdot),u(s,\rho_{0,2})) \dx{s} > \frac{\delta}{4}} \\
\leq &\mathbb{P}\bra*{\int_{t-1}^t d(\rho(s,\rho_{0,1},\cdot),u(s,\rho_{0,1})) \dx{s}  > \frac{\delta}{4}} + \mathbb{P}\bra*{\int_{t-1}^t d(\rho(s,\rho_{0,2},\cdot),u(s,\rho_{0,2})) \dx{s} > \frac{\delta}{4}} \\=:& \gamma_{\rho_{0,1},\rho_{0,2}}(t,\delta) <1 \, ,
\end{align}
 where we have used both Assumption~\ref{ass:support}~\ref{ass:supporta} and~\ref{ass:supportb} along with the fact that $t \geq t_0$. Since $\mathcal{M} \times \mathcal{M}$ is compact, it is totally bounded and so we can choose finite $\frac \delta 2$-covering $\cG$ of it. We set
 \begin{align}
\gamma(t,\delta) = \max_{(f,g) \in \cG}\gamma_{f,g}(t,\delta) <1 \, .
 \end{align}
 For any $(\rho_{0,1},\rho_{0,2}) \in \mathcal{M} \times \mathcal{M}$ there exists some $(f,g) \in \cG$ such that $d(\rho_{0,1},f)+ d(\rho_{0,2},g)\leq \frac \delta 2$. Thus, we have 
 \begin{align}
&d\bra{\rho(t,\rho_{0,1},\omega),\rho(t,\rho_{0,2},\omega)} \\
\leq &d\bra{\rho(t,f,\omega),\rho(t,g,\omega)} + d\bra{\rho(t,f,\omega),\rho(t,\rho_{0,1},\omega)} + d\bra{\rho(t,g,\omega),\rho(t,\rho_{0,2},\omega)} \\
\leq & d\bra{\rho(t,f,\omega),\rho(t,g,\omega)} + \frac \delta 2 \, .
 \end{align}
 It follows that, for any $(\rho_{0,1},\rho_{0,2}) \in \mathcal{M} \times \mathcal{M} $, we have
 \begin{align}
&\mathbb{P} \bra{d\bra{\rho(t,\rho_{0,1},\cdot),\rho(t,\rho_{0,2},\cdot)} >  \delta }\\\leq& \mathbb{P} \bra*{d\bra{\rho(t,f,\cdot),\rho(t,g,\cdot)} >  \frac \delta 2} \leq   \gamma(t,\delta) <1,
 \end{align}
 which completes the proof.  \end{proof}
 We now proceed with the following result:
 \begin{lemma}
  There exists a compact subset of $\cK \subset X$, such that for all $\cC \subset X$, such that there exists an $R_0\geq 2$ with $\sup_{\rho_0 \in \mathcal{C}}\Psi_1(\rho_0)< R_0$, for all $t \geq R_0$, and all $\rho_0 \in \cC$, we have
\begin{align}
Q_t(\rho_0,\cK):= \frac{1}{t} \int_0^t P_s(\rho_0,\cK) \dx{s} \geq \frac{1}{2} \, . 
\end{align}
Similarly, for all $t \geq R_0$ and all $(\rho_{0,1},\rho_{0,2}) \in \cC \times \cC$, it holds that
\begin{align}
Q_t^2((\rho_{0,1},\rho_{0,2}),\cK \times \cK):= \frac{1}{t} \int_0^t P_s^2((\rho_{0,1},\rho_{0,2}),\cK \times \cK) \dx{s} \geq \frac{1}{2} \, .
\end{align}
\label{lem:compact}
\end{lemma}
\begin{proof}
The proof of this result relies on Assumption~\ref{ass:entene}. Using~\eqref{dissest}, we obtain
\begin{align}
&\esup_{0 \leq t \leq T}\mathbb{E}\bra*{\Psi_1(\rho(t,\rho_0,\cdot))}+ \mathbb{E}\bra*{ \int_0^T \Psi_2(\rho(t,\rho_0,\cdot)) \dx{t}  } \\ & \leq C (T + \Psi_1(\rho_0)) \,  ,
\end{align}
for all $\rho_0 \in \cC$. It follows that
\begin{align}
\mathbb{E}\bra*{\frac{1}{T} \int_0^T\Psi_2(\rho(t,\rho_0,\cdot)) \dx{t}  } \leq C \bra*{1 + \frac{\Psi_1(\rho_0)}{T}} \, .
\end{align}
We define the set
\begin{align}
\cK:=\set*{\rho \in X :  \Psi_2(\rho) \leq C_2} \, ,
\end{align}
for some constant $C_2>0$ which we specify below. We note that by Assumption~\ref{ass:entene}, $\cK\subset X$ is compact. We thus have that
\begin{align}
Q_t(\rho_0,\cK^c) =& \frac{1}{t} \int_0^t P_s(\rho_0,\cK^c) \dx{s} \\
=& \frac{1}{t} \int_0^t \mathbb{P}\bra*{\Psi_2(\rho(s,\rho_0,\cdot))> C_2} \dx{s}\\
\leq & \frac{1}{C_2 t} \int_0^t \mathbb{E}\bra*{\Psi_2(\rho(s,\rho_0,\cdot))} \dx{s} \\\label{eq:tight}
\leq & \frac{C}{C_2} \bra*{1 + \frac{ \Psi_1(\rho_0) }{t}} \\
\leq & \frac{C}{C_2} \bra*{1 + \frac{R_0}{t}} \, .
\end{align}
Choosing $C_2=4C$ and $t \geq R_0$, we obtain the desired result. Similarly for the two-point process, we have
\begin{align}
&Q_t^2((\rho_{0,1},\rho_{0,2}),(\cK \times \cK)^c) \\=& \frac{1}{t} \int_0^t P_s^2((\rho_{0,1},\rho_{0,2}),(\cK \times \cK)^c) \dx{s}  \\
\leq & \frac{1}{t} \int_0^t \bra*{\mathbb{P}\bra*{\Psi_2(\rho(s,\rho_{0,1},\cdot)) > C_2} + \mathbb{P}\bra*{\Psi_2(\rho(s,\rho_{0,2},\cdot)) > C_2} }\dx{s} \\
\leq & \frac{1}{2} \, ,
\end{align} 
if we choose $C_2=8C$ and $t \geq R_0$. This completes the proof.
\end{proof}
We now set $t_0=t_0(\diam(\mathcal{K}\times \mathcal{K}),\delta)$ for some $\delta > 0$ and $\mathcal{K}$ the compact set constructed in the proof of Lemma~\ref{lem:compact}. We now show that paths of the process $\rho(t,\cdot,\cdot)$ get close to each other with finite probability in a time span depending only on the value of $\Psi_1(\rho_0)$.
\begin{lemma}
Let $\cC \subseteq X$  be such that $\sup_{\rho_0\in \cC} \Psi_1(\rho_0)< R_0< \infty$. Then, for all $\delta>0$ there exists an $\alpha=\alpha(t_0,\delta)$ such that there exists $T_1=T_1(R_0,\delta)$, such that for all $T \geq T_1=T_1(R_0,\delta)$, we have that
\begin{align}
\inf_{\rho_{0,1},\rho_{0,2}\in \cC}\mathbb{P}\bra*{d(\rho(T,\rho_{0,1},\cdot),\rho(T,\rho_{0,2},\cdot)) \leq \delta} > \alpha >0 \, .
\end{align}
\end{lemma}
\label{lem:T1}
\begin{proof}
We define the set $\Delta_\delta \subset X \times X$ as follows
\begin{align}
\Delta_\delta := \set*{(\rho_{0,1},\rho_{0,2}) \in X \times X: d(\rho_{0,1},\rho_{0,2}) \leq \delta} \, .
\end{align}
Using the Markov property (see Assumption~\ref{ass:Markov}), we have for all $s \in [0,t]$ and $t \in [0,T]$.
\begin{align}
&\mathbb{P}\bra*{d(\rho(T,\rho_{0,1},\cdot),\rho(T,\rho_{0,2},\cdot)) \leq \delta}  \\=&  P_T^2((\rho_{0,1},\rho_{0,2}), \Delta_\delta) \\
 = & \int_{X \times X} P_{T-s}^2((\bar{\rho}_{0,1},\bar{\rho}_{0,2}), \Delta_\delta) P_{s}^2(((\rho_{0,1},\rho_{0,2}), \dx{(\bar{\rho}_{0,1},\bar{\rho}_{0,2})}) \, .
\end{align}
 Integrating from $0$ to $t$, we obtain
\begin{align}
&\mathbb{P}\bra*{d(\rho(T,\rho_{0,1},\cdot),\rho(T,\rho_{0,2},\cdot)) \leq \delta} \\
\geq&  \int_{X \times X} \frac{1}{t}\int_0^tP_{T-s}^2((\bar{\rho}_{0,1},\bar{\rho}_{0,2}), \Delta_\delta) P_{s}^2(({\rho}_{0,1},{\rho}_{0,2}), \dx{(\bar{\rho}_{0,1},\bar{\rho}_{0,2})}) \dx{s} \\
\geq & \int_{X \times X} P_{T-t}^2((\bar{\rho}_{0,1},\bar{\rho}_{0,2}), \Delta_\delta) \frac{1}{t}\int_0^t P_{s}^2(({\rho}_{0,1},{\rho}_{0,2}), \dx{(\bar{\rho}_{0,1},\bar{\rho}_{0,2})}) \dx{s}\, .
\end{align} 
In the last step, we have used the pathwise non-expansivity from Assumption~\ref{ass:nonexpansive}. Thus, we have 
\begin{align}
&\mathbb{P}\bra*{d(\rho(T,\rho_{0,1},\cdot),\rho(T,\rho_{0,2},\cdot)) \leq \delta}\\
 \geq& \int_{\cK \times \cK} P_{T-t}^2((\bar{\rho}_{0,1},\bar{\rho}_{0,2}), \Delta_\delta) Q_t^2(({\rho}_{0,1},{\rho}_{0,2}),\dx{(\bar{\rho}_{0,1},\bar{\rho}_{0,2})}) \, ,
\end{align} 
where $\cK \times \cK$ is the set from Lemma~\ref{lem:compact}. We choose $t=T/2$ and  note that by Lemma~\ref{lem:close} if $T \geq 2 t_0$, we have that
\begin{align}
P_{T-t}^2((\bar{\rho}_{0,1},\bar{\rho}_{0,2}), \Delta_\delta) > \beta(t_0,\delta):= 1- \gamma(t_0,\delta) >0.
\end{align}
Using Lemma~\ref{lem:compact}, this leaves us with
\begin{align}
\mathbb{P}\bra*{d(\rho(T,\rho_{0,1},\cdot),\rho(T,\rho_{0,2},\cdot)) \leq \delta} \geq & \frac{\beta}{2}= : \alpha(t_0,\delta) \, , 
\end{align}
if we choose $T \geq \max(2t_0, R_0)=:T_1$. This completes the proof.
\end{proof}
We now iterate the above argument to show that for large times, with arbitrarily large probability, paths of the process $\rho(t,\cdot,\cdot)$ get arbitrarily close.
\begin{lemma}
Let $\rho_{0,1},\rho_{0,2} \in X$ such that $\Psi_1(\rho_{0,1}) +\Psi_1(\rho_{0,2}) < \infty $. Then, for every $\delta_1,\delta_2>0$, there exists a $t_2=t_2(\delta_1,\delta_2)>0$, such that for all $t \geq t_2$ we have that
\begin{align}
\mathbb{P}\bra*{d(\rho(t,\rho_{0,1},\cdot),\rho(t,\rho_{0,2},\cdot)) > \delta_1} \leq \delta_2 \, .
\end{align}
 Furthermore, for any $\delta_1>0$ sufficiently small, we have the following rate of convergence
\begin{align}
\mathbb{P}\bra*{d(\rho(t,\rho_{0,1},\cdot),\rho(t,\rho_{0,2},\cdot)) > \delta_1} \lesssim (1-\alpha)^{\sqrt{\ln(t)}} \, ,
\end{align} 
for $t$ sufficiently large and some $\alpha=\alpha(t_0,\delta_1)$ with $t_0$ as in Lemma~\ref{lem:close}.
\label{lem:t2}
\end{lemma}
\begin{proof}
We consider a sequence of times $\set{n_k}_{k \in \N}$ and define $N_k= \sum_{i=1}^k n_i$. Using the Markov property and the pathwise non-expansivity of Assumption~\ref{ass:nonexpansive}, we have
\begin{align}
&\mathbb{P}\bra*{d(\rho(N_k,\rho_{0,1},\cdot),\rho(N_k,\rho_{0,2},\cdot)) >\delta_1} \\
=& \int_{X \times X} P_{n_k}^2((\bar{\rho}_{0,1},\bar{\rho}_{0,2}), \Delta_{\delta_1}^c) P_{N_{k-1}}^2(({\rho}_{0,1},{\rho}_{0,2}), \dx{(\bar{\rho}_{0,1},\bar{\rho}_{0,2})}) \\
=& \int_{\Delta_{\delta_1}^c} P_{n_k}^2((\bar{\rho}_{0,1},\bar{\rho}_{0,2}), \Delta_{\delta_1}^c) P_{N_{k-1}}^2(({\rho}_{0,1},{\rho}_{0,2}), \dx{(\bar{\rho}_{0,1},\bar{\rho}_{0,2})}) \, .
\end{align}
We now define the set
\begin{align}
F_{M_k} = \set*{({\rho}_{0,1},{\rho}_{0,2}) \in X \times X : \Psi_1(\rho_{0,1}) +\Psi_1(\rho_{0,2})   \leq M_k } \, ,
\end{align}
for some constants $M_k \in \R, k \in \N$ which we will specify later. Thus, we can rewrite the previous expression as
\begin{align}
&\mathbb{P}\bra*{d(\rho(N_k,\rho_{0,1},\cdot),\rho(N_k,\rho_{0,2},\cdot)) >\delta_1} \\
=& \int_{\Delta_{\delta_1}^c \cap F_{M_k}} P_{n_k}^2((\bar{\rho}_{0,1},\bar{\rho}_{0,2}), \Delta_{\delta_1}^c) P_{N_{k-1}}^2(({\rho}_{0,1},{\rho}_{0,2}), \dx{(\bar{\rho}_{0,1},\bar{\rho}_{0,2})}) \\
&+ \int_{\Delta_{\delta_1}^c \cap F_{M_k}^c} P_{n_k}^2((\bar{\rho}_{0,1},\bar{\rho}_{0,2}), \Delta_{\delta_1}^c) P_{N_{k-1}}^2(({\rho}_{0,1},{\rho}_{0,2}), \dx{(\bar{\rho}_{0,1},\bar{\rho}_{0,2})}) \\
\leq &\bra*{ \sup_{\rho_{0,1},\rho_{0,2} \in F_{M_k}} P_{n_k}^2(({\rho}_{0,1},{\rho}_{0,2}), \Delta_{\delta_1}^c) } \bra*{  P_{N_{k-1}}^2(({\rho}_{0,1},{\rho}_{0,2}), \Delta_{\delta_1}^c) } \\
&+ \frac{1}{M_k} \mathbb{E}\bra*{(\Psi_1(\rho(N_{k-1},\rho_{0,1},\cdot))  + \Psi_1(\rho(N_{k-1},\rho_{0,2},\cdot)) )} \, .
\end{align}
Applying the dissipation estimate from Assumption~\ref{ass:entene}, we obtain 
\begin{align}
&\mathbb{E}\bra*{(\Psi_1(\rho(N_{k-1},\rho_{0,1},\cdot))  + \Psi_1(\rho(N_{k-1},\rho_{0,2},\cdot)) )} \\\leq& C\bra*{2 N_{k-1} +  \Psi_1(\rho_{0,1}) +\Psi_1(\rho_{0,2})  } \, . 
\end{align}
We choose
\begin{align}
M_k=&\max\set*{\frac{2C\bra*{2 N_{k-1} + \Psi_1(\rho_{0,1}) +\Psi_1(\rho_{0,2})  }}{\delta_2},0} \\
n_k=& M_k + T_1(M_k,\delta_1) \, .
\end{align}
with $T_1$ the same as in Lemma~\ref{lem:T1}. Applying Lemma~\ref{lem:T1}, it follows that
\begin{align}
\sup_{{\rho}_{0,1},{\rho}_{0,2} \in F_{M_k}} P_{n_k}^2(({\rho}_{0,1},{\rho}_{0,2}), \Delta_{\delta_1}^c) < 1-\alpha(t_0,\delta_1)<1 \,.
\end{align}
Furthermore, our choice of $M_k$ yields
\begin{align}
\frac{1}{M_k}\mathbb{E}\bra*{(\Psi_1(\rho(N_{k-1},\rho_{0,1},\cdot))  + \Psi_1(\rho(N_{k-1},\rho_{0,2},\cdot)) )}  \leq \frac{\delta_2}{2} \, .
\end{align}
Putting it all together, we obtain
\begin{align}
&\mathbb{P}\bra*{d(\rho(N_k,\rho_{0,1},\cdot),\rho(N_k,\rho_{0,2},\cdot)) >\delta_1}\\\leq& \frac{\delta_2}{2} + (1-\alpha) \mathbb{P}\bra*{d(\rho(N_{k-1},\rho_{0,1},\cdot),\rho(N_{k-1},\rho_{0,2},\cdot)) >\delta_1}.
\end{align}
Iterating, we obtain
\begin{align}
\mathbb{P}\bra*{d(\rho(N_k,\rho_{0,1},\cdot),\rho(N_k,\rho_{0,2},\cdot)) >\delta_1} \leq ((1-\alpha)^{k-1}+1)\frac{\delta_2}{2} + (1-\alpha)^k \, .
\end{align}
For a fixed $\delta_2>0$, we choose
\[
k= \left\lceil1+ \frac{\ln(\delta_2/3)}{\ln(1-\alpha)}\right\rceil
\]
which yields
\begin{align}
\mathbb{P}\bra*{d(\rho(N_k,\rho_{0,1},\cdot),\rho(N_k,\rho_{0,2},\cdot)) >\delta_1} \leq \frac{\delta_2^2}{6} + \frac{\delta_2}{2} + \frac{\delta_2}{3} \leq \delta_2 \, ,
\end{align}
for $\delta_2 \leq 1$. 

To obtain the rate we fix $k=\sqrt{\ln t}$ fixed and set $\delta_2 =(1-\alpha)^k$. We have the upper bound
\begin{align}
N_k \lesssim \frac{2^{2k}C^k k^2 k! }{(1-\alpha)^{k^2}} \lesssim (1-\alpha)^{-k^2} \lesssim t \, ,
\end{align}
from which the result of the lemma follows.
\end{proof}

\begin{theorem}
The Markov process $\rho(t,\cdot,\cdot)$ on the state space $X$ is uniquely ergodic, i.e. it has a unique invariant probability measure $\mu^* \in \mathcal{P}(X)$. Furthermore, $\rho(t,\cdot,\cdot)$ is strongly mixing, i.e. $P_t^* \mu$ converges weakly to $\mu^*$ for all $\mu \in \mathcal{P}(X)$.

\label{thm:abstractergodicity}
\end{theorem}
\begin{proof}
We start by showing that $P_t(\cdot,\cdot)$ is Feller. Fix $\rho_{0,1},\rho_{0,2} \in X$ and consider some $\varphi \in \Lip(X)$. Then, we have that
\begin{align*}
& \abs*{\int_ X \varphi(\rho) \dx \bra*{P_t(\rho_{0,1},\rho) - P_t(\rho_{0,2},\rho)}} \\
\lesssim  & \norm{\varphi}_{\Lip}\mathbb{E}\pra*{d\bra{\rho(t,\rho_{0,1},\cdot),\rho(t,\rho_{0,2},\cdot)}} \\
\leq & \norm{\varphi}_{\Lip} d\bra{\rho_{0,1},\rho_{0,2}} \, ,
\end{align*}
where $\rho_{0,1},\rho_{0,2} \in X$ and we have used Assumption~\ref{ass:nonexpansive}. It follows that for all $t \geq 0$ the map $X \ni \rho \mapsto \int_X \varphi(\cdot) \dx{P_t(\rho,\cdot)}$ is continuous (indeed Lipschitz) for all $\varphi \in \Lip(X)$. By~\cite[Remark 8.3.1, Volume II]{Bog07}, the same holds true for all $\varphi \in C_b(X)$. Thus, $P_t(\cdot,\cdot)$ is Feller.

For a fixed $\rho_0 \in X$, consider the family of measures $\set{Q_t(\rho_0,\cdot)}_{t >a}$ for some $a>0$, where $Q_t(\cdot,\cdot)$ is defined in Lemma~\ref{lem:compact}. By estimate ~\eqref{eq:tight}, this family is tight and so we can find a sequence $t_n \to \infty$ and a measure $\mu^*$ such that 
\begin{align*}
\lim_{n \to \infty} \int_{X} \varphi(\rho) \dx{Q_{t_n}(\rho_0,\rho)} =  \int_{X} \varphi(\rho) \dx{\mu^*(\rho)} \, ,
\end{align*}
for all $\varphi \in C_b(X)$. Since $P_t(\cdot,\cdot)$ is Feller, we know that $P_s\varphi(\rho):= \int_X \varphi (\cdot)P_{s}(\rho,\cdot) \in C_b(X)$ for all $s\geq 0$. Thus, we have
\begin{align*}
&\lim_{n \to \infty} \frac{1}{t_n}\int_0^{t_n} \int_{X} P_s \varphi(\rho) \dx{P_{r}(\rho_0,\rho)} \dx{r} \\
=&  \lim_{n \to \infty} \frac{1}{t_n}\int_s^{t_n+s} \int_{X}  \varphi(\rho) \dx{P_{r}(\rho_0,\rho)} \dx{r} \\
= &\lim_{n \to \infty}\left( \frac{1}{t_n}\int_0^{t_n} \int_{X}  \varphi(\rho) \dx{P_{r}(\rho_0,\rho)} \dx{r}  \right. \\ &+\left. \frac{1}{t_n}\int_{t_n}^{t_n+s} \int_{X}  \varphi(\rho) \dx{P_{r}(\rho_0,\rho)} \dx{r} - \frac{1}{t_n}\int_0^{s} \int_{X}  \varphi(\rho) \dx{P_{r}(\rho_0,\rho)} \dx{r}\right) \\
=& \int_{X} \varphi(\rho) \dx{\mu^*(\rho)} \, ,
\end{align*}
and $\mu^*$ is invariant for $P_t(\cdot ,\cdot)$. Let $\varphi \in \Lip_1(X)$. Then, we have that
\begin{align}
&\mathbb{E}\bra*{\varphi(\rho(t,\rho_{0,1},\cdot))- \varphi(\rho(t,\rho_{0,2},\cdot))} \\\leq& \mathbb{E}\pra*{d\bra{\rho(t,\rho_{0,1},\cdot),\rho(t,\rho_{0,2},\cdot)}} \\
\leq & \delta + \mathbb{P}\pra*{d\bra{\rho(t,\rho_{0,1},\cdot),\rho(t,\rho_{0,2},\cdot)}>\delta}\label{eq:lipbound}\\
\leq & 2\delta \, , 
\end{align}
if $t > t_2$, where $t_2(\delta,\delta)$ is given by Lemma~\ref{lem:t2}. Assume $\mu_1,\mu_2\in \mathcal{P}(X)$ are both invariant for $P_t$. Then, using the definition Kantorovich--Rubinstein distance (see~\cite[Section 8.3]{Bog07}), we have
\begin{align}
d_{\mathrm{KR}}(\mu,\nu)=&\sup_{\varphi \in \Lip_1(X)} \int_{X}\varphi \dx{(\mu_1-\mu_2)} \\
=&\sup_{\varphi \in \Lip_1(X)}\int_{X}\varphi \dx{P_t^*(\mu_1-\mu_2)} \\
=&\sup_{\varphi \in \Lip_1(X)}\int_{X}P_t\varphi \dx{(\mu_1-\mu_2)} \\
=& \sup_{\varphi \in \Lip_1(X)}\int_{X \times X}\mathbb{E}\bra*{\varphi(\rho(t,\rho_{0,1},\cdot))- \varphi(\rho(t,\rho_{0,2},\cdot))} \dx{\mu_1 \otimes \mu_2}(\rho_{0,1},\rho_{0,2}) \,.
\end{align}
Clearly, by~\eqref{eq:lipbound}, the above quantity can be made arbitrarily small by choosing $t$ to be sufficiently large. It follows that $\mu^* \in \mathcal{P}(X)$ is the unique invariant probability measure of $P_t$.
 Using an identical argument to~\eqref{eq:lipbound} along with the invariance of $\mu^*$, we have that for all $\varphi \in \Lip(X)$ 
\begin{align}
\lim_{t \to \infty}\mathbb{E}\bra*{\varphi(\rho(t,\rho_0,\cdot))} = \int_{X}\varphi(\rho) \dx{\mu^*}(\rho) \, ,
\end{align}
for all $\rho_0 \in X$. Again, by~\cite[Remark 8.3.1, Volume II]{Bog07}, the result  carries over to all $\varphi \in C_b(X)$. One can check that this implies that $P_t^* \mu$ converges weakly to $\mu$ for all $\mu \in \mathcal{P}(X)$. This completes the proof of the lemma.
\end{proof}

\subsection{Application to SPDEs with conservative noise}\label{apptospde}
We will now apply the abstract result of the previous subsection to the class of conservative SPDEs~\eqref{eq:DK}. More precisely, we consider the time-homogeneous Markov process $\rho(t,\rho_0,\cdot)$ associated to stochastic kinetic solutions of~\eqref{eq:DK} and we show that, under certain restrictions on the coefficients of~\eqref{eq:DK}, Assumptions \ref{ass:nonexpansive}, \ref{ass:support}, \ref{ass:Markov}, and \ref{ass:entene} are all satisfied.  

We choose as the state space $X=\Ent_{\Psi_1,\kappa}(\T^d)\cap L^{p}(\T^d)$ for some $\kappa >0$ equipped with the topology of $L^1(\T^d)$, where
\[ 
\Ent_{\Psi_1(\T^d),\kappa}:=\set*{\rho \in L^1(\T^d) : \Psi_1(\rho):=\int_{\T^d}\Phi_1(\rho) \dx{x} <\infty, \int_{\T^d}\rho\dx{x}=\kappa, \rho \geq 0}\, ,
\]
and $p \in [2,\infty)$ is as in Assumption~\ref{ass2} and with $\Phi_1$ as in the following assumption.

\begin{assumption}[Dissipation estimate for SPDE]
We assume that there exist functions $\Phi_1,\Phi_2: \R_+ \to \R_+$ such that any stochastic kinetic solution $\rho(\cdot,\rho_0,\cdot)$ of~\eqref{eq:DK} with initial datum $\rho_0 \in X$  satisfies~\eqref{dissest} with 
\begin{align}
\Psi_1(\rho) := \int_{\T^d} \Phi_1(\rho(x)) \dx{x} \qquad \Psi_2(\rho):=\int_{\T^d}\abs{\nabla \Phi_2(\rho(x))}^2 \dx{x} \,. 
\end{align} 
Furthermore, we assume that $\Phi_1$ is lower semicontinuous and $\Phi_2$ is such that there exists some $r_{\Phi_2}>0$, $m \geq 1$, $C>0$ such that
\begin{align}
\Phi_2(\xi)\leq& C(\xi^m +1)\label{psi1} \\
\Phi_2(\xi) \geq& C\sqrt{\xi} \label{psi2}\, ,
\end{align}
for all $\xi > r_{\Phi_2}$. We also assume that for all $K>0$ there exist $p_K\geq 1, C_K>0$, such that
for all $\xi_1,\xi_2 \leq K $, we have
\begin{align}
\abs{\Phi_2(\xi_1)-\Phi_2(\xi_2)} \geq C_K \abs{\xi_1-\xi_2}^{p_K}\, . \label{psi3}
\end{align}
\label{ass:entene2}
\end{assumption}
We will also make the following assumption.
\begin{assumption}[Contractivity of the deterministic dynamics and support property]
Let $u:[0,T]\times X \to X$ be the solution map of the deterministic dynamics associated to~\eqref{eq:DK}, i.e.~\eqref{eq:DK} with $\sigma\equiv 0$. Then, $u,\rho$ satisfy Assumption~\ref{ass:support}~\ref{ass:supporta} and~\ref{ass:supportb}. 
\label{ass:contractivity}
\end{assumption}
We can finally state a more concrete version of the ergodicity result.
\begin{theorem}
Assume that Assumptions~\ref{ass:entene2} and ~\ref{ass:contractivity} are satisfied along with the assumptions of Theorem~\ref{thm_unique}. If $f\equiv B \equiv 0$, then, the  Markov process $\rho(t,\rho_0,\cdot)$ on $X$ associated to stochastic kinetic solutions of~\eqref{eq:DK} is uniquely ergodic and strongly mixing.
\label{thm:semiabstract}
\end{theorem}
\begin{proof}
Firstly, Assumption~\ref{ass:nonexpansive} follows from the stability estimate of Theorem~\ref{thm_unique} and the fact that $f\equiv B \equiv 0$. Clearly, Assumption~\ref{ass:contractivity} implies Assumption \ref{ass:support}. Furthermore, Assumption~\ref{ass:Markov} follows from Propositions~\ref{prop:markov} and~\ref{prop:strongmarkov}. Finally, by Lemma~\ref{compactnesspsi}, Assumption~\ref{ass:entene2} implies Assumption~\ref{ass:entene}.  This completes the proof.
\end{proof}
We apply the above result to the specific example of a regularized Dean--Kawasaki equation.
\begin{cor}[Regularized Dean--Kawasaki is ergodic]
Consider the following regularized Dean--Kawasaki equation
\begin{align}
\partial_t \rho = \Delta \rho -\sqrt{\ve}\nabla \cdot (\sigma(\rho) \circ\dd \xi) \, ,
\label{eq:regDK}
\end{align}
with initial datum $\rho_0$ where $\sigma \in C_b^{5+\delta}([0,\infty))$, $\delta>0$, bounded with $\sigma(0)=0$, where $\xi$ satisfies Assumption~\ref{assumption_1}. Then,~\eqref{eq:regDK}  has a unique stochastic kinetic solution $\rho(t,\rho_0,\cdot)$, in the sense of Definition \ref{def_sol_new}, for all $\rho_0 \in X= L^2(\T^d)\cap \Ent_{\Psi_1,\kappa}(\T^d)$ with
\[
\Psi_1(\rho) = \int_{\T^d}\Phi_1(\rho) \dx{x}\, , \quad \Phi_1(\xi)= \xi \log \xi \, . 
\]
Furthermore, the associated Markov process on $X$ is uniquely ergodic and strongly mixing.
\end{cor}
\begin{proof}
We note that~\eqref{eq:regDK} is equivalent to the following It\^o equation
\[
\partial_t \rho = \Delta \left(\rho + \frac{\eps F_1}{2} g(\rho) \right) +\sqrt{\eps} \nabla \cdot (\sigma(\rho)  \xi^F) \, ,
\]
where $g:[0,\infty) \to \R$ is the anti-derivative of $(\sigma')^2$ with $g(0)=0$. We will check that this choice of $\Phi(\rho)=\rho+\frac{\eps F_1}{2}g(\rho)$, $\sigma$ satisfies Assumptions~\ref{assumption_2} and \ref{assumption_3}.  For Assumption~\ref{assumption_2}, we note that properties 1, 2, 4, and 5 are trivially satisfied.  Property 3 follows from the differentiability of $\sigma^2$. 

For Assumption~\ref{assumption_3}, we have that properties 1 and 2 are trivially satisfied. For $p=2$, we have
\[
\Theta_{\Phi,2}(\xi) = \bra*{ \int_0^\xi 1 + \frac{\eps F_1}{2}(\sigma'(\xi'))^2 \dx{\xi'}}^{\frac12} \, .
\] 
Since
\[
\Phi'(\xi) \geq 1
\]
for all $\xi \in[0,\infty)$, property 3 is satisfied. Property 5 is also satisfied by a similar argument.  We note that $\Theta_{\Phi,2}$ is invertible and in fact
\begin{align}
(\Theta^{-1}_{\Phi,2})'(\xi) =& \frac{1}{\Theta_{\Phi,2}'(\Theta^{-1}_{\Phi,2}(\xi))} \\
=&\frac{1}{\bra*{1+ \frac{\eps F_1}{2}(\sigma'(\Theta^{-1}_{\Phi,2}(\xi)))^2 }} \leq 1 \, .
\end{align}
This implies property 4 with $p=2$.  Finally, property 5 follows from the fact that $\sigma \in C_b^{5 + \delta}([0,\infty))$ and $1 \leq \Phi'(\xi) \leq C$ for all $\xi \in [0,\infty)$.

 Thus, we have a unique stochastic kinetic solution $\rho(t,\rho_0,\cdot)$ of~\eqref{eq:regDK} for all $\rho_0\in X$. Additionally, by ~\cite[Proposition 5.18]{FG21}, stochastic kinetic solutions of~\eqref{eq:regDK} satisfy Assumption~\ref{ass:entene2} with $\Phi_1$ as specified and $\Phi_2= \rho^{\frac{1}{2}}$. The associated deterministic dynamics $u(t,\rho_0)$ is the heat equation which is exponentially contractive in $L^1(\T^d)$ for initial data with finite entropy by the log-Sobolev inequality on $\T^d$. Additionally, since $\sigma \in C_b^{5+\delta}([0,\infty))$, ~\cite[Equation (2.3)]{FehGes2019} is satisfied. Since by Remark~\ref{rem:twosolutions}, rough path solutions and stochastic kinetic solutions of~\eqref{eq:regDK} are the same, ~\cite[Theorem 1.3]{FehGes2019} tells us that $u$ lies in the support of the path space measure of $\rho$ and so the pair $(u,\rho)$ satisfies Assumption~\ref{ass:contractivity}. Finally, since $f \equiv B \equiv 0$, we can apply Theorem~\ref{thm:semiabstract} to complete the proof.
\end{proof}
\begin{remark}
  In the previous result,  the smoothness assumption $\sigma \in C_b^{5 + \delta}([0,\infty))$ is only required so that we can obtain the support property~\cite[Theorem 1.3]{FehGes2019}, which itself follows from the fact that rough path solutions of~\eqref{eq:regDK} are continuous with respect to the driving noise (equipped with the appropriate rough path metric). In principle, one could expect to find a purely probabilistic proof of the support property for~\eqref{eq:regDK} which does not rely on the theory of rough paths and thus only requires a weaker assumption on the smoothness of $\sigma$. 
\end{remark}

\begin{appendices}

\section{An abstract perfection result}

We present a perfection result due to Kager and Scheutzow~\cite[Theorem 3]{KS97} which we use in Section~\ref{sec:rds}.

\begin{theorem}
Let  $(\mathcal{G},\circ)$ be a second countable, Hausdorff topological semigroup. Let $(\Omega,\mathcal{F}^0, \mathbb{P}, (\theta_t)_{t \in \R})$ be a MDS and assume that $\phi : \Delta \times \Omega \to \mathcal{G}$ satisfies the following properties:
\begin{enumerate}
\item[\textup{(1)}] (Perfect) semiflow property: There exists a $\mathbb{P}$-null set $N$, such that for all $0 \leq s\leq s_1 \leq t \leq T$ and $\omega \in N_s^c$, we have
\[
\phi (t,s,\omega) =\phi(t,s_1,\omega)\circ\phi(s_1,s,\omega) \, .
\]
\item[\textup{(2)}]  Crude cocycle property: for all $s\in [0,T]$ there exists a $\mathbb{P}$-null set $N_{s}$, such that for all $t \in[0, T-s]$, and $\omega \in N_s^c$, we have
\[
\phi(t+s,s,\omega)=\phi(t,0,\theta_s\omega) \, . 
\]
\item[\textup{(3)}] $\phi(t,s,\omega)$ is $(\mathcal{B}(\Delta)\otimes \mathcal{F}^0  ,\mathcal{B}(\mathcal{G}) )$-measurable.
\item[\textup{(4)}]  There exists a $\mathbb{P}$-null set $N$, such that for all $\omega \in N^c$, $\Delta  \ni(t,s) \mapsto \phi(t,s,\omega) \in \mathcal{G}$ is continuous.
\end{enumerate}
Then, there exists a map $\tilde \phi : \Delta \times \Omega \to \mathcal{G}$ which satisfies \textup{(1)}, \textup{(3)}, and \textup{(4)}  along with
\begin{enumerate}
\item[\textup{(2')}] (Perfect) cocycle property: There exists a $\mathbb{P}$-null set $N$, such that for all $s\in  [0,T], t \in[0, T-s]$, and $\omega \in N^c$, we have
\[
\tilde \phi(t+s,s,\omega)=\tilde \phi(t,s,\theta_s\omega) \, . 
\]
\item[\textup{(5)}] $\mathbb{P}(\{\omega \in \Omega: \tilde \phi(t,s,\omega) =\phi(t,s,\omega), \textrm{ for all } (t,s)\in \Delta\})=1$.
\end{enumerate}
\label{thm:ks}
\end{theorem}
\section{Compactness of sublevel sets}
\begin{lemma}
Let $\Phi_2$ be as defined in Assumption~\ref{ass:entene2}. Then, the set
\begin{align}
\cK=\set*{\rho: \norm{\rho}_{\Leb^1(\T^d)} + \int_{\T^d}\abs{\nabla \Psi_2(\rho)}^2 \dx{x} \leq C} \, ,
\end{align}
is compact in $\Leb^1(\T^d)$.
\label{compactnesspsi}
\end{lemma}
\begin{proof}
Fix $\rho \in \cK$. We start by noting that
\begin{align}
\norm{\Psi_2(\rho)}_{\Leb^2(\T^d)}^2 \leq C( \norm{\rho}_{\Leb^1(\T^d)}^K + \norm{\nabla \Psi_2(\rho)}_{\Leb^2(\T^d)}^2) \, ,
\label{bbound}
\end{align}
for some constant $C>0$ and $K$ large enough.  Indeed, for some $p,q$ such that $1/p +1/q=1$, we have that
\begin{align}
\int_{\T^d} \Psi_2(\rho)^2 \dx{x} =& \int_{\T^d}\rho^{1/p}\rho^{-1/p} \Psi_2(\rho)^2 \dx{x} \\
=& \norm{\rho^{1/p}}_{\Leb^p(\T^d)} \norm{\rho^{-1/p}\Psi_2(\rho)^2}_{\Leb^q(\T^d)} \\
=&\norm{\rho}_{\Leb^1(\T^d)}^{1/p} \norm{\rho^{-q/p}\Psi_2(\rho)^{2(q-2^*/2)} \Psi_2(\rho)^{2^*} }_{\Leb^1(\T^d)} \\
=&\norm{\rho}_{\Leb^1(\T^d)}^{1/p} \norm{\rho^{1-q}\Psi_2(\rho)^{2(q-2^*/2)} \Psi_2(\rho)^{2^*} }_{\Leb^1(\T^d)} \, .
\end{align}
Since, by~\eqref{psi1}, $\Psi_2$ has polynomial growth (for large enough $\rho$) we can choose $q$ close to, but greater than, $2^*/2$, to argue that $\rho^{1-q}\Psi_2(\rho)^{2(q-2^*/2)} \leq C$. It follows that
\begin{align}
\norm{\Psi_2(\rho)}_{\Leb^2(\T^d)}^2 \leq& C  \norm{\rho}_{\Leb^1(\T^d)}^{1/p} \norm{ \Psi_2(\rho) }_{\Leb^{2^*}(\T^d)}^{2^*/q} \\
\leq & C_1 \norm{\rho}_{\Leb^1(\T^d)}^{K} + \eps \norm{ \Psi_2(\rho) }_{\Leb^{2^*}(\T^d)}^{2} \, ,
\end{align}
where in the last step we have applied Young's inequality since $2^*/q<2$. Apply Gagliardo--Nirenberg, we obtain
\begin{align}
\norm{\Psi_2(\rho)}_{\Leb^2(\T^d)}^2 \leq &C_1 \norm{\rho}_{\Leb^1(\T^d)}^{K} + \eps \norm{ \nabla \Psi_2(\rho) }_{\Leb^{2}(\T^d)}^{2}  + \eps \left(\int_{\T^d}\Psi_2(\rho )\dx{x}\right)^2 \, .
\end{align}
Applying Jensen's inequality,~\eqref{bbound} follows.

It follows then, by the Gagliardo--Nirenberg inequality, that
\begin{align}
\norm{\Psi_2(\rho)}_{2^*}^2 \leq& C \bra*{\norm{\nabla \Psi_2(\rho)}_{\Leb^2(\T^d)}^2 + \bra*{\int_{\T^d} \Psi_2(\rho) \dx{x}}^2  } \\
\leq & C \bra*{\norm{\Psi_2(\rho)}_{\Leb^2(\T^d)}^2 + \norm{\nabla \Psi_2(\rho)}_{\Leb^2(\T^d)}^2} \\
\leq & C  \bra*{\norm{\nabla \Psi_2(\rho)}_{\Leb^2(\T^d)}^2 + \norm{\rho}_{\Leb^1(\T^d)}^K} \\
\leq&  C \, ,
\end{align}
for a constant $C>0$ independent of $\rho \in \cK$. We now use~\eqref{psi2} to assert that, for $\eta=2^*/2$,
\begin{align}
\norm{\rho}_{\Leb^{\eta}(\T^d)} = &\norm{\rho \mathbf{1}_{\rho \leq r_{\Psi_2}}}_{\Leb^{\eta}(\T^d)} + \norm{\rho \mathbf{1}_{\rho > r_{\Psi_2}}}_{\Leb^{\eta}(\T^d)}  \\
\leq & r_{\Psi_2} + \norm{\Psi_2(\rho)}_{\Leb^{2^*}(\T^d)}^2 \leq C \, ,
\end{align}
for a constant $C>0$ independent of $\rho \in \cK$. We now note that for every $h>0$
\begin{align}
\int_{\T^d  }\mathbf{1}_{\rho \geq K} \abs{\rho(x+h)-\rho(x)} \dx{x} \leq & \norm{\rho(x+h)-\rho(x)}_{\Leb^\eta(\T^d)} \abs{\set{\rho\geq K}}^{1/\eta'} \\
\leq & \norm{\rho(x+h)-\rho(x)}_{\Leb^\eta(\T^d)} \abs{\set{\rho^\eta\geq K^\eta}}^{1/\eta'} \\
\leq & 2 \norm{\rho}_{\Leb^\eta(\T^d)}^{1+\eta/\eta'} K^{\eta/\eta'} \, ,
\end{align}
where $\eta'=\eta/(\eta-1)$. We now use~\eqref{psi3} to note that
\begin{align}
&\int_{\T^d  } \abs{\rho(x+h)-\rho(x)} \dx{x} \\ \leq & \int_{\T^d  }\mathbf{1}_{\max{\rho,\rho(\cdot +h)} \leq K} \abs{\rho(x+h)-\rho(x)} \dx{x} \\
&+ \int_{\T^d  }\mathbf{1}_{\rho \geq K} \abs{\rho(x+h)-\rho(x)} \dx{x} + \int_{\T^d  }\mathbf{1}_{\rho(\cdot + h) \geq K} \abs{\rho(x+h)-\rho(x)} \dx{x} \\
\leq &  \bra*{\int_{\T^d  }\mathbf{1}_{\max{\rho,\rho(\cdot +h)} \leq K} \abs{\rho(x+h)-\rho(x)}^{p_K} \dx{x}}^{1/p_K} + C K^{-\eta/\eta'} \\
\leq & \bra*{C_K \int_{\T^d  } \abs{\Psi_2(\rho(x+h))-\Psi_2(\rho(x))} \dx{x}}^{1/p_K}   + C K^{-\eta/\eta'} \\
\leq &  \bra*{C_K\abs{h}\int_{\T^d  }\abs{\nabla \Psi_2(\rho)} \dx{x}}^{1/p_K}   + C K^{-\eta/\eta'} \\
\leq & C(\abs{h}^{1/p_K} + K^{-\eta/\eta'}) \, ,
\end{align}
where the constant $C>0$ is again independent of $\rho \in \cK$. Now for any $\eps>0$, we can first choose $K$ large enough and $\abs{h}$ small enough such that
\begin{align}
\int_{\T^d  } \abs{\rho(x+h)-\rho(x)} \dx{x} \leq \eps \, ,
\end{align} 
for all $\rho \in \cK$. By the Fr\'echet--Kolmogorov theorem, it follows that $\cK$ is compact in $\Leb^1(\T^d)$.
\end{proof}

\section{The strong Markov property}\label{app:markov} 
This section is dedicated to proving the strong Markov property for solutions of~\eqref{eq:DK}. The proof uses  similar ideas as the proofs of~\cite[Theorems 9.14 and 9.20]{DPZ14}. We present it in our setting for the convenience of the reader. We start by proving the Markov property for stochastic kinetic solutions of~\eqref{eq:DK}.

\begin{prop}[Markov property]
Let $\bar \rho(t,s,\rho_0,\cdot)$ denote the unique stochastic kinetic solution of~\eqref{eq:DK} started at time $s\geq 0$ with initial datum $\rho_0 \in L^p(\Omega; L^p(\T^d))$ and under Assumptions~\ref{assumption_1},~\ref{ass1},~\ref{ass2}, and~\ref{assdrift}. Then, $\bar\rho(\cdot,s,\rho_0,\cdot)$ is a Markov process, i.e. given any $F \in B_b(L^1(\T^d))$, we have almost surely
\begin{equation}
\mathbb{E}\bra*{F(\bar \rho(t,u,\rho_0,\cdot))|\mathcal{F}_{s,u}} (\omega) = \mathbb{E}(\rho(t,s,\bar\rho(s,u,\rho_0,\omega) ,\cdot)) \, \, ,
\label{eq:markov}
\end{equation}
for all $0 \leq u \leq s \leq t$. The same property is satisfied for the two-point process $(\bar \rho(t,s,\rho_{0,1},\cdot),\bar \rho(t,s,\rho_{0,2},\cdot))$.
\label{prop:markov}
\end{prop}
\begin{proof}
We start the proof by noting that using essentially the same argument as in Section 2 the following identity holds almost surely, 
\begin{equation}
\bar\rho(t,u,\rho_0,\omega)=\bar\rho(t,s, \bar \rho(s,u,\rho_0,\cdot),\omega) \, ,
\end{equation}
for all $0 \leq u \leq s \leq t$. We can thus rewrite the Markov property~\eqref{eq:markov} as 
\begin{equation}
\mathbb{E}\bra*{F(\bar\rho(t,s, \bar \rho(s,u,\rho_0,\cdot),\cdot))|\mathcal{F}_{s,u}}(\omega) = \mathbb{E}(\bar{\rho}(t,s,\bar\rho(s,u,\rho_0,\omega) ,\cdot)) \, .
\end{equation} 
Thus, the Markov property will hold true if we replace $\bar \rho(s,u,\rho_0,\cdot)$ by an arbitrary $L^p$-valued $\mathcal{F}_{s,u}$-measurable random variable $\eta$ and establish the identity
\begin{equation}
\mathbb{E}\bra*{F(\bar\rho(t,s, \eta,\cdot))|\mathcal{F}_{s,u}}(\omega) = \mathbb{E}(\bar{\rho}(t,s,\eta(\omega) ,\cdot)) \, .
\label{eq:identity}
\end{equation}
As before, we start with $\eta$ that only take on a finite number of values and assume that $F \in C_b(L^1(\T^d))$.  Let $A_1,\dots, A_N \subset \mathcal{F}_{s,u}$ be a partition of $\Omega$ and let $g_1, \dots g_N \in L^p(\T^d)$. As in Section 2, we can check that almost surely
\begin{equation}
\bar\rho(t,s,\eta_N,\omega) =\sum_{i=1}^N \mathbf{1}_{A_i}\rho(t,s,g_i,\omega) \, ,
\end{equation}
where $\eta_N=\sum_{i=1}^N\mathbf{1}_{A_i}g_i$. It follows that
\begin{align}
\mathbb{E}\bra*{F(\bar\rho(t,s, \eta_N,\cdot))|\mathcal{F}_{s,u}}(\omega) =& \sum_{i=1}^N\mathbb{E}\bra*{ \mathbf{1}_{A_i}F(\bar\rho(t,s, g_i,\cdot))|\mathcal{F}_{s,u}}(\omega) \\
=& \sum_{i=1}^N \mathbf{1}_{A_i}(\omega)\mathbb{E} \bra*{F(\bar\rho(t,s, g_i,\cdot))|\mathcal{F}_{s,u}}\, .
\end{align}
We note now that $\rho(t,s,g_i,\cdot)$ is $\mathcal{F}^t_s$-measurable and so is independent of $\mathcal{F}_{s,u}$. Since $F\in C_b(\T^d)$, we have that
\begin{align}
\mathbb{E}\bra*{F(\bar\rho(t,s, \eta_N,\cdot))|\mathcal{F}_{s,u}}(\omega)= &  \sum_{i=1}^N \mathbf{1}_{A_i}(\omega)\mathbb{E} \bra*{F(\bar\rho(t,s, g_i,\cdot))} \\
=&\mathbb{E} \bra*{F(\bar\rho(t,s, \eta_N(\omega),\cdot))} \, .
\end{align}
Now given any $\mathcal{F}_{s,u}$ measurable $L^p(\T^d)$-valued random variable $\eta$ such that
\begin{equation}
\mathbb{E}\bra*{\norm{\eta}_{L^p(\T^d)}^p} < \infty \, ,
\end{equation}
 we can find a sequence of such simple $\eta_N$'s such that 
 \begin{equation}
\lim_{N\to \infty} \mathbb{E}(\norm{\eta-\eta_N}_{L^p(\T^d)}^p) =0 \, ,
 \end{equation}
and $\eta_N \to \eta$ in $L^p(\T^d)$ almost surely. Furthermore, by the stability estimate~\eqref{eq:contraction}, we have that $\rho(t,s,\eta_N,\cdot) \to \rho(t,s,\cdot)$  in $L^1(\T^d)$ for all $t\geq s$ almost surely. We can thus pass to the limit in~\eqref{eq:identity} to obtain that for all  $\mathcal{F}_{s,u}$ measurable $\eta \in L^p(\Omega;L^p(\T^d))$ and $F \in C_b(L^1(\T^d))$, we have
\begin{equation}
\mathbb{E}\bra*{F(\bar\rho(t,s, \eta,\cdot))|\mathcal{F}_{s,u}}(\omega) = \mathbb{E}(\bar{\rho}(t,s,\eta(\omega) ,\cdot)) \, ,
\end{equation} 
almost surely.  By an approximation argument, we can establish the above result for any bounded and measurable $F:L^1(\T^d) \to \R $. Since we know that $\rho(s,u,\rho_0,\cdot) \in L^p(\Omega;L^p(\T^d))$, this completes the proof. The proof for the two-point process is identical.
\end{proof}
We can now move on to the proof of the strong Markov property.
\begin{prop}[Strong Markov property]\label{prop:strongmarkov}
Consider the setting of Proposition~\ref{prop:markov}. Fix $u\geq 0$ and let $\tau$ be a $\mathcal{F}^t_u$ stopping time such that $\tau \geq u$ almost surely. Then, for all $F \in B_b(L^1(\T^d))$, we have  almost surely on $\{\omega \in \Omega:\tau <\infty\}$,
\begin{equation}
\mathbb{E}(F(\bar\rho (t+\tau,u,\rho_0,\cdot ))| \cF_\tau) (\omega) = \mathbb{E}(F(\rho(t+\tau(\omega),\tau(\omega) ,\bar \rho(\tau,u,\rho_0,\omega),\cdot))) \, ,
\label{eq:strongmarkov}
\end{equation}
for all $t\geq 0$ where $\cF_\tau$ consists of all events $A \in \cF$ such that $\set{\tau \leq t} \cap A \in \cF^t_u$ for all $t \geq u$. The same property is satisfied for the two-point process $(\bar \rho(t,s,\rho_{0,1}), \bar \rho(t,s,\rho_{0,2}))$.
\end{prop}
\begin{proof}
It is relatively straightforward to show that~\eqref{eq:strongmarkov} follows from~\eqref{eq:markov} for any $\mathcal{F}^t_u$ stopping time that takes only finitely many values. For an arbitrary stopping $\tau$, we define $\tau_n=2^{-n}(\lfloor2^n\tau\rfloor +1)$. Clearly, $\tau_n\to \tau$ as $n \to \infty$ and $\tau_n \geq \tau$. As a consequence, we have $\mathcal{F}_{\tau} \subset \mathcal F_{\tau_n}$ and $\{\omega \in \Omega:\tau <\infty\} \subset \{\omega \in \Omega:\tau_n <\infty\}$ for all $n \in \N$. It follows that on $\{\omega \in \Omega:\tau <\infty\}$, we have almost surely
\begin{align}
 &\mathbb{E}(F(\bar\rho (t+\tau,u,\rho_0,\cdot ))| \cF_\tau) (\omega) \\= & \mathbb{E}(F(\bar\rho (t+\tau_n,u,\rho_0,\cdot ))| \cF_\tau) (\omega) \\&+ \mathbb{E}(F(\bar\rho (t+\tau,u,\rho_0,\cdot ))- F(\bar\rho (t+\tau_n,u,\rho_0,\cdot ))| \cF_\tau) (\omega) \\
 =& \mathbb{E}\bra*{\mathbb{E}(F(\bar\rho (t+\tau_n,u,\rho_0,\cdot ))| \cF_{\tau_n})|\mathcal{F}_\tau} (\omega) \\&+ \mathbb{E}(F(\bar\rho (t+\tau,u,\rho_0,\cdot ))- F(\bar\rho (t+\tau_n,u,\rho_0,\cdot ))| \cF_\tau) (\omega) \\
 =&  \mathbb{E}\bra*{ \mathbb{E}(F(\rho(t+\tau_n(\omega),\tau_n(\omega) ,\bar \rho(\tau_n,u,\rho_0,\omega),\cdot)))|\mathcal{F}_\tau} (\omega)  \\
  &+ \mathbb{E}(F(\bar\rho (t+\tau,u,\rho_0,\cdot ))- F(\bar\rho (t+\tau_n,u,\rho_0,\cdot ))| \cF_\tau) (\omega) \, .
\end{align}
Note that since $\tau_n$ converges pointwise to $\tau$ and, by Definition~\ref{def_sol_new} , $\bar \rho(\cdot,s,\rho_0,\omega)$ is almost surely continuous as a map from $[s,T]$ to $L^1(\T^d)$, we have that  $F(\bar\rho (t+\tau_n,u,\rho_0,\cdot ))$ converges almost surely to $F(\bar\rho (t+\tau,u,\rho_0,\cdot ))$ as $n \to \infty$. By the dominated convergence theorem, the second term in the above expression tends to $0$ almost surely. For the first term, we use~\eqref{eq:contraction}  to observe that
\[\rho(t+\tau_n(\omega),\tau_n(\omega) ,\bar \rho(\tau_n,u,\rho_0,\omega),\cdot),\]
converges almost surely to $\rho(t+\tau(\omega),\tau(\omega) ,\bar \rho(\tau,u,\rho_0,\omega),\cdot))$ as $n \to \infty$. Applying the dominated convergence theorem again, the result of the proposition follows. As before, the proof for $F \in B_b(L^1(\T^d))$ follows by an approximation argument. The proof for the two-point process is identical.
\end{proof}

\section*{Acknowledgments}

BF acknowledges financial support from the EPSRC through the EPSRC Early Career Fellowship EP/V027824/1.  BG acknowledges support by the Max Planck Society through the Research Group "Stochastic Analysis in the Sciences (SAiS)". This work was funded by the Deutsche Forschungsgemeinschaft (DFG, German Research Foundation) - SFB 1283/2 2021 - 317210226.

\end{appendices}
\bibliographystyle{myalpha}
\bibliography{rds}

\end{document}